\documentclass[10pt,letterpaper]{amsart}

\usepackage[left=2cm,right=2cm,top=2cm,bottom=2cm]{geometry}
\usepackage[latin1]{inputenc}
\usepackage{amsmath,amsfonts,amssymb,amsthm,graphicx,color}
\usepackage{caption}  
\usepackage{subcaption}
\usepackage[labelfont=normalfont,textfont=normalfont,singlelinecheck=off,justification=centering]{subcaption}
\usepackage[colorlinks=true, pdfstartview=FitV, linkcolor=blue, citecolor=blue, urlcolor=blue]{hyperref}
\usepackage[normalem]{ulem}
\usepackage{enumitem}

 \theoremstyle{plain}

\newtheorem*{theorem*}{Theorem}
\newtheorem*{claim*}{Claim}
\newtheorem*{lemma*}{Lemma}

\newtheorem{thm}{Theorem}[section]
\newtheorem{theorem}[thm]{Theorem}
\newtheorem{lemma}[thm]{Lemma}
\newtheorem{cor}[thm]{Corollary}
\newtheorem{corollary}[thm]{Corollary}
\newtheorem{prop}[thm]{Proposition}
\newtheorem{proposition}[thm]{Proposition}

\theoremstyle{definition}

\newtheorem{definition}[thm]{Definition}

\newtheorem{remark}[thm]{Remark}

\newtheorem{Open questions}[thm]{Open questions}
\newtheorem{Open question}[thm]{Open question}
\newtheorem{Open problems}[thm]{Open problems}
\newtheorem{Open problem}[thm]{Open problem}

\definecolor{magenta}{rgb}{.5,0,.5} 
\definecolor{dred}{rgb}{.5,0,0} 
\definecolor{green}{rgb}{0,.5,0} 
\definecolor{blue}{rgb}{0,0,1} 
 \definecolor{black}{rgb}{0,0,0} 
\definecolor{dgreen}{rgb}{0,.3,0} 
\definecolor{vdred}{rgb}{.3,0,0} 
\definecolor{red}{rgb}{1,0,0} 
\definecolor{gray}{rgb}{.5,.5,.5} 
\definecolor{cerulean}{rgb}{0,.48,.65} 
\definecolor{gold}{rgb}{0.80,0.58,0.05} 
\definecolor{orange}{rgb}{1.00,0.50,0} 
\definecolor{pink}{rgb}{1.00,0.08,0.58}

\def \N{\mathbb{N}}
\def \Z{\mathbb{Z}}
\def \R{\mathbb{R}}
\def \C{\mathbb{C}}

\def \SL{\textup{SL}}
\def \GL{\textup{GL}}
\def \Area{\textup{Area}}
\def \Out{\textup{Out}}
\def \Aut{\textup{Aut}}
\def \Inn{\textup{Inn}}

\def \CAT(0){\textup{CAT(0)}}
\def \Diam{\textup{Diam}}
\def \link{\textup{link}}
\def \star{\textup{star}}
\def\isom{\cong}
\def\onto{{\kern3pt\to\kern-8pt\to\kern3pt}}
 
 \newcommand{\set}[1]{\left\{#1\right\}}
\newcommand{\restricted}[1]{\left|_{#1} \right.}
\newcommand{\abs}[1]{\left|#1\right|}
\renewcommand{\ni}{\noindent}
\renewcommand{\ss}{\smallskip}

\newcommand*{\vcenteredhbox}[1]{\begingroup
	\setbox0=\hbox{#1}\parbox{\wd0}{\box0}\endgroup}

\def\*{^{\star}}

\begin{document}
 
 \title{Dehn functions of mapping tori of right-angled Artin groups}
\author{Kristen Pueschel and Timothy  Riley}
\date \today

\begin{abstract}
	The algebraic mapping torus $M_{\Phi}$ of a group $G$ with an automorphism $\Phi$ is the HNN-extension of $G$ in which conjugation by the stable letter performs $\Phi$.    
	We classify the Dehn functions of $M_{\Phi}$ in terms of $\Phi$ for a number of   right-angled Artin groups $G$, including all $3$-generator  right-angled Artin groups and $F_k \times F_l$ for all $k,l \geq 2$.  
	\ss
	\\
	\footnotesize{\ni  2010 Mathematics Subject
		Classification:  20F65, 20F36   \\ \ni \emph{Key words and phrases:} Dehn function, right-angled Artin group, mapping torus}
\end{abstract}

\maketitle

\section{Our results} \label{our results}

When studying mapping tori, a natural question is how  the maps used to define them determine their geometry. The paradigm is the Nielsen--Thurston classification of topological mapping tori.  When $S$ is a compact orientable surface of genus at least $2$ and $f:S \to S$ is a homeomorphism, the mapping torus is $(S \times [0,1])/ \!\! \sim$, with $\sim$ defined so that $ (x, 1) \sim (f(x),0)$ for all $x \in S$. The classification is that, up to isotopy, $f$ is of one of three types:  (1) reducible, in which case the mapping torus  contains an incompressible torus,  (2) periodic,  in which case the mapping torus admits an $\mathbb{H}^2 \times \mathbb{R}$ structure, and (3) psuedo-Anosov, in which case the mapping torus admits a hyperbolic structure. 

Here, we study algebraic mapping tori of right-angled Artin groups.  For a finitely presented group $G= \langle X  \! \mid \!   R \rangle$ and an injective endomorphism $\Phi: G \to G$, the \textit{algebraic mapping torus} is the group $$M_{\Phi} \ := \  \langle X, t \mid   {R}, \  t^{-1}xt= \Phi(x),~ \forall x \in X \rangle.$$ 
In this article,  $\Phi$ will always be an automorphism, so  $M_{\Phi} = G \rtimes_{\Phi}  \langle t \rangle$, and  $G$ will always be  a \emph{right-angled Artin group} (`RAAG')---that is, $G$ is encoded by a finite graph $\Gamma$ with vertex set $X$ in that $G$ is presented by  $$\langle X \mid uv=vu \text{ when } (u, v) \text{ is an edge in } \Gamma  \rangle.$$Algebraic mapping tori arise naturally as fundamental groups of topological mapping tori of surfaces or complexes. 

The lens through which we will study  the geometry of $M_{\Phi}$ will be the \emph{Dehn function} (which we will always consider qualitatively---that is, up to an equivalence relation $\simeq$:  for $f,g: \mathbb{N} \to \mathbb{N}$, write $f\preceq g$ when there exists $C>0$ such that $f(n) \leq Cg(Cn+C) +Cn+C$ for all $n \in \mathbb{N}$, and write  $f \simeq g$ when $f \preceq g$ and $g \preceq f$). The Dehn function is an invariant of finitely presentable groups which can be framed either as an algorithmic complexity measure for the word problem or as an isoperimetric function  recording the  minimal area of discs spanning loops as a function of the lengths of the loops. (More details are in Section~\ref{sec: vK diagrams,corridors, DF}.)

Our study of Dehn functions for mapping tori of RAAGs is motivated by the following two classifications.    
The first concerns $G=\mathbb{Z}^k$, the RAAG associated to the complete graph with $k$ vertices.

\begin{theorem*}[Bridson--Gersten \cite{BridsonGersten}, Bridson--Pittet \cite{BridsonPittet}]  Suppose $\Phi \in \Aut(\Z^k) = \GL(k,\Z)$. If $\Phi$ has an eigenvalue $\lambda$ with $|\lambda| \neq 1$, then the Dehn function of the mapping torus  $M_{\Phi}$ is exponential. Else, the Dehn function of $M_{\Phi}$ is polynomial of degree $c+1$, where $c \times c$ is the size of the largest Jordan block in the Jordan Canonical form of the matrix associated to $\Phi$. 
\end{theorem*}  

The second concerns $G=F_k$, the rank-$k$ free group, i.e.\ the RAAG associated to  the graph with $k$ vertices and no edges. An automorphism $\Phi$ of $F_k$ is \emph{atoroidal} when there are no periodic conjugacy classes---that is, for all $w \in F_k$ and   $n \in \Z$, if  $w$ and $\Phi^n(w)$ are conjugate, then $w=1$ or $n=0$. 

\begin{theorem*}[Bestvina--Handel \cite{Bestvina-Handel}, Brinkmann \cite{Brinkmann},  Bridson--Groves \cite{BridsonGroves}] Suppose $\Phi \in \Aut(F_k)$. The mapping torus $M_{\Phi}$ is hyperbolic (that is, has linear Dehn function) if and only if $\Phi$ is atoroidal.  All other $M_{\Phi}$ have quadratic Dehn functions.
\end{theorem*}

RAAGs can be viewed as   interpolating between free abelian   and free groups, so it is natural to look to extend the above theorems to other RAAGs.  We  thank Karen Vogtmann for suggesting this problem. 

A classification of the Dehn functions of all RAAGs remains out of reach. Here we complete the classification for three-generator RAAGs and all groups $F_k \times F_l$ where $k,l \geq 2$.

For $\mathbb{Z}$, $\mathbb{Z}^2$, $\mathbb{Z}^3$, $F_2$, and $F_3$, the theorems above classify the Dehn functions of $M_{\Phi}$.  The remaining  three-generator RAAGs are $F_2 \times\mathbb{Z}$ and $\mathbb{Z}^2 \ast \mathbb{Z}$. Here are our results.

\begin{theorem} \label{F2xZ thm} Suppose $\Psi \in \Aut(F_2 \times \mathbb{Z})$.  Let $\psi \in \Aut(F_2)$  be the map induced by $\Psi$ via the map $F_2 \times \mathbb{Z} \onto F_2$ killing the $\Z$ factor.  Let   $\psi_{ab} \in \Aut(\Z^2)$ be the map induced via  the abelianization map $F_2 \to \Z^2$,   $g \mapsto g_{ab}$.
	
	Let $p: F_2 \times \mathbb{Z} \to  \mathbb{Z}$ be projection to the second factor. Exactly one of the following holds:
	\begin{enumerate}
		\item There exists $g \in F_2$ and $m \in \N$  such that $\psi_{ab}^m(g_{ab}) = g_{ab}$ and $p(\Psi^m(g)) \neq 0$, in which case $M_{\Psi}$ has cubic Dehn function. \label{thm1,cubic}
		\item $M_{\Psi}$ has quadratic Dehn function. \label{thm1,quadratic}
	\end{enumerate}
\end{theorem}

\begin{theorem} \label{Z2astZ} Suppose $\Psi \in \Aut(\mathbb{Z}^2 \ast \mathbb{Z})$. Suppose $\Phi \in  \Aut(\mathbb{Z}^2 \ast \mathbb{Z})$    restricts to an automorphism on the $\mathbb{Z}^2$ factor and satisfies $[\Psi] = [\Phi] \in \Out(\mathbb{Z}^2 \ast \mathbb{Z})$.  Let $\phi$ be the restriction of  $\Phi$ to the $\mathbb{Z}^2$-factor. Exactly one of the following holds:
	\begin{enumerate}
		\item $\phi$ is finite order, in which case $M_{\Psi}$ has quadratic Dehn function. \label{two1}
		\item $\phi$ has an eigenvalue $\lambda$ such that $|\lambda| \neq 1$, in which case $M_{\Psi}$ has exponential Dehn function. \label{two2}
		\item  $M_{\Psi}$ has cubic Dehn function. \label{Case 3}  \label{two3}
	\end{enumerate}
\end{theorem}

Theorem~\ref{Z2astZ} is effective in that given a $\Psi$, a $\Phi$ as per the statement is easy to produce: see Lemma~\ref{get Phi}.

Suppose $F$ is a  free group with a finite  basis  $X$.  For $x\in F$, $|x|$ denotes the length of the reduced word on $X^{\pm 1}$ representing $x$. The \emph{growth} $g_{\Phi, X}: \N \to \N$ of an automorphism $\Phi: F  \to F$  is defined by  $g_{\Phi, X}(n) := \max_{x\in X}\{|\Phi^n(x)|\}$. While the growth type of $g_{\Phi, X}$ does not depend on the choice of   $X$, it is not invariant under inner automorphisms. For example, the automorphism  $\phi: a \mapsto b^{-1}ab,~~ b \mapsto b$ has linear growth, whereas $\psi: a \mapsto a,~~ b\mapsto b$ has constant growth.  The \emph{cyclic growth} $g_{\Phi}^{cyc}$ of an automorphism accounts for this issue; it describes the growth of (all) conjugacy classes under iteration of automorphisms, and  is invariant under inner automorphisms. (Details are in Section~\ref{subsec:growth}.)

We classify the Dehn functions of mapping tori of products $F_k \times F_l$ of free groups with $k,l \geq 2$ as follows.

\begin{theorem} \label{F_k x F_l} If $G= F_k \times F_l,$ where $k,l \geq 2$, and $\Psi \in \Aut(F_k \times F_l)$, then we can find  
	$\phi_1 \in \Aut(F_k)$ and   $\phi_2 \in \Aut(F_l)$ such that
	$\Phi = \phi_1 \times \phi_2$ satisfies  $[\Phi] = [\Psi^2]$  in $\Out(F_k \times F_l)$.  The  Dehn functions  of the associated mapping tori satisfy $\delta_{M_{\Phi}} \simeq \delta_{M_{\Psi}}$ and their asymptotics can be read off $\phi_1$ and $\phi_2$ in that:  
	\begin{enumerate}
		\item \label{dos}  If $[\phi_i^p] = [\textup{Id}] \in \Out(F_k)$ for some $p\in \mathbb{N}$, and $i$ either 1 or 2, then  $\delta_{M_{\Psi}}(n) \simeq n^2$.
		\item\label{uno} If $n^{d_1} \simeq g_{\phi_1}^{cyc}(n) \preceq g_{\phi_2}^{cyc}(n)$ for some $d_1 \geq 1$, then  $\delta_{M_{\Psi}}(n) \simeq n^{d_1+2}$, and likewise with the indices $1$ and $2$ interchanged.		
		\item\label{quatro} If   $g_{\phi_1}^{cyc}(n) \simeq g_{\phi_2}^{cyc}(n) \simeq 2^n$, then $\delta_{M_{\Psi}}$  grows exponentially.
	\end{enumerate}
\end{theorem}

As we will explain in Section~\ref{subsec:growth}, the three cases in this theorem are exhaustive and mutually exclusive.

Since all automorphisms of $F_2$ are periodic or have cyclic growth that is linear or exponential, this implies:
\begin{corollary} \label{F_2 x F_2}
	If $G= F_2 \times F_2,$ and $\Psi \in \Aut(G)$, then $M_{\Psi}$ has quadratic, cubic, or exponential Dehn function. 
\end{corollary}

The case $G = F_k \times \mathbb{Z}$ when $k \geq 3$, stands in the way of a full classification of Dehn functions of mapping tori over $F_k \times F_l$.  It differs from  $F_k \times F_l$ with $k,l\geq 2$ because $F_k\times \mathbb{Z}$ has  non-trivial center, which results in additional transvections. What we can say about the   Dehn functions $\delta$ of mapping tori of $F_k \times \mathbb{Z}$ is that they satisfy $n^2 \preceq \delta(n) \preceq n^3$. The cubic upper bound comes by recognizing $M_{\Psi}$ as a central extension of $M_{\psi}$ and then applying Corollary~\ref{cor: electrostatic upperbounds}.  The quadratic lower bound comes from the presence of a $\Z^2$-subgroup: the square of the stable letter commutes with the $\Z$-factor. In  special cases we can determine the Dehn function.

For all   $\Psi \in \Aut(F_k \times \Z)$, there exists $\Phi \in \Aut(F_k \times \Z)$ with the form $\Phi : x_i \mapsto \phi(x_i)c^{k_i},  \ c \mapsto c,$ such that $[\Psi^2] = [\Phi]$ in $\Out(F_k \times \Z)$. 
\begin{enumerate}
	\item If $\phi$ is atoroidal, then $M_{\Psi}$ has quadratic Dehn function by Theorem \ref{cor: electrostatic upperbounds}, because the base of the central extension is hyperbolic and maximal trees have linear diameter.
	\item If there is $w \in F_k$ such that $\Phi(w)=wc^k$, then $M_{\Psi}$ has cubic Dehn function by Theorem \ref{BridsonGersten}. 
\end{enumerate}

Our techniques for  $F_2 \times \mathbb{Z}$ do not apply to $F_k \times \mathbb{Z}$  for $k \geq 3$. We heavily use the isomorphism   ${\Out(F_2) \cong \textup{GL}(2, \mathbb{Z})}$ and the fact that  for any given $\phi \in \Aut(F_2)$ some iterate $[\phi]^m$ fixes the conjugacy class $\left[a^{-1}b^{-1}ab\right]$.  These fail in higher rank. 

Another example to investigate next is the RAAG whose graph is the path with four vertices and three edges.

\section{Overview}

This article is organized as follows. In Section~\ref{Preliminaries}, we give    background on Dehn functions and on  corridors in van~Kampen diagrams.  In Section~\ref{sec:electrostatic model} we review the \textit{electrostatic model} of Gersten and Riley from \cite{GerstenRiley}. We prove Theorems~\ref{F2xZ thm}, \ref{Z2astZ}, and  \ref{F_k x F_l} in Sections~\ref{F2xZ}, \ref{Z2*Z} and \ref{HigherRank}, respectively.   

Here is an overview of our strategy.  Given a RAAG $G$, we organize its automorphisms  $\Phi$  into cases, chosen so that within each case we can  present $M_{\Phi}$ in a manner which facilitates analysis of its Dehn function.  In some cases we find it convenient to replace $\Phi$ by a power; this, in turn, replaces  $M_{\Phi}$ by a finite index subgroup, which does not qualitatively change the Dehn function. 

In the setting of Theorem~\ref{F2xZ thm}, our presentation expresses $M_{\Phi}$ as a central extension of another mapping torus  $M_{\phi}$.  Then we use what Gersten and Riley called an electrostatic model in \cite{GerstenRiley} to get upper bounds on the the Dehn function of $M_{\Phi}$.  The idea is that a van~Kampen diagram over $M_{\phi}$ can be `charged' by  elements of the kernel of the extension (elements of the center of $M_{\Phi}$). The diagram is then `inflated' by adding in suitable corridors to connect up these charges and get a van Kampen diagram over $M_{\Phi}$.  This  leads to diagrams of cubic area (as a function of their boundary length) and so a cubic upper bound on the Dehn function.  For certain $\Phi$,  we improve this estimate to quadratic by noticing that $M_{\phi}$ is hyperbolic relative to a $\mathbb{Z}^2$ subgroup that receives no charges. This implies that only linearly many charges appear in the diagram, and thereby that the resulting van~Kampen diagram over $M_{\Phi}$ has quadratic area.  For  other $\Phi$ we define \textit{partial corridors} in van Kampen diagrams and then use Hall's Marriage Theorem to give a special pattern for discharging the diagrams, which again improves the Dehn function upper bound to quadratic.

As for obtaining the matching lower bounds, the Dehn function of $M_{\Phi}$ is always at least quadratic because $M_{\Phi}$  is not hyperbolic.   For certain $M_{\Phi}$, a result of Bridson and Gersten (see Lemma~\ref{BridsonGersten})    improves this to a cubic lower bound by identifying a suitable quasi-isometrically embedded abelian subgroup of $G$ to  which the action of $\Phi$ restricts. 

For Theorem~\ref{Z2astZ} the main innovation is for a case where  even though the $M_{\Phi}$ are not central extensions, $\Phi$  is such that one generator $b$ commutes with all but a particular generator $c$ that forms corridors. By specifying how  $c$-corridors have to join in van~Kampen diagrams over the quotient of $M_{\Phi}$ in which $b$ is killed, we are able to apply the electrostatic model to regions complementary to the $c$-corridors. We then define \textit{alternating corridors} which string together two types of partial corridors. We show that these alternating corridors can intersect themselves and each other at most once, and that every 2-cell in the diagram is contained in some alternating corridor. This lets us show that the area of the van~Kampen diagram in the quotient is at most quadratic in the length of the boundary word. The electrostatic model then produces a van~Kampen diagram with at most cubic area.

For the lower bounds of Theorem~\ref{F_k x F_l} we exhibit a family of words such that any van~Kampen diagram for one of these words has area we can bound below on account of having a belt of corridors of controlled length. For the upper bound we estimate the number of relators that need to be applied to convert a word $w$ representing the identity over the mapping torus of $F_k \times F_l$ to a word $v$ with $|v| \leq |w|$ that represents the identity in $F_k \rtimes_{\phi_1} \mathbb{Z}$, and then we use the  fact that the Dehn function of  $F_k \rtimes_{\phi_1} \mathbb{Z}$ is at most quadratic. The upper and lower bounds on the Dehn function are derived from two different notions of free group automorphism growth, which we reconcile by appealing to a number of results in the literature.

\section{Preliminaries} \label{Preliminaries}

We write $\abs{w}$ to denote the length of a word $w$.  Our conventions are $a^t := t^{-1} a t$ and $[a,b] := a^{-1} b^{-1} a b$. 

\subsection{Van Kampen diagrams, corridors, and Dehn functions} \label{sec: vK diagrams,corridors, DF}
These topics feature in many surveys, for instance \cite{BridsonHaefliger}.  Here are the essentials.  

Suppose $G=\langle X \! \mid  \! R \rangle$ is a finitely presented group (so  $X$ and $R$ are finite). Suppose $w$ is a word on $X \cup X^{-1}$ such that $w=1$ in $G$. A \emph{van~Kampen diagram} $\Delta$ for  $w$ is a simply-connected planar 2-complex with edges labeled by elements of $X$ and directed so that the following holds.    When traversing $\partial \Delta$ counterclockwise from some base vertex, we read off $w$, and around the boundary of each  2-cell in one direction or the other and from a suitable base vertex, we read an element of $R$.   (If an edge is traversed in the direction of its orientation, the positive generator is implied, and if  against its orientation, the inverse of the generator.) The 1-skeleton  $\Delta^{(1)}$ of $\Delta$ has the path metric in which every edge has length 1. The area of $\Delta$ is the number of 2-cells it has.  $\Area(w)$  denotes the minimum  area among  all van~Kampen diagrams with boundary word $w$. 

The \emph{Dehn function}  $\delta: \mathbb{N} \to \mathbb{N}$ of $\langle X \! \mid  \! R \rangle$ is  $ \delta(n)    :=    \max\set{\Area(w) \mid  \abs{w} \leq n \text{ and } w =1 \text{ in } G }$.

Up to the   equivalence relation $\simeq$ defined in Section~\ref{our results}, the Dehn function does not depend on the choice of finite presentation for $G$, and, moreover,  is a quasi-isometry invariant among finitely presented groups. In particular, we will need: 

\begin{prop}\label{finite index}
	If $G$ is finitely presented and $H\leq G$ is a finite index subgroup, then $H$ is also finitely presentable and $G$ and $H$ have equivalent Dehn functions.
\end{prop}

\textit{Corridors} appear in van~Kampen diagrams over a presentation $\langle X \! \mid  \! R \rangle$  when there is some $a \in X$ such that all relators $r \in R$ in which $a$ appears can be expressed as $w_1a^{\pm 1}w_2a^{\mp 1}w_3$ where $w_1$, $w_2$, and $w_3$ are words not containing $a^{\pm 1}$. Such presentations naturally arise for HNN-extensions, with $a$ being the  stable letter. Suppose $\Delta$ is a van~Kampen diagram for a word $w$ over such a presentation and that $w$ contains an $a$. This edge is either in the \textit{thin part} of that diagram---that is, this edge is not in the boundary of any 2-cell (as in Figure \ref{fig:thin})---or it is in the \textit{thick part} and there is a 2-cell in $\Delta$ with that edge in its boundary, as in Figure \ref{fig:thick}. This 2-cell will have exactly one other $a$-edge. In turn, this other $a$-edge either is in $\partial\Delta$ or is common with another 2-cell. This continues likewise and eventually must end elsewhere in the boundary. The resulting collection of 2-cells is an \emph{$a$-corridor}. The number of 2-cells involved is the \emph{length} of the corridor. An $a$-corridor is \textit{reduced} if it contains no two 2-cells sharing an $a$-edge for which the word around the boundary of their union is freely reducible to the identity in the group. 

\begin{remark} \label{reduced remark}
	Many of  the presentations we will work with  will have the form $\langle X, a \mid R, \ x^a = w_x; \  x \in X  \rangle$ where $X$ is some alphabet (not containing $a$), and $R$ and $\set{w_x \mid x \in X}$ are  sets of words on $X^{\pm 1}$.  An $a$-corridor in a diagram over such a presentation is reduced exactly when the word along the bottom is reduced.
\end{remark}

\begin{figure}[!ht]
	\centering
	\begin{subfigure}[t]{4cm}
		\centering
		\includegraphics[scale=0.27]{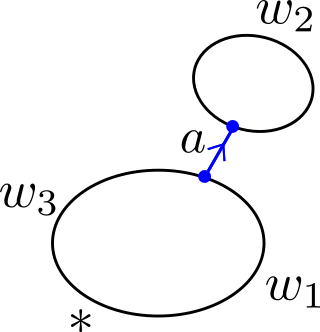}
		\caption{$a$-edge in the thin part of $\Delta$}
		\label{fig:thin}
	\end{subfigure}
	\begin{subfigure}[t]{4cm}
		\centering
		\includegraphics[scale=0.3]{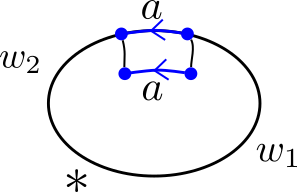} 
		\caption{$a$-edges in the thick part of $\Delta^{\prime}$}
		\label{fig:thick}
	\end{subfigure}
	\label{fig:thinthick}
	\begin{subfigure}[t]{5.7cm}
		\centering
		\includegraphics[scale=0.45]{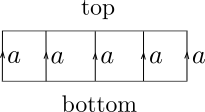}
		\caption{$a$-corridors inherit an orientation from edges labeled by $a$.}
		\label{fig:corridortopbottom}
	\end{subfigure}
	\caption{Corridors}
\end{figure}

Suppose $\Delta$ is a van Kampen diagram with $N$ $a$-corridors.  Then $N$ is at most half  the length of the boundary (at most half the number of $a^{\pm 1}$  in $w$). 
Since $a$-corridors cannot cross, removing all the $a$-corridors leaves $N+1$ connected subdiagrams called \textit{$a$-complementary regions}. The words around the perimeters of each of these regions contain no $a^{\pm 1}$. So analysis of the  lengths of the $a$-corridors and of the areas of the $a$-complementary regions can lead to estimates on the area of $\Delta$.

The \emph{dual tree} to the set of $a$-corridors has vertices corresponding to $a$-complementary regions, and has an edge between two vertices when an $a$-corridor borders the two corresponding $a$-complementary regions. (There is no vertex corresponding to the outside of the van Kampen diagram.)

\begin{definition}\label{def:partialcorridor} A letter $a$ forms \textit{partial corridors} when all the defining relations which contain both $a$ and $a^{-1}$ have the form of a corridor relation, $awa^{-1}=w^{\prime}$ for some $w$ and $w^{\prime}$ without $a$. A partial corridor is a maximal set of 2-cells joined by corridor relations as above.
	We refer to such 2-cells which contain one or more $a$  or $a^{-1}$   (but not both) in their boundary words as \textit{capping faces}, since they cap off partial corridors. 
\end{definition}

An $a$ in the boundary of a van~Kampen diagram will either be connected by a full $a$-corridor to another edge labeled by $a$ in the boundary, or   begins a partial $a$-corridor ending at one of the capping faces. An $a$-edge on a capping face is either connected to the boundary via a partial $a$-corridor (possibly of length zero), or is connected to an $a$-edge of another capping face via a partial $a$-corridor.

Like standard corridors, partial corridors cannot cross.  However,  there is no  immediate  control on the number of partial corridors in terms of $|w|$, since they may begin and end within the diagram.

\subsection{General bounds on Dehn functions of mapping tori of RAAGs}  
RAAGs  are (bi)automatic   \cite{HermillerMeier} and so have either linear or quadratic Dehn functions.  A finitely presented group is hyperbolic  if and only if it has linear Dehn function. Finite-rank free groups are hyperbolic.   Non-free RAAGs have $\Z^2$ subgroups and so are not hyperbolic (e.g.\ \cite{BridsonHaefliger} and  references therein).   So  RAAGs have either linear or  quadratic Dehn functions, the former case only occurring for  free RAAGs.  This will be useful for the following lemma.  

\begin{lemma} \label{Exponential}
	If $G$ is a non-free \textup{RAAG} and $\Psi \in \Aut(G)$, then the Dehn function  of $M_{\Psi}$ satisfies  $ n^2 \preceq \delta(n) \preceq 2^n$.  \end{lemma}

\begin{proof} Suppose $G$ is a non-free RAAG.  So $G$ has a  finite presentation $\langle X \! \mid \! R \rangle$ derived from a graph with at least one edge.  Then $G$ and hence $M_{\Psi}$ will contain a $\mathbb{\Z}^2$ subgroup. This implies that $M_{\Psi}$ is not hyperbolic and therefore $n^2 \preceq \delta(n)$ (again, \cite{BridsonHaefliger} and  references therein).	
	
	A word $w$ on the generators of $$M_{\Psi} \ = \  \langle X, t \mid   {R}, \  t^{-1}xt= \Psi(x),~ \forall x \in X \rangle$$  can be expressed as $t^{k_0}a_1t^{k_1}\cdots a_mt^{k_m}$ for some $a_1, \ldots, a_m \in X^{\pm 1}$ and some $k_1, \ldots, k_m \in \Z$.  Suppose $w$ represents the identity in $M_{\Psi}$. Then shuffling all the $t^{\mp 1}$ to the right, replacing each $a_i$ by the freely reduced word representing $\Phi^{\pm 1}(a_i)$ does not change the element of $M_{\Psi}$ represented. Eventually we arrive at $u t^{k_0 + \cdots + k_m} $ where $u$ is a word on  $X^{\pm 1}$ that represents $1$ in $G$ and $k_0 + \cdots + k_m =0$.  Applying $\Psi^{\pm 1}$ in this way to a word on $X^{\pm 1}$ increases its length by at most a constant factor, specifically by at most  $C:=\max_{a \in X}  |\Psi^{\pm 1}(a)|$.  So $m C^{\abs{k_0}+ \cdots + \abs{k_m}} \leq \abs{w} C^{\abs{w}}$ is an upper bound on both $\abs{u}$ and on the number of relation applications needed to convert $w$ to $u$.   
	
	The Dehn function of $G$ is at most quadratic, so $u$ can be reduced to the empty word using at most  a constant times $\abs{u}^2$ defining relations. Thus $\Area(w)$ is at most a constant times  $\abs{w} C^{\abs{w}} + (\abs{w} C^{\abs{w}})^2$, and therefore $\delta(n) \preceq 2^n$. \end{proof}

Our next lemma is the special case of Theorem~4.1 of  \cite{BridsonGersten} in which (in the notation of \cite{BridsonGersten}) $G=H$ and $K$ is quasi-isometrically embedded. We will call on this repeatedly to establish lower bounds on the Dehn functions.

\begin{lemma}[adapted from Bridson--Gersten~\cite{BridsonGersten}] \label{BridsonGersten} Suppose $K  = \langle k_1, \dots, k_m \rangle$ is a  quasi-isometrically embedded infinite abelian subgroup of a finitely presented group $G$.   If $\Phi \in \Aut(G)$ and  $\Phi(K) = K$, then the Dehn function  $\delta$  of ${\langle G, t~ |~ g^t=\Phi(g) \rangle}$ satisfies $${n^2\max_{1 \leq i \leq m}{ \abs{\Phi^{\pm  n}(k_i)}} \  \preceq \ \delta(n)}.$$ Equivalently, if $\phi = \Phi \restricted{K}$ is associated to the matrix $A$, and if 
	\begin{enumerate}
		\item $\phi$ has an eigenvalue $\lambda$ such that $|\lambda| \neq 1$, then $M_{\Phi}$ has exponential Dehn function. \label{BGexp}
		\item $\phi$ only has eigenvalues $\lambda$ such that $|\lambda|=1$, then if the largest Jordan block for $A$ is $c \times c$, then $n^{c+1} \preceq \delta(n)$. \label{BGpoly}
	\end{enumerate}
\end{lemma}

The following lemma allows us to specialize to convenient $\Psi$ when analyzing the Dehn functions of mapping tori. We include the proof because it is brief and the result is vital to this paper.
\begin{lemma}[Bogopolski \cite{Bogopolski}] \label{lemma: simplifying automorphisms doesn't change DF} The following mapping tori have equivalent Dehn functions:
	\begin{enumerate} 
		\item $M_{\Psi}$ and $M_{\Psi^n}$, for any $n \in \mathbb{N}$.
		\item $M_{\Psi}$ and $M_{\Psi^{-1}}$.
		\item $M_{\Psi_1}$ and $M_{\Psi_2}$ when $\Psi_1$ and $\Psi_2$ are conjugate in $\Out(G)$.
	\end{enumerate}
\end{lemma}
\begin{proof}[Proof, following \cite{Bogopolski}] 
	Map $M_{\Psi}$ onto $\langle t \rangle = \Z$ by killing $G$ and then onto $\mathbb{Z}/n\mathbb{Z}$ by the natural quotient map. The kernel of this composition is the index-$n$ subgroup $M_{\Psi^n}$. By Proposition \ref{finite index}, $M_{\Psi}$ and $M_{\Psi^n}$ have equivalent Dehn functions.
	
	As $w^t = \Psi(w)$ for all $w \in G$, it follows that $w^{t^{-1}} = \Psi^{-1}(w)$, so mapping $t \mapsto t^{-1}$ and fixing $G$ gives an isomorphism $M_{\Psi} \to M_{\Psi^{-1}}$. Thus $M_{\Psi}$ and $M_{\Psi^{-1}}$ have equivalent Dehn functions.  
	
	If $\Psi_1$ and $\Psi_2$ are  conjugate  in $\Out(G)$, there exists $\eta \in \Aut(G)$ and $h\in G$ such that $\Psi_2(g) = \eta^{-1}(\Psi_1(\eta(g^h)))$  for all $g \in G$. We will show that $M_{\Psi_1}$ and $M_{\Psi_2}$ are isomorphic. Consider $F: M_{\Psi_2} \to M_{\Psi_1}$ given by $x \mapsto \eta(x)$ for $x \in G$ and $t \mapsto t~\hat{h}$, where $\hat{h}:=\Psi_1(\eta(h))$.  It is a homomorphism because the relators $(g^{-1})^t\Psi_2(g)$ for $g \in G$ are mapped to the identity  in $M_{\Psi_1}$. Indeed,
	\begin{eqnarray*}F\left((g^{-1})^t\Psi_2(g)\right) \ = \  F\left((g^{-1})^t\eta^{-1}(\Psi_1(\eta(g^h)))\right) \  = \  \eta(g^{-1})^{t\hat{h}}\Psi_1( \eta(g^h)) \\
		\ = \  (\eta(g)^{-1})^{t\hat{h}} \Psi_1(\eta(g))^{\hat{h}} \ = \  \left((\eta(g)^{-1})^{t} \Psi_1(\eta(g))\right)^{\hat{h}} \  = \  1^{\hat{h}} \ =  \  1.
	\end{eqnarray*}
	It is certainly onto. This homomorphism has inverse given by $x \mapsto \eta^{-1}(x)$ for $x \in G$ and $t \mapsto t~\eta^{-1}(\hat{h}^{-1})$, so it is an isomorphism.
\end{proof}

\subsection{Growth and automorphims of \texorpdfstring{$\mathbb{Z}^2$}{Z2} }\label{growth}

For a matrix $A \in  \SL(2, \Z)$, let $|| A||$ denote the maximum of the absolute values of the entries in $A$.   
We say $A \in  \SL(2, \Z)$ has linear growth when the function  $\N \to \N$ mapping $n \mapsto ||A^n||$ is $\simeq$-equivalent to  $n \mapsto n$.

The following lemmas will allow us to specialize to convenient cases of $\Phi$ when analyzing Dehn functions of mapping tori $M_{\Phi}$ of $F_2 \times \mathbb{Z}$ and  $\mathbb{Z}^2 \ast \mathbb{Z}$. 

\begin{lemma}\label{linear growth} If $A \in \SL(2, \mathbb{Z})$ has linear growth, then there   are integers $\alpha$ and $k$ such that $k>0$ and $A^k$ is conjugate to $\left(\begin{smallmatrix}
	1 & \alpha \\
	0 & 1 \end{smallmatrix}\right)$  in  $\SL(2,\mathbb{Z})$. 
\end{lemma}

\begin{proof}
	As $A$ has linear growth,  Theorem~2.1 of \cite{BridsonGersten} tells us that there exists an integer $k >0$ such that  $A^k$ is $I+N$ for some non-zero matrix $N$ such that $N^2=0$.   
	As $N^2=0$, the trace of $N$ is zero, and    $N=\left( \begin{smallmatrix}
	a & b \\
	c & -a\\
	\end{smallmatrix}\right)$ for some integers $a, b, c$  not all zero such that   $a^2 = -bc$.  If $a=c=0$, then the result holds with $\alpha =b$.  So assume they are not both zero. Notice that  $N  
	\left(\begin{smallmatrix}
	a   \\
	c   \end{smallmatrix}\right) 
	= 
	\left(\begin{smallmatrix}
	0   \\
	0   \end{smallmatrix}\right)$.  So there are coprime integers  $p$ and $q$ (in particular not both zero) with  
	$N  
	\left(\begin{smallmatrix}
	p   \\
	q   \end{smallmatrix}\right) 
	= 
	\left(\begin{smallmatrix}
	0   \\
	0   \end{smallmatrix}\right)$.
	By Bezout, there are $r, s \in \Z$ such that $ps-qr=1$, and so   
	$B:= \left(\begin{smallmatrix}
	p   & r \\
	q   & s  \end{smallmatrix}\right)$ is in $\SL(2, \mathbb{Z})$. 
	And then $B^{-1}NB = \left(\begin{smallmatrix}
	0 & \alpha \\
	0 & 0 \end{smallmatrix}\right)$
	where $\alpha = 2ars +bs^2 -cr^2$, and the result follows. 
\end{proof}

\begin{lemma}\label{finite order examples} 
	Suppose $A \in \SL(2, \Z)$ has eigenvalues $\lambda^{\pm 1}$ with $\abs{\lambda}=1$, then either $\lambda \in \R$  or $A$ has finite order dividing six.
\end{lemma}

\begin{proof}
	As $A \in \SL(2, \Z)$,  $\lambda^2-\mbox{tr}(A) \lambda+1 =0$.
	So if  $\lambda$ is not real, then the discriminant $\mbox{tr}(A)^2-4<0$, and as $A$ has only integer entries, $\mbox{tr}(A)\in \{0,  \pm 1\}$.    Suppose then that  $\lambda = x+y i$ and $\lambda^{-1} = x-y i$.  Then, as $\mbox{tr}(A)= 2x$, we find $x \in \{0, \pm \frac{1}{2}\}$.  It follows that  $A$ is conjugate in $\SL(2, \C)$ to $\left( \begin{smallmatrix}
	e^{i\theta} & 0 \\
	0 & e^{-i \theta }  \\
	\end{smallmatrix} \right)$
	where  $\theta$ is $\pm \pi/2$, $\pm 2\pi/3$, or $\pm\pi/3$, and so $A$ has order dividing $6$. 
\end{proof}

\section{The electrostatic model for central extensions} \label{sec:electrostatic model}
Gersten and Riley's electrostatic model of \cite{GerstenRiley} is a method of constructing van~Kampen diagrams for central extensions. We will use it and variants to obtain upper bounds on the Dehn functions of some mapping tori.

Suppose a group $\Gamma$  is a central extension  $1 \to \mathbb{Z} \to  {\Gamma} \to \overline{\Gamma} \to 1$ with kernel $\Z = \langle c \rangle$. If $\overline{\Gamma}$ has presentation $$\mathcal{P}_{\overline{\Gamma}} \ = \ \langle X  \! \mid \! r_1= \cdots = r_n=1  \rangle,$$ then for some $k_1, \ldots, k_n \in \Z$, ${\Gamma}$ has presentation $${\mathcal{P}_{\Gamma} \ = \ \langle X, c \mid r_1=c^{k_1}, \ \dots, \ r_n=c^{k_n}, \ [c,x]=1,~\forall x \in X \rangle}.$$  

Suppose $w \in F(X \cup \{c\})$. Since $c$ is central, $w = \overline{w}c^m$ in $\Gamma$, for some $m \in \Z$ and $\overline{w}$ is $w$ with all $c^{\pm 1}$ removed. If $w$  represents the identity in $\Gamma$, the word $\overline{w} \in F(X)$ represents the identity in $\overline{\Gamma}$. We will describe how to construct a van~Kampen diagram for  $w$ over  $\mathcal{P}_{\Gamma}$ from  a diagram $\overline{\Delta}$ for $\overline{w}$ over $\mathcal{P}_{\overline{\Gamma}}$. 

We read  a  defining relator $r_{i_\sigma}$ clockwise or counterclockwise from an appropriate vertex $\ast_\sigma$  around the boundary of each 2-cell $\sigma$ in $\overline{\Delta}$. Now `charge'  every 2-cell: insert $|k_{i_\sigma}|$  loops at $\ast_{\sigma}$ each labeled with $c$'s and oriented in such a way that around the interior of the 2-cell we now read $r_{i_\sigma}c^{-k_{i_\sigma}}$ (to reflect  the relation  $r_{i_\sigma}=c^{k_{i_\sigma}}$), as in Figure \ref{fig:electrostatic charging}. If $C:= \max_i |k_i|$, then  at most $C \Area(\overline{\Delta})$  such loops labeled by $c$ are introduced by charging.  

\begin{figure}[!ht]
	\centering
	\includegraphics[scale=0.2]{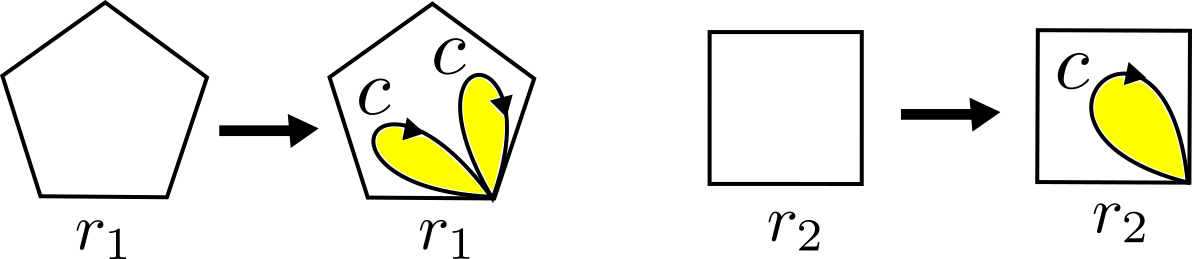}
	\caption{How `charges' would appear if $r_1=c^2$ and $r_2=c$.}
	\label{fig:electrostatic charging}
\end{figure}

To discharge, pick a \emph{geodesic spanning tree} $\mathcal{T}$ in $\overline{\Delta}^{(1)}$---that is, a maximal tree such that the distance in the tree from any vertex to the base vertex of $\overline{\Delta}$ is the same as its distance in $\overline{\Delta}^{(1)}$. In \cite{GerstenRiley}, for each introduced $c$-edge, a $c$-corridor is added which follows $\mathcal{T}$ to the root of the tree.  (Figures 4--7 in \cite{GerstenRiley} show how these corridors appear.) Each $c$-corridor has length bounded above by $\textup{Diam}(\overline{\Delta})$. This produces a diagram $\Delta^{\prime}$ for $\overline{w}c^m$ in $\Gamma$ with area at most $C \Area(\overline{\Delta})(\Diam(\overline{\Delta})+1)$.

As $w=\overline{w}c^m$ in   $\mathcal{P}_{\Gamma}$, there is a van~Kampen diagram $\Theta$ for $wc^{-m}\overline{w}^{-1}$ over $\mathcal{P}_{\Gamma}$. Since the arrangement of generators other than $c$ is the same in $w$ and in $\overline{w}c^m$, $\Theta$ can be filled with $c$-corridors and $\Area( wc^{-m}\overline{w}^{-1}) \leq |w|^2$. To get a diagram $\Delta$ for $w$ we wrap the diagram $\Theta$ around $\overline{\Delta}$ as in Figure \ref{fig: annularwrap}.    

\begin{figure}[!ht]
	\centering
	\includegraphics[scale=0.22]{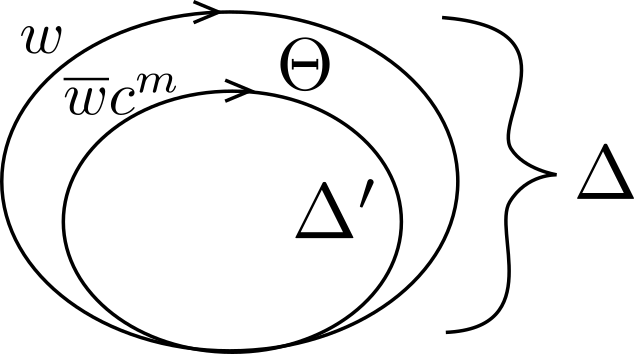}
	\caption{Constructing $\Delta$ from $\Delta^{\prime}$ and $\Theta$}
	\label{fig: annularwrap}
\end{figure}

This leads to the following theorem.

\begin{thm}[Gersten--Riley \cite{GerstenRiley}]\label{GerstenRiley} Suppose we have a central extension $1 \to \mathbb{Z} \to  \Gamma  \to \overline{\Gamma} \to 1$ of a finitely presented group $\overline{\Gamma}$, and $f,g : \N \to \N$ are functions such that for every   word $\overline{w}$ representing the identity in $\overline{\Gamma}$, there exists a van Kampen diagram $\overline{\Delta}$ such that $\Area(\overline{\Delta})\leq f(|\overline{w}|)$ and  the diameter $\Diam(\overline{\Delta})$ of the 1-skeleton of $\overline{\Delta}$ is at most $g(|\overline{w}|)$.  Then the Dehn function of $\Gamma$ is bounded above by a constant times $f(n)(g(n)+1)+n^2$.
\end{thm}

To use Theorem~\ref{GerstenRiley}, we need simultaneous control on both area and diameter of diagrams.  This is available in the setting we will be concerned with thanks to the following theorem of Papasoglu.  The radius $r(\Delta)$ of a van~Kampen diagram $\Delta$ is the minimal $N$ such that  for every vertex in $\Delta$ there is a path of length at most $N$ in the 1-skeleton of $\Delta$ from that vertex to $\partial \Delta$.  Since one can travel between any two vertices by concatenating shortest paths to the boundary with a path part way around the boundary, 
\begin{equation}
\textup{Diam}(\Delta) \ \leq  \  2r(\Delta) + |\partial \Delta|. \label{diam by rad}  
\end{equation}

\begin{thm}[Papasoglu \cite{Papa}]\label{Papa}
	For a group $G$ given by a finite presentation in which every relator has length at most three,  if $\Delta$ is a minimal area van Kampen diagram such that $|\partial \Delta |= n$ and $\Area(\Delta) \leq M n^2$, then  $r(\Delta)  \leq 12Mn$.
\end{thm}

Every finitely presentable group has such a presentation, and changing between two finite presentations of a group alters diameter and area by at most a multiplicative constant, so, in the light of \eqref{diam by rad}, Theorem~\ref{Papa} gives us:  

\begin{cor} \label{Papa cor}
	If a finitely presented group $G$ has Dehn function bounded above by a quadratic function, then there exists $K>0$ such that for every word of length $n$ representing the identity, there is a van~Kampen diagram whose area is at most  $Kn^2$ and whose diameter is at most $Kn$.     
\end{cor}

We are now ready to deduce:

\begin{corollary}\label{cor: electrostatic upperbounds} 
	All mapping tori $M_{\Phi}$ of $F_2 \times \mathbb{Z}$ have at most a cubic Dehn function.
\end{corollary}
\begin{proof}
	Bridson and Groves \cite{BridsonGroves} prove  that for all $\phi \in \Aut(F_2)$, $F_2 \rtimes_{\phi} \mathbb{Z}$   has a quadratic Dehn function, so   Corollary~\ref{Papa cor} applies and allows us to use  Theorem~\ref{GerstenRiley} to deduce that every central extension of $F_2 \rtimes_{\phi} \mathbb{Z}$   has at most cubic Dehn function.  Lemmas~\ref{pres lemma} and \ref{lemma: simplifying automorphisms doesn't change DF} together imply that $M_{\Phi}$ has at most a cubic Dehn function. \end{proof}

In Sections \ref{relhyperbolicF2xZ} and \ref{sec: F2xZ unit eigenvalues} we will refine this method to improve the upper bound from cubic to quadratic in special cases. In Section \ref{Z2*Z} we will adapt the arguments to certain related situations which fail to be central extensions.

\section{Mapping tori of \texorpdfstring{$G = F_2 \times \mathbb{Z} = \langle a,b \rangle \times \langle c \rangle$}{G=F2 x Z =<a,b> x <c>}} \label{F2xZ}
\subsection{Automorphisms of \texorpdfstring{$F_2 \times \mathbb{Z}$}{F2 x Z}}\label{types of autos}

Recall the notation of Theorem~\ref{F2xZ thm}:  $\Psi \in \Aut(F_2 \times \Z)$    induces $\psi \in \Aut(F_2)$ via the map $F_2 \times \Z \onto F_2$   killing the $\Z$ factor, and  $\psi$ induces  $\psi_{ab} \in \Aut(\mathbb{Z}^2)$  via  the abelianization map $F_2 \to \Z^2$, $g \mapsto g_{ab}$; and $p: F_2 \times \mathbb{Z} \to  \mathbb{Z}$ is projection onto the second factor.  Let  $\lambda^{\pm 1}$ be the (complex) eigenvalues of $\psi_{ab}$. We will prove Theorem~\ref{F2xZ thm} by separately addressing three comprehensive and mutually exclusive cases.  
\begin{enumerate}[label=(\roman*)]
	\item \label{one1} $\abs{\lambda}\neq 1$, 
	\item  \label{one2}  $\abs{\lambda} =  1$ and there exists $g \in F_2$ and $m \in \N$  such that $\psi_{ab}^m(g_{ab}) = g_{ab}$ and  $p(\Psi^m(g)) \neq 0$,
	\item   \label{one3} all other cases---that is, $\abs{\lambda} =  1$ and  for every $g \in F_2$ and every $m \in \N$, if $p(\Psi^m(g)) = 0$  then $\psi_{ab}^m(g_{ab}) = g_{ab}$.
\end{enumerate}

In Section~\ref{relhyperbolicF2xZ} we will prove that the Dehn function of the mapping torus $M_{\Psi}$ of $F_2 \times \mathbb{Z} = \langle a,b \rangle \times \langle c \rangle$  is quadratic in case~\ref{one1}.  In Section~\ref{sec: F2xZ unit eigenvalues} we will prove that it is cubic in case~\ref{one2} and is quadratic in case~\ref{one3}. First we narrow the family of automorphisms $\Psi$ that must be explored.

\begin{lemma} \cite[Proposition 5.1 (Nielsen)]{LyndonSchupp}
	\label{fate of a commutator} For all $~\theta \in \Aut(F(a,b))$, there is $h \in F(a,b)$ such that ${\theta^2([a,b])=[a,b]^h}$. 
\end{lemma}

\begin{proof} $\Aut(F(a,b))$ is generated by the following five elementary Nielsen transformations: $(a,b)$ maps to $(a^{-1}, b)$, $(a, b^{-1})$, $(b, a)$, $(ab, b)$, or $(a, ba)$. Each of these automorphisms sends $[a,b]$ to a conjugate of  $[a,b]^{\pm 1}$. 
\end{proof}

To prove Theorem~\ref{F2xZ thm}  it will suffice to focus only on the mapping tori of the form $M_{\Phi}$ described in the next lemma.

\begin{lemma} \label{pres lemma} 
	Let $G = F_2 \times \Z =  \langle a, b \rangle \times \langle c \rangle$. 	Given $\Psi \in \Aut(G)$, there exists $\Phi \in \Aut(G)$ satisfying the following properties:
	\begin{itemize}
		\item \label{form} there is $k \geq 0$ such that $[\Phi] = [\Psi^k]$ in $\Out(F_2\times \mathbb{Z})$,
		\item \label{Dehnfunction} the Dehn functions of  $M_{\Psi}$ and $M_{\Phi}$ are equivalent,
		\item\label{commutator} $\phi([a,b]) = [a,b]$,
		\item \label{det} $\phi_{ab}$ has  determinant $1$ and so is in $\SL(2,\Z)$,
		\item \label{poseigen} the eigenvalues of $\phi_{ab}$ are real and positive, and
		\item \label{central} $M_{\Phi}$ is a central extension of $M_{\phi}  =  \langle a, b,  t \mid a^{t}=\phi(a),  \ b^{t}=\phi(b)  \rangle$ by $\Z = \langle c \rangle$, 
	\end{itemize}
	where $\phi \in \Aut(F(a,b))$ is the map induced from $\Phi$ by killing $c$, and $\phi_{ab}$ is the map induced by the abelianization of $F(a,b)$ to $\mathbb{Z}^2$. 
	Thus there exist $k_a, k_b \in \Z$ such that $$M_{\Phi}   \ = \  \langle a, b, c, t \mid a^{t}=\phi(a)c^{k_a},  \ b^{t}=\phi(b)c^{k_b},\  c^{t}=c,  \ [a,c]=1, \ [b,c]=1 \rangle.$$
	
	Moreover, conditions \ref{one1}, \ref{one2}, and \ref{one3} above hold for $\Phi$  exactly when they hold for $\Psi$.
\end{lemma}

\begin{proof}
	The center $\langle c \rangle$ of $G$, being  characteristic,  is preserved by $\Psi$, so $\Psi$ maps $c$ to $c$ or $c^{-1}$, and  $\Psi^2$ maps $c$ to $c$.  On killing $c$, $\Psi$ induces some  $\theta \in \Aut(F(a,b))$. Therefore,
	the mapping torus $M_{\Psi^2} = G \rtimes_{\Psi^2} \Z$ is   
	$$\langle a, b ,c, t \mid a^{t}=\theta^2(a)c^{k_a},  \ b^{t}=\theta^2(b)c^{k_b},\  c^{t}=c, \ [a,c]=1, \ [b,c]=1\rangle$$
	for some $k_a, k_b \in \Z$.  
	By Lemma~\ref{fate of a commutator},  $\theta^2([a,b]) =[a,b]^h$ for some $h \in F(a,b)$. In case~\ref{one1}, define $\Phi$ to be  $\Psi^2$ composed with conjugation by $h^{-1}$. This new automorphism satisfies the properties above  by definition and by  Lemma~\ref{lemma: simplifying automorphisms doesn't change DF}.    Lemma~\ref{finite order examples} tells us that  in cases \ref{one2} and \ref{one3}, either $\phi_{ab}$ has a real unit eigenvalue or it has finite order dividing six.  Define $\Phi$ to be $\Psi^6$ composed with an appropriate inner automorphism so that $\phi = \textup{id}$. Then $\phi$ and $\Phi$ satisfy all the required properties.

	Conditions \ref{one1}, \ref{one2}, and \ref{one3} hold for  $\Phi$ exactly when they hold for $\Psi$: 
	Suppose that $\Psi$ satisfies condition \ref{one1}: Then $\Phi = \iota_{h^{-1}}\circ\Psi^2$ for some $h \in F(a,b)$. The map $\psi_{ab}$ has a non-unit eigenvalue if and only if $\phi_{ab}= \psi^2_{ab}$ has a non-unit eigenvalue. Suppose that $\Psi$ satisfies condition \ref{one2}:  Then $\Phi = \iota_{h^{-1}}\circ\Psi^6$ for some $h\in F(a,b)$. If for some $g \in F(a,b)$, we have that $g_{ab}$ is a fixed point of $\psi_{ab}^m$, then it is also a fixed point of  $\phi_{ab}^m=\psi_{ab}^{6m}$. If there is $m\in \mathbb{N}$ so that $\phi_{ab}^m$ has a fixed point, then $\psi_{ab}^{6m}$ will also have a fixed point.  Moreover, $p(\Phi^{m}(g)) = 6p(\Psi^m(g))$, so $p(\Phi^m(g))=0$ if and only if $p(\Psi^m(g))=0$. 
\end{proof}

\subsection{Theorem \ref{F2xZ thm} when \texorpdfstring{$\psi_{ab}$}{psi ab} has non-unit eigenvalues}\label{relhyperbolicF2xZ}

The primary tool for this section is \textit{relative hyperbolicity}, a concept  introduced by Gromov, and then developed by  Bowditch, Farb, Osin, and others \cite{Bowditch, Farb, Osin}.

Suppose $M_{\phi}$ is a group presented by $$\mathcal{P}_1 \ := \  \langle a, b,  t \mid a^{t}=\phi(a),  \ b^{t}=\phi(b) \rangle$$ where $\phi \in \Aut(F(a,b))$ such that $\phi([a,b]) = [a,b]$.

\begin{lemma} \label{rel hyp lemma}	If $\phi_{ab}$ has non-unit eigenvalues,    $M_{\phi}$ is strongly hyperbolic relative to the subgroup $H := \langle [a,b], \ t \rangle  \cong \Z^2$. 
\end{lemma}

\begin{proof}
	$M_{\phi}$ is the fundamental group of a finite-volume hyperbolic once-punctured torus bundle. In Theorem~4.11 of \cite{Farb}, Farb showed that such groups are strongly hyperbolic relative to their cusp subgroups. In our case, that is the subgroup $\langle [a,b], \ t \rangle$. (See also Section~4 of \cite{ButtonKropholler} for a survey of when mapping tori of free groups are relatively hyperbolic and acylindrically hyperbolic.)
\end{proof}

Consider the presentation 
$$\mathcal{P}_2 \ := \  \langle a, b, z,  t \mid a^{t}=\phi(a),  \ b^{t}=\phi(b), \ z=[a,b], \ z^t=z \rangle$$ for $M_{\phi}$ 
obtained from $\mathcal{P}_1$ by adding an extra generator $z$, an extra relation which declares that $z$ equals $[a,b]$ in the group, and a further extra relation which declares that $[a,b]$ commutes with $t$ (which is, a consequence of the other defining relations since  $\phi([a,b]) = [a,b]$).  Then $\langle  t, z \rangle \cong \Z^2$ is the subgroup $H$ of Lemma~\ref{rel hyp lemma}.  Refer to faces of a van~Kampen diagram over $\mathcal{P}_2$ as \emph{$\Z^2$-faces} when they  correspond to the relation $z^t=z$, and refer to the remaining faces as \emph{$\mathcal{R}$-faces}.  

\begin{lemma}  \label{diagram properties}
	There exists $C>0$ such that every word $w$ on $\set{a, b, t}^{\pm 1}$ of length $n$ that represents the identity has a van~Kampen diagram $\Delta$ over $\mathcal{P}_2$ with the following properties.    
	\begin{enumerate}
		\item The number of $\mathcal{R}$-faces is at most $Cn$.  \label{claim1}
		\item The number of $\Z^2$-faces in $\Delta$ is at most $Cn^2$.  \label{claim2}
		\item From every vertex of  $\Delta$ on the perimeter of an $\mathcal{R}$-face, there is a path to $\partial \Delta$ of length at most $Cn$ in the 1-skeleton of the union of the $\mathcal{R}$-faces.\label{claim3} 
	\end{enumerate}
\end{lemma}

\begin{proof}
	Let $H=\set{t,z} \cup \{h_{ij} \mid i,j \in \Z, \ (i,j) \neq (0,0), (1, 0), (0,1) \}$ be an alphabet, with a letter for each non-identity element of the subgroup $\langle  t, z \rangle \cong  \Z^2$ of $M_{\phi}$. Here $h_{ij}$ corresponds to the element represented by $t^{i}z^{j}$.  Let $S$ denote the set of words in $H^{\ast}$ that represent the identity in $M_{\phi}$. For example, $S$ includes the word $[z,t]$ and $h_{ij}z^{-j}t^{-i}$ for all $(i,j) \neq (0,0), (1, 0), (0,1)$.
	
	The presentation 
	$$ \mathcal{P}_3  \ := \  \langle a, b, H \mid a^{t} = \phi(a),  \ b^{t} = \phi(b), \ z=[a,b],       \  S \rangle$$  again gives $M_{\phi}$. Note that the elements $t$ and $z$ appear in $H$, and the defining relation $z^t=z$ appears in $S$.  Again, we will refer to  van Kampen diagram faces that correspond to elements of $S$ as \emph{$\mathbb{Z}^2$-faces}.

	Then $\mathcal{P}_3$ is a \emph{finite relative presentation} for $M_{\phi}$  with respect to $H$, as per Definition~2.2 of Osin in \cite{Osin}. Theorem~1.5  in \cite{Osin} says (in particular) that a finitely generated group which is hyperbolic relative to a subgroup in the sense of Farb, as is the case for $M_{\phi}$ relative to $H$ by Lemma~\ref{rel hyp lemma}, has a \emph{linear relative Dehn function}.  This implies that there exists $C>0$ such that for every word $w$ on $\{a,b,t\}^{\pm 1}$ representing the identity, there is a van~Kampen diagram $\hat{\Delta}$ over $\mathcal{P}_3$ whose number of $\mathcal{R}$-faces is at most $C |w|$.  
	
	Osin proves further facts that we will need concerning  the geometry of $\hat{\Delta}$. A diagram for the word $w$ is of \textit{minimal type} over all diagrams for $w$ if under lexicographic ordering it minimizes
	\begin{center} ($N_{\mathcal{R}} =$ \# of $\mathcal{R}$-faces, $ \ \ N_{\Z^2} =$ \# of $\Z^2$-faces, $\ \ E=$ total \# of edges). \end{center}
	Choose $\hat{\Delta}$ be of minimal type.	
	
	Let $M$ be the maximum length of the relators $a^{t}\phi(a)^{-1}$,  $b^{t}\phi(b)^{-1}$,  $z[a,b]^{-1}$,  and  $z^tz^{-1}$. Call an edge of $\hat{\Delta}$ \emph{internal to the $\Z^2$-faces} when it  has $\Z^2$-faces (or a $\Z^2$-face) on both sides. Osin (Lemma~2.15  of \cite{Osin}) tells us that if $\hat{\Delta}$ is of minimal type then it has no edges which are internal to the $\Z^2$-faces and deduces (Corollary 2.16) that the sum of the lengths of the perimeters of $\Z^2$-faces in $\hat{\Delta}$ is at most $|w| + M N_{\mathcal{R}}$. 
	
	Suppose that $w$ is a word on $\{a,b, t\}^{\pm 1}$ and take $\hat{\Delta}$ to be a diagram of minimal type for $w$ over $\mathcal{P}_3$.	
	
	The words around $\mathcal{R}$-faces only include the letters $t, z, a$ and $b$, so they can overlap $\Z^2$-faces only in edges labeled by $t$ and $z$. Therefore the word around each $\Z^2$-face  is a word on  $\{t, z\}^{\pm 1}$ since every edge in the boundary of a $\Z^2$-face is either in $\partial \hat{\Delta}$ or is also in the boundary of an $\mathcal{R}$-face. Let $\Delta$ be a diagram obtained from $\hat{\Delta}$ by excising all $\Z^2$-faces and replacing each $\Z^2$-face with the appropriate minimal area diagram over $\langle  t, z \mid z^t =z \rangle$.  So $\Delta$ is a van~Kampen diagram over $\mathcal{P}_2$. By Osin's Theorem~1.5, as discussed above, $\Delta$ satisfies \eqref{claim1}.   As the Dehn function of  $\langle  t, z \mid z^t =z \rangle$ enjoys a quadratic upper bound, and, given the bound on the lengths of the boundaries of $\Z^2$-faces explained in the previous paragraph,  $\Delta$ also satisfies \eqref{claim2}.    
	
	Because of the minimality assumption on the number of  $\Z^2$-faces, no two $\Z^2$-faces will have a vertex in common in $\hat{\Delta}$: two  $\Z^2$-faces with a vertex in common could be replaced by a single $\Z^2$-face.  Also the boundary circuit of any $\Z^2$-face in $\hat{\Delta}$ will be a simple loop.  This is because $E(\hat{\Delta})$ is minimal:  a  $\Z^2$-face with a non-simple loop as its boundary circuit could be excised and a $\Z^2$-face with a shorter and simple boundary loop inserted in its place. Thus the $\Z^2$-faces form disjoint \emph{islands} in $\hat{\Delta}$ and there are no $\mathcal{R}$-faces enclosed within these islands. In the light of this, \eqref{claim3} follows from \eqref{claim2}.   
\end{proof}

\begin{proof}[Proof of Theorem~\ref{F2xZ thm} in Case~\ref{one1}]
	We suppose $\Psi \in \Aut(F_2 \times \Z)$.  By Lemma~\ref{pres lemma} there exists $\Phi$ so that $M_{\Phi}$ and $M_{\Psi}$ have equivalent Dehn function, and $M_{\Phi}$  has presentation
	$$\langle a, b ,c, t \mid a^{t}=\phi(a)c^{k_a},  \ b^{t}=\phi(b)c^{k_b},\  c^{t}=c, \ [a,c]=1, \ [b,c]=1\rangle,$$
	which is a central extension of 
	$$ \langle a, b,  t \mid a^{t}=\phi(a),  \ b^{t}=\phi(b) \rangle$$ where $\phi \in \Aut(F(a,b))$ has the property that $\phi([a,b]) = [a,b]$. 
	If $z = [a,b]$, then in $M_{\Phi}$ $$z^t  \   =  \  [a^t,b^t] \  =  \  [\Phi(a),\Phi(b)] \  = \  [\phi(a)c^{k_a}, \phi(b)c^{k_b}]  \  = \ [\phi(a), \phi(b)]  \ = \   \phi ([a,b])  \  = \  [a,b]   \ = \  z.$$
	We change presentations, and work with the central extension 
	$$\mathcal{Q}   \ := \  \langle a, b ,c, t, z \mid a^{t}=\phi(a)c^{k_a},  \ b^{t}=\phi(b)c^{k_b},\  c^{t}=c, \ [a,c]=1, \ [b,c]=1, \ z = [a,b], \ z^t=z \rangle,$$
	of  $$ \mathcal{P}_2  \ =  \  \langle a, b,  t, z \mid a^{t}=\phi(a),  \ b^{t}=\phi(b), \ z=[a,b],  \  z^t =z \rangle.$$
	
	Suppose $w$ is a word of length $n$ representing the identity in $\mathcal{Q}$.  Let $\overline{w}$ be $w$ with all $c^{\pm 1}$ deleted.  
	
	Let $\overline{\Delta}$ be a van~Kampen diagram for $\overline{w}$ as per Lemma~\ref{diagram properties}.  Given \eqref{claim3} of that lemma, there is a  forest $\mathcal{F}$ in the 1-skeleton of the union of the $\mathcal{R}$-faces in $\overline{\Delta}$ joining every vertex  of an $\mathcal{R}$-face to $\partial \overline{\Delta}$ by a path of length at most $Cn$ through the 1-skeleton of the $\mathcal{R}$-faces.    
	
	Charge $\overline{\Delta}$. Given that the defining relation $z^t =z$ is unchanged on lifting to  the central extension, the  $\mathbb{Z}^2$-faces of Lemma~\ref{diagram properties}\eqref{claim2}, are unchanged. There are $Cn^2$ such $\Z^2$ faces. Let $m = \max\{|k_a|, |k_b|\}$. The remaining $Cn$ $\mathcal{R}$-faces of Lemma~\ref{diagram properties}\eqref{claim1}, each acquire at most $m$ charges. These are discharged by adding partial $c$-corridors that follow the forest $\mathcal{F}$ to the boundary and then around the boundary to a base vertex. Each partial $c$-corridor has length at most $(C+1)n$: the length of the path to the boundary is at most $Cn$ by Lemma~\ref{diagram properties} \eqref{claim3} and the length of the path to the base vertex is at most $n$. In total then, $c$-partial corridors contribute at most $(C+1)^{2}mn^2$ 2-cells to the new diagram. The result is a diagram over $\mathcal{Q}$ of area at most $((C+1)^{2}m +C) n^2$ for a word $\overline{w}c^k$, which has length less than $n$. By adding in an annular region to rearrange $\overline{w}c^k$  to $w$, as per the electrostatic model of Section~\ref{sec:electrostatic model}, it follows that $w$ has a diagram over $\mathcal{Q}$ of area at most $({(C+1)}^{2}m +C+1) n^2$.
\end{proof}

\
\subsection{Theorem \ref{F2xZ thm} in the case where all eigenvalues of \texorpdfstring{$\psi_{ab}$}{psi ab} are unit}\label{sec: F2xZ unit eigenvalues}

We begin by arguing that for the purpose of determining Dehn functions we can further specialize the family of presentations to examine:  
\begin{lemma} \label{specialize again}
	Suppose that $\Phi \in \Aut(G)$  is as per Lemma~\ref{pres lemma} and that the eigenvalues of  $\phi_{ab}$ are $1$.  Then there exists $\Xi \in \Aut(G)$ such that  the eigenvalues of   $\xi_{ab}$ are also $1$, the Dehn functions of  $M_{\Phi}$ and $M_{\Xi}$ are equivalent, and 
	$$M_{\Xi}=\langle a, b, c, t \mid a^t=ab^{\beta}c^{k_a}, b^t= bc^{k_b}, c^t =c, \ [a,c]=1, \ [b,c]=1\rangle$$
	for some $\beta \in \Z$.    Moreover,  conditions \ref{one2} and \ref{one3} of Section \ref{types of autos} hold for $\Psi$  exactly when they hold for $\Xi$, and they are characterized by $k_b \neq 0$ and $k_b = 0$, respectively.   
\end{lemma}

\begin{proof}
	Let $\Phi$ be as per Lemma~\ref{pres lemma}: $\Phi(a) = \phi(a)c^{k_a'}, \  \Phi(b) = \phi(b)c^{k_b'}, \ \Phi(c) =c$, for some   $k_a', k_b' \in \Z$ and some ${\phi \in \Aut(F(a,b))}$  such that $\phi_{ab}$ has determinant $1$. We assume its only eigenvalue is 1, and thus there is some $w \in F_2$  such that  $\phi_{ab}(w_{ab})= w_{ab}$.
	
	We will now show that there is $\Xi \in \Aut(F_2 \times \mathbb{Z})$ such that $[\Phi]$ and $[\Xi]$ are conjugate in $\Out(F_2 \times \mathbb{Z})$ and   $\xi = \Xi \!\! \upharpoonright_{F(a,b)}$ maps $b \mapsto b$.

	As  $\phi_{ab}$ has only  eigenvalue 1, it is either the identity or it has linear growth.  So (by Lemma~\ref{linear growth} in the latter case) $\phi_{ab}$ is conjugate in $\SL(2, \Z)$ to $\left(\begin{smallmatrix}
	1 & \beta \\
	0 & 1 \end{smallmatrix}\right)$ for some $\beta \in \Z$. On account of the standard isomorphism between $\Out(F_2)$ and $\GL(2,\Z)$,  $[\phi]$ is conjugate in $\Out(F_2)$ to $[\xi]$ where $\xi(a) = ab^{\beta}$ and $\xi(b) = b$. So $\xi= f^{-1}\circ \phi \circ f\circ \iota_g$ for some $f \in \Aut(F_2)$ and some $\iota_g \in \textup{Inn}(F_2)$.  
	We lift $f, \xi, \iota_g \in \Aut(F_2)$   to $F, \Xi, \hat{\iota}_g \in \Aut(F_2 \times \Z)$ by defining $F(gc^k) = f(g)c^k$ for $g \in F_2$ and $k \in \mathbb{Z}$, by taking $\hat{\iota}_g$ to be conjugation by $g$,  and by defining $\Xi:= F^{-1} \circ \Phi \circ F \circ \hat{\iota}_g$.  Because $c$ is central, $\hat{\iota}_g(c) = c$. In particular,  $$\Xi: \ \ a \mapsto ab^{\beta}c^{k_a}, \ \ b \mapsto bc^{k_b}, \ \ c \mapsto c,$$
	for some $k_a, k_b \in \mathbb{Z}$. (Note that $p(\Phi(b)) = k_b'$ and $p(\Xi(b)) = k_b$ may not be equal, as $\Phi(f(b)^{g^{-1}})$ and $\Phi(b)$ will not generally have the same index sum of $c$.)

	Therefore $M_{\Xi}$  has the presentation claimed and $\xi_{ab}$ has  only $1$ as an eigenvalue, as required.  And, by Lemma \ref{lemma: simplifying automorphisms doesn't change DF}, the mapping tori $M_{\Phi}$ and $M_{\Xi}$ have equivalent Dehn functions.

	Next we will show that
	\begin{enumerate}
		\item If $\phi_{ab} \neq \textup{id}$, then  $p(\Xi(b))=k_b\neq 0$ if and only if $p(\Phi(w)) \neq 0$. \label{nonidcase}
		\item \label{idcasew}  If $\phi_{ab} = \textup{id}$, then exactly one of the following hold:
		\begin{enumerate}
			\item  $p(\Phi(x)) = 0$ for all $x \in \langle a,b \rangle$, in which case $\Phi \in \Inn(F_2 \times \Z)$ and $\Xi =\textup{Id}$, \label{idcasewquadratic}
			\item   $p(\Phi(x)) \neq 0$ for some $x$, in which case  $p(\Xi(a))$ or $p(\Xi(b))$ is non-zero. \label{idcasecubic}
		\end{enumerate}
	\end{enumerate}
	 
	We wish to compare $p(\Phi(w))$ and $p(\Xi(b))$. Let $w^{\prime} = f(b)$. The following calculation shows   that $w_{ab}^{\prime}$ is another fixed point of $\phi_{ab}$ and that $p(\Phi(w^{\prime})) =p(\Xi(b))= k_b$: 
	$$\Phi(w^{\prime}) \ = \  F\circ \Xi \circ \iota_{g^{-1}}\circ F^{-1}(f(b)) \ = \  F(\Xi (b^{g^{-1}}))  \ = \  F(b^{g^{-1}}c^{p(\Xi(b))}) \ = \ f(b)^{g^{-1}}c^{p(\Xi(b))} \ = \ (w')^{g^{-1}}c^{p(\Xi(b))}.$$
	
	Now we prove \ref{nonidcase}.  If $\phi_{ab} \neq \mbox{id}$, then since $w_{ab}$ and $w^{\prime}_{ab}$ are both fixed by $\phi_{ab}$, $w_{ab} = dw^{\prime}_{ab}$ for some $d\neq 0$, and therefore $p(\Phi(w)) = dp(\Phi(w^{\prime})) = dk_b$. So $p(\Xi(b))=k_b\neq 0$ if and only if $p(\Phi(w)) \neq 0$.

	Next we prove~\ref{idcasew}.  If $\phi_{ab} = \mbox{id}$, then $\Xi$  maps $a \mapsto ac^{k_a}$, $b \mapsto bc^{k_b}$,  and $c \mapsto c$  for some $k_a, k_b \in \mathbb{Z}$, and so $\xi_{ab} = \mbox{id}$ also. 
	
	Lemma~\ref{pres lemma} shows that $\Psi$ and $\Phi$ either both satisfy condition~\ref{one2}  or both satisfy condition~\ref{one3}.
	
	Under  case \ref{nonidcase}, it is immediately evident that, as required, $\Phi$ and $\Xi$ either both satisfy condition~\ref{one2}  or both satisfy condition~\ref{one3}, and so this holds for  $\Psi$ and $\Xi$ also.  
	
	Finally consider case~\ref{idcasew}.  When $\phi_{ab}= \mbox{id}$,   condition~\ref{one2} `there exists $g \in F_2$ and $m \in \N$  such that $\phi_{ab}^m(g_{ab}) = g_{ab}$ and  $p(\Phi^m(g)) \neq 0$' amounts to `there exists $g \in F_2$ such that  $p(\Phi(g)) \neq 0$.' Observations  \ref{idcasewquadratic} and  \ref{idcasecubic} show that this holds for $\Phi$ if and only if it holds for $\Xi$. Again $\Phi$ and $\Xi$ either both satisfy condition~\ref{one2}  or both satisfy condition~\ref{one3}, and so this holds for  $\Psi$ and $\Xi$ also. 
\end{proof}

\begin{proof}[Proof of Theorem~\ref{F2xZ thm} in Case \ref{one2}.]
	By Lemmas~\ref{pres lemma} and \ref{specialize again}, for the purpose of calculating the Dehn function we may work with $$M_{\Xi}=\langle a, b, c, t ~|~ a^t=ab^{\beta}c^{k_a}, b^t= bc^{k_b}, c^t =c, \ [a,c]=1, [b,c]=1\rangle$$ where $k_b \neq 0$. 
	The subgroup $K := \langle b, c  \mid  [b,c] \rangle \cong \mathbb{Z}^2$ quasi-isometrically embeds in $F_2 \times \langle c \rangle$  and $\Xi(K) \subseteq K$.  So, by Lemma~\ref{BridsonGersten}, ${n \mapsto n^2 \max\{|\Xi^n(b)|, |\Xi^n(c)|\}= n^2(k_bn+1)}$ is a lower bound for the Dehn function of $M_{\Xi}$. This lower bound is  cubic (as $k_b \neq 0$), matching our upper bound from Corollary~\ref{cor: electrostatic upperbounds}, so the claim is established.
\end{proof}

We now turn to Case~\ref{one3}. This time, Lemmas~\ref{pres lemma} and \ref{specialize again} allow us to work with $M_{\Xi}$ which has the form $$M_{\Xi}=\langle a, b, c, t ~|~ a^t=ab^{\beta}c^{k_a}, b^t= b, c^t =c, \ [a,c]=1, [b,c]=1\rangle$$ where, $\beta$ is non-zero. 
(The case $\beta =0$ and $k_a \neq 0$ is covered by Theorem~\ref{F2xZ thm} in Case \ref{one2}---the Dehn function of this mapping torus is cubic.) 

The methods of Case~\ref{one1} cannot be used here. Indeed, Button and R. Kropholler \cite{ButtonKropholler} have shown that for $\xi$ with this form, $M_{\xi}$ is not strongly hyperbolic relative to any finitely generated proper subgroup, so van~Kampen diagrams over $M_{\xi}$ do not decompose into uncharged islands with linear-area complement.  Instead will use a variant of the electrostatic model whereby the diagram will be discharged along \textit{partial corridors} (see Section~\ref{sec: vK diagrams,corridors, DF}) in a manner controlled by an application of Hall's Marriage Theorem, which we now review. 

A subgraph $F$ of a graph $\Gamma$ is a \emph{1-factor for} $\Gamma$ if it contains all vertices of $\Gamma$ and each vertex meets precisely one edge of $F$.  In other words, a 1-factor pairs every vertex with a neighbor. We will be interested in the following special case:

\begin{lemma} \label{lemma: Hall's Marriage Theorem} A $k$-regular bipartite graph  $\Gamma$  with $k\geq 1$ has a 1-factor.
\end{lemma}
This is a consequence of Hall's Marriage Theorem. See \cite{Diestel} for a proof.

\begin{proof}[Proof of Theorem~\ref{F2xZ thm} in Case~\ref{one3}.]

	By Lemma~\ref{specialize again},  it suffices to prove that $M_{\Phi}$, presented by $$\mathcal{P}= \langle a, b, t, c  \mid  a^t=ab^{\beta}c^{k_a}, \ b^t=b,\ ac=ca,\ bc=cb, \ ct=tc \rangle,$$ has quadratic Dehn function.   
	If $k_a =0$, then $M_{\Phi}\isom M_{\phi}\times \langle c \rangle$ and so the Dehn functions of $M_{\Phi}$ and $M_{\phi}$ agree and will be quadratic.	Therefore we may restrict our attention to the case where $\beta$ and  $k_a$ are both non-zero.

	Van Kampen diagrams over $\mathcal{P}$ have both partial $b$-corridors and partial $c$-corridors. $M_{\Phi}$  is a central extension of $M_{\phi}$ by $\langle c \rangle$, where $M_{\phi}$ is presented by   
	$$\mathcal{Q} \ = \ \langle a, b, t \mid a^t=ab^{\beta}, b^t=b \rangle.$$
	Van Kampen diagrams over  $\mathcal{Q}$ may have partial $b$-corridors.

	Suppose $w$ is a word of length $n$  representing the identity in $\mathcal{P}$.  Let $\overline{w}$ be $w$ with all $c^{\pm 1}$ removed. Then $w = \overline{w}c^m$ in $M_{\Phi}$ for some $m \in \Z$ and $|\overline{w}| \leq |w|$. Since $\mathcal{Q}$ has a quadratic Dehn function, there exists a minimal area diagram $\overline{\Delta}$ for $\overline{w}$ over  $\mathcal{Q}$  such that $\Area(\overline{\Delta}) \leq C |\overline{w}|^2$. We charge $\overline{\Delta}$ by replacing 2-cells in $\overline{\Delta}$ with 2-cells labeled by the defining relators from $\mathcal{P}$, as in the first steps of the Electrostatic Model (see Section \ref{sec:electrostatic model}). What follows is a scheme for adding in 2-cells to `discharge' $\overline{\Delta}$ so as to create a diagram for $\overline{w}c^m$ over $\mathcal{P}$.
	
	The  idea is that if we can pair off oppositely-oriented capping faces that are joined by partial $b$-corridors, then we can add in partial $c$-corridors following the $b$-corridors, as in Figure \ref{fig:follow b corridor}, in order to discharge the $c$-edges in our diagram.  As $c$ is central in $\mathcal{P}$, partial $c$-corridors can be run alongside this partial $b$-corridor, and the word one reads along both the top and bottom  of the $c$-corridor will be the same as that word along the top and bottom of the $b$-corridor, namely some power of $t$. We wish to find a consistent way of partnering vertices so that we can replicate the picture in Figure~\ref{fig:follow b corridor}, adding in partial $c$-corridors to discharge between partners throughout the van Kampen diagram, with no leftover charges to consider.  
	
	\begin{figure}[!ht]
		\centering
		\vcenteredhbox{\includegraphics[scale=1]{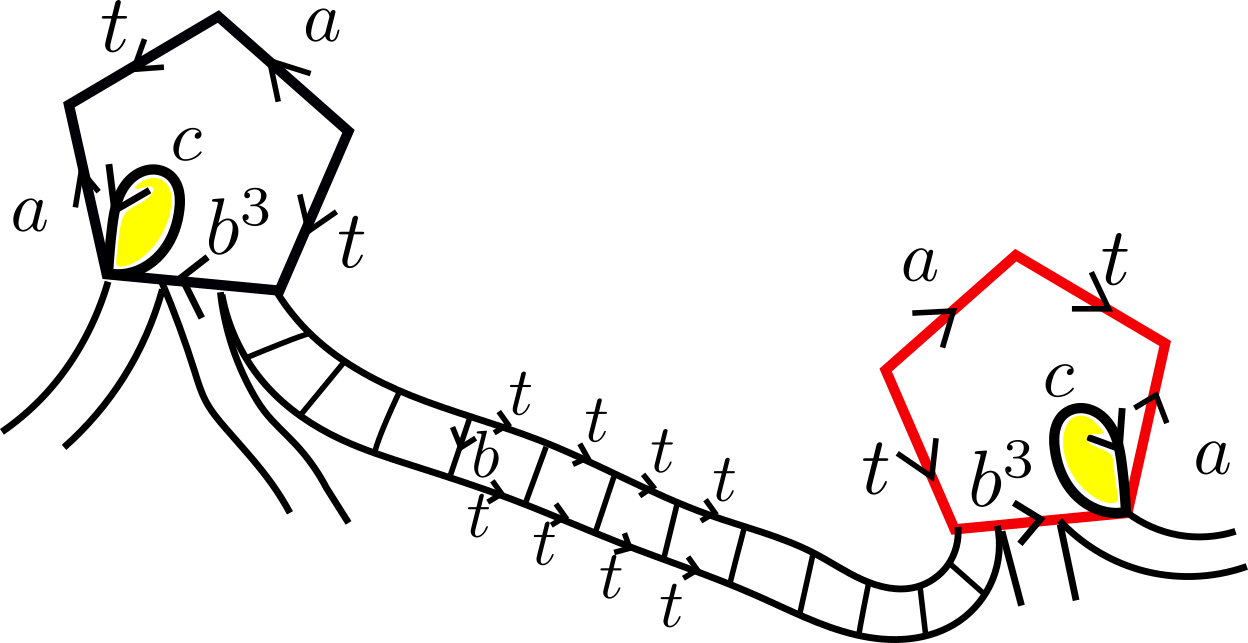}} \hspace{1cm} \vcenteredhbox{\includegraphics[scale=1]{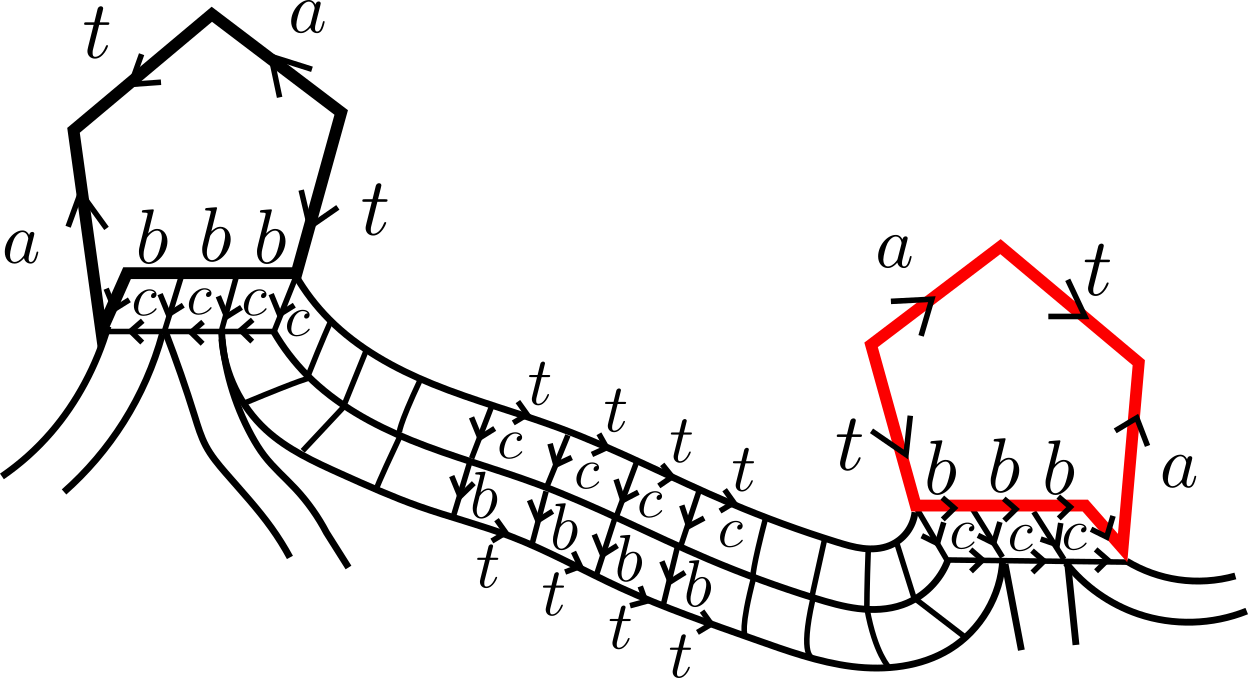}}
		\caption[Discharging $c$-partial corridors along $b$-partial corridors]{If a partial $b$-corridor joins two capping faces in $\overline{\Delta}$, their $c$-charges can be discharged by adding partial $c$-corridors   `following' that partial $b$-corridor.}
		\label{fig:follow b corridor}
	\end{figure}
	
	\textbf{I. Modeling $\overline{\Delta}$ with a graph.}  Construct a planar graph with multi-edges, $\Gamma$, from $\overline{\Delta}$ as illustrated in Figure \ref{fig:partial corridor}: $\Gamma$ has a black vertex for each capping face in $\overline{\Delta}$; whenever two capping faces are connected by a partial $b$-corridor, possibly of length zero, an edge connects the corresponding vertices (two vertices may share multiple edges); we also add an edge and a white vertex to $\Gamma$  for each partial $b$-corridor that goes to the boundary. Every black vertex in the graph $\Gamma$ is degree $|\beta|$ and every white vertex has degree 1.   
	\begin{figure}[!ht]
		\centering
		\vcenteredhbox{\includegraphics[scale=0.65]{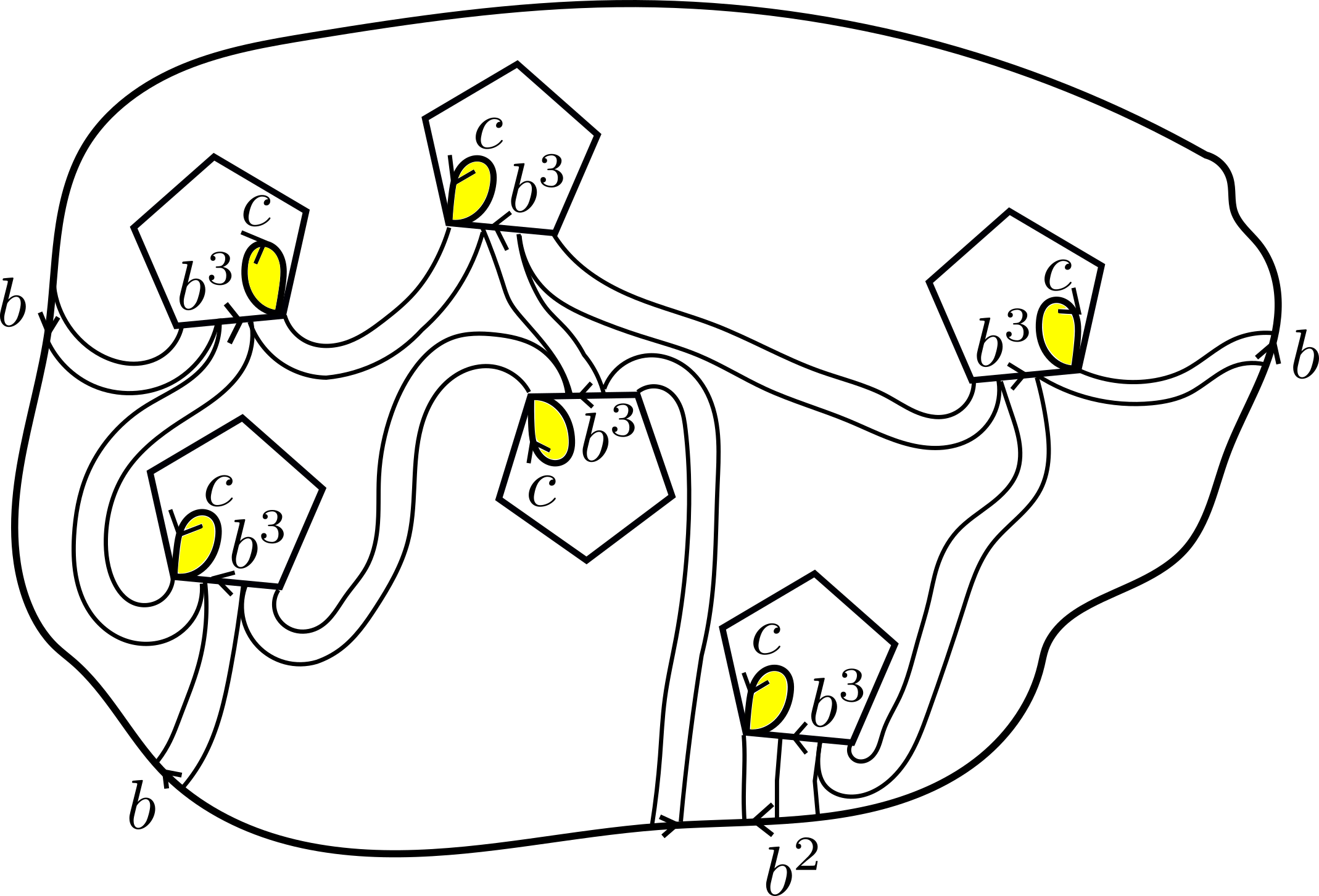}}\hspace{1cm} \vcenteredhbox{\includegraphics[scale=1.2]{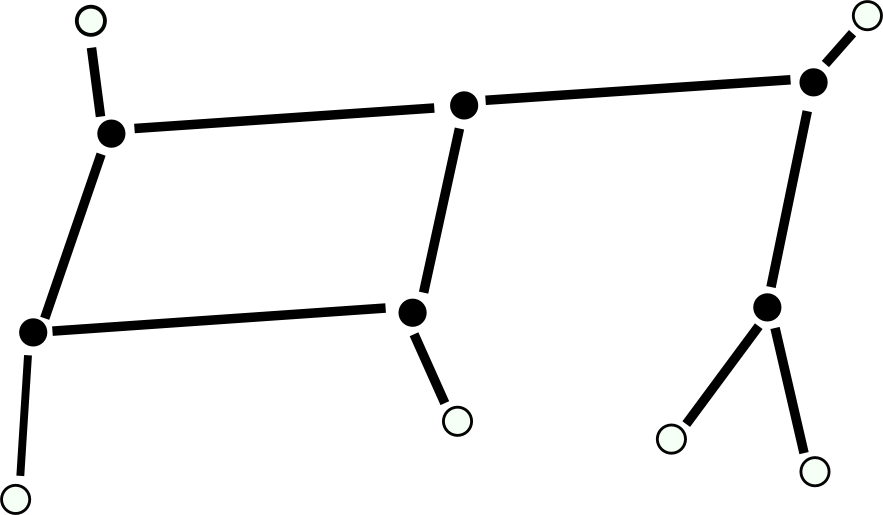}}
		\caption{From capping faces and partial-corridors in $\overline{\Delta}$, construct a graph $\Gamma$.  Black vertices correspond to capping faces, white vertices correspond to 1-cells labeled $b$ in $\partial\overline{\Delta}$, and the edges correspond to partial $b$-corridors. }
		\label{fig:partial corridor}
	\end{figure}
	
	The graph $\Gamma$ is naturally bipartite (but not generally black-white bipartite, as you can see in Figure~\ref{fig:partial corridor}):  partition the black vertices according to whether they correspond to capping faces with clockwise or anticlockwise oriented $b$-edges, and extend this partition to the white vertices.

	\textbf{II. Building a regular bipartite graph.} We would like to apply Corollary \ref{lemma: Hall's Marriage Theorem}, but  $\Gamma$ may not be regular: black vertices   have degree $\abs{\beta}$, but white vertices have degree $1$. So, as illustrated in Figure \ref{fig: bipartite regular graph}, we  construct a regular graph $\hat{\Gamma}$ which has $\Gamma$ as a subgraph.   Take $|\beta|$ many copies of $\Gamma$, and identify the white vertices in each of the copies.   That is, $$\hat{\Gamma} \ := \ \displaystyle \left(\bigsqcup_{i=1}^{|\beta|} \Gamma\times\{i\}\right)/ \sim,$$ where  $(v,i) \sim (v,j)$ for all $i,j$ when $v$ is a white vertex.  
	White vertices are degree one, so the identification of $|\beta|$ copies of $\Gamma$ forces $\hat{\Gamma}$ to be a $|\beta|$-regular graph. If $\Gamma$ is bipartite with respect to a partition $A \sqcup B$ of its vertices, then $\hat{\Gamma}$ is bipartite with respect to $$\left( \bigcup_{i=1}^{|\beta|}A \times \set{i} \right)/ \sim   \ \bigsqcup \  \left( \bigcup_{i=1}^{|\beta|}B  \times \set{i}  \right)/ \sim.$$
	\begin{figure}[!ht]
		\centering
		\begin{subfigure}[t]{5cm}
			\includegraphics[scale=0.8]{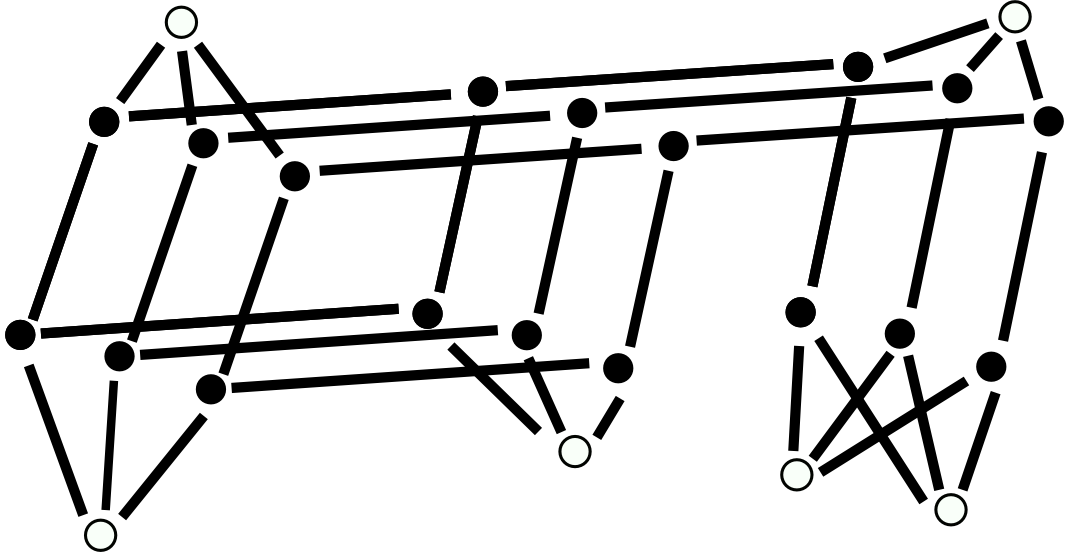}
			\caption{The regular bipartite graph $\hat{\Gamma}$: $|\beta|$ copies of $\Gamma$ glued together at degree-1 vertices}
			\label{fig: bipartite regular graph}
		\end{subfigure}
		\begin{subfigure}[t]{5cm}
			\includegraphics[scale=0.8]{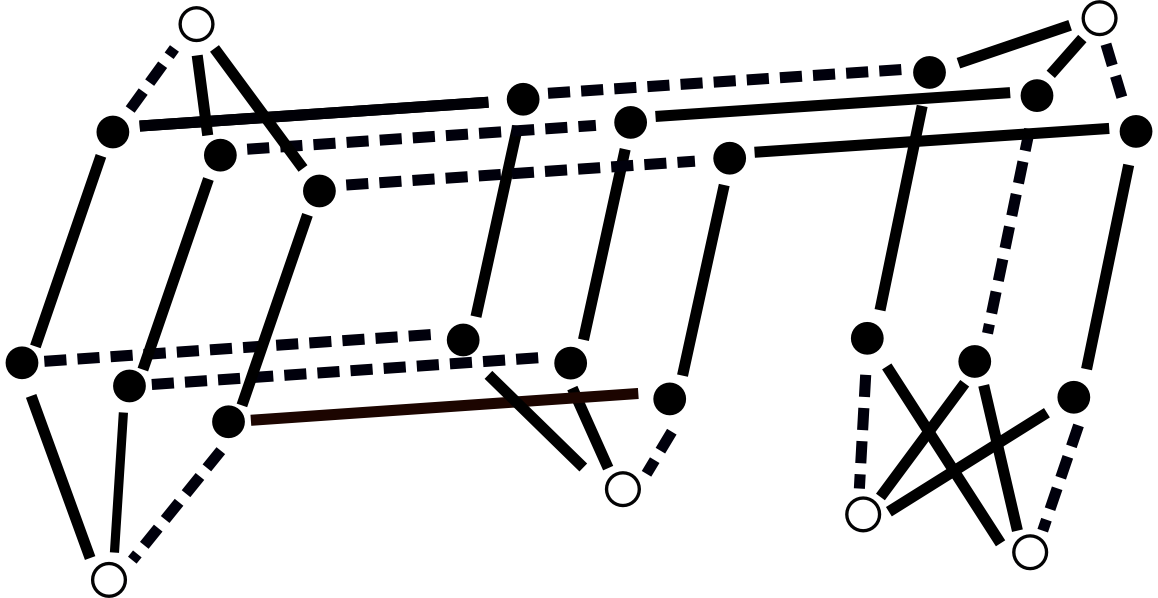}
			\caption{Partners for the vertices of $\hat{\Gamma}$ (indicated by solid lines)}
			\label{fig: marriages in bipartite regular graph}
		\end{subfigure}
		\begin{subfigure}[t]{5cm}
			\includegraphics[scale=0.8]{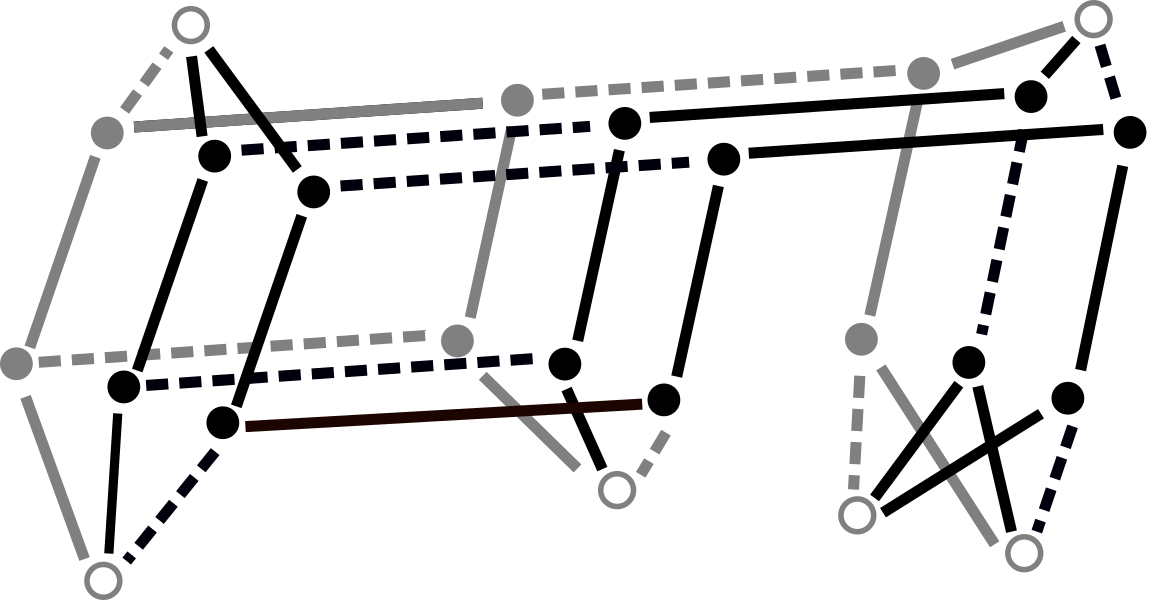}
			\caption{The copy of $\Gamma$ corresponding to the image of $\Gamma \times\{1\}$ gives a partnering for the vertices of $\Gamma$}
			\label{fig: marriage scheme in Gamma}
		\end{subfigure}
		\caption{Finding neighbor partners for $\Gamma$ via Hall's Marriage Theorem}
		\label{fig: pairing vertices}
	\end{figure}
	
	\textbf{III. Finding pairing partners for $b$- and $c$-corridors.} Corollary \ref{lemma: Hall's Marriage Theorem} tells us that $\hat{\Gamma}$ has a 1-factor. This partners each vertex $v \in \hat{\Gamma}$ with an  adjacent vertex  $v^{\prime}$. View the  image of $\Gamma \times \{1\}$ in $\hat{\Gamma}$ as  $\Gamma$, sitting as a subgraph in $\hat{\Gamma}$. In the example of Figure~\ref{fig: pairing vertices}c, $\Gamma$ is the grey subgraph at the back. If $v \in \Gamma$ is a black vertex, its partner $v'$ is also a vertex  of $\Gamma$, but this may fail for white vertices.
	
	\textbf{IV. Completing to a van Kampen diagram.} If $v$ and $v^{\prime}$ are partnered black vertices in $\Gamma$ then the corresponding capping faces are connected by at least one partial $b$-corridor (possibly of length zero). In $\overline{\Delta}$, the capping faces $f$ and $f'$ corresponding to $v$ and $v^{\prime}$ have $|k_a|$ many oppositely oriented charges. We will connect these charges with  $|k_a|$ partial $c$-corridors, as in Figure~\ref{fig:follow b corridor}.  Choose one of the partial $b$-corridors joining $f$ to $f'$ (there is at least one). Run all of the partial $c$-corridors for one capping face alongside the partial $b$-corridor. If a black vertex $v$ is paired with a white vertex $v'$  in $\Gamma$, run all of the partial $c$-corridors alongside the partial $b$-corridor to the boundary. Two white vertices will never be paired. At the ends of partial $b$-corridors on capping faces, it may be necessary to insert rectangles in which the $b$-and $c$-corridors cross, as in Figure \ref{fig: pairing in Gamma gives us discharging pattern for c charges}, but this requires no more than $|\beta||k_a|\Area(\overline{\Delta})$ additional 2-cells. The total number of 2-cells added to $\overline{\Delta}$ in this process is no more than $(|\beta|+1)|k_a|\Area(\overline{\Delta})$.  
	
	\begin{figure}[!ht]
		\centering
		\vcenteredhbox{\includegraphics[scale=0.8]{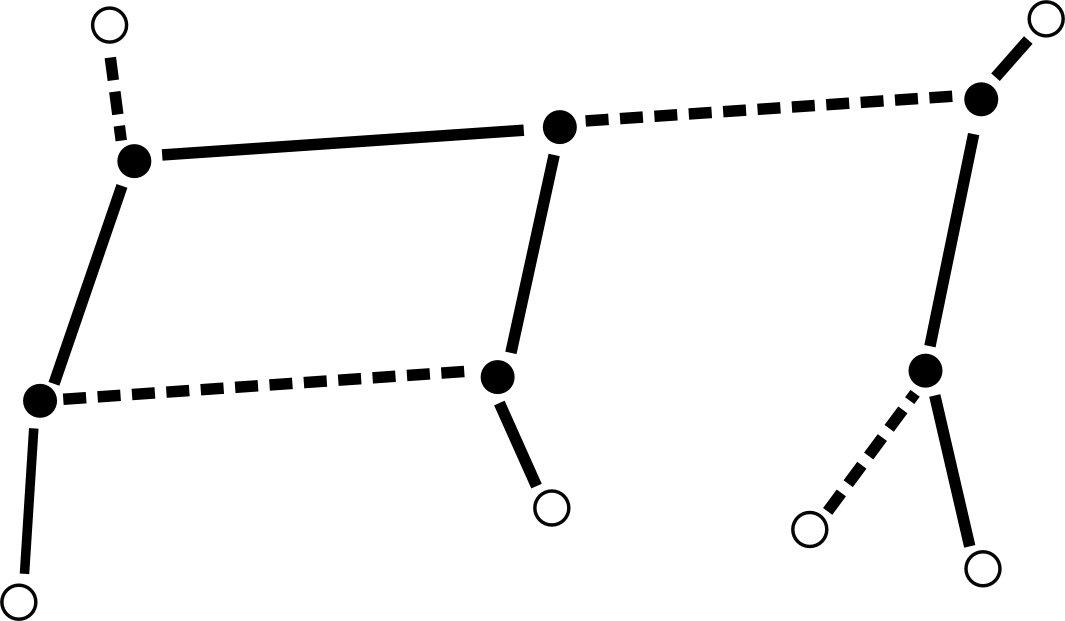}} \hspace{0.5cm} \vcenteredhbox{ \includegraphics[scale=0.6]{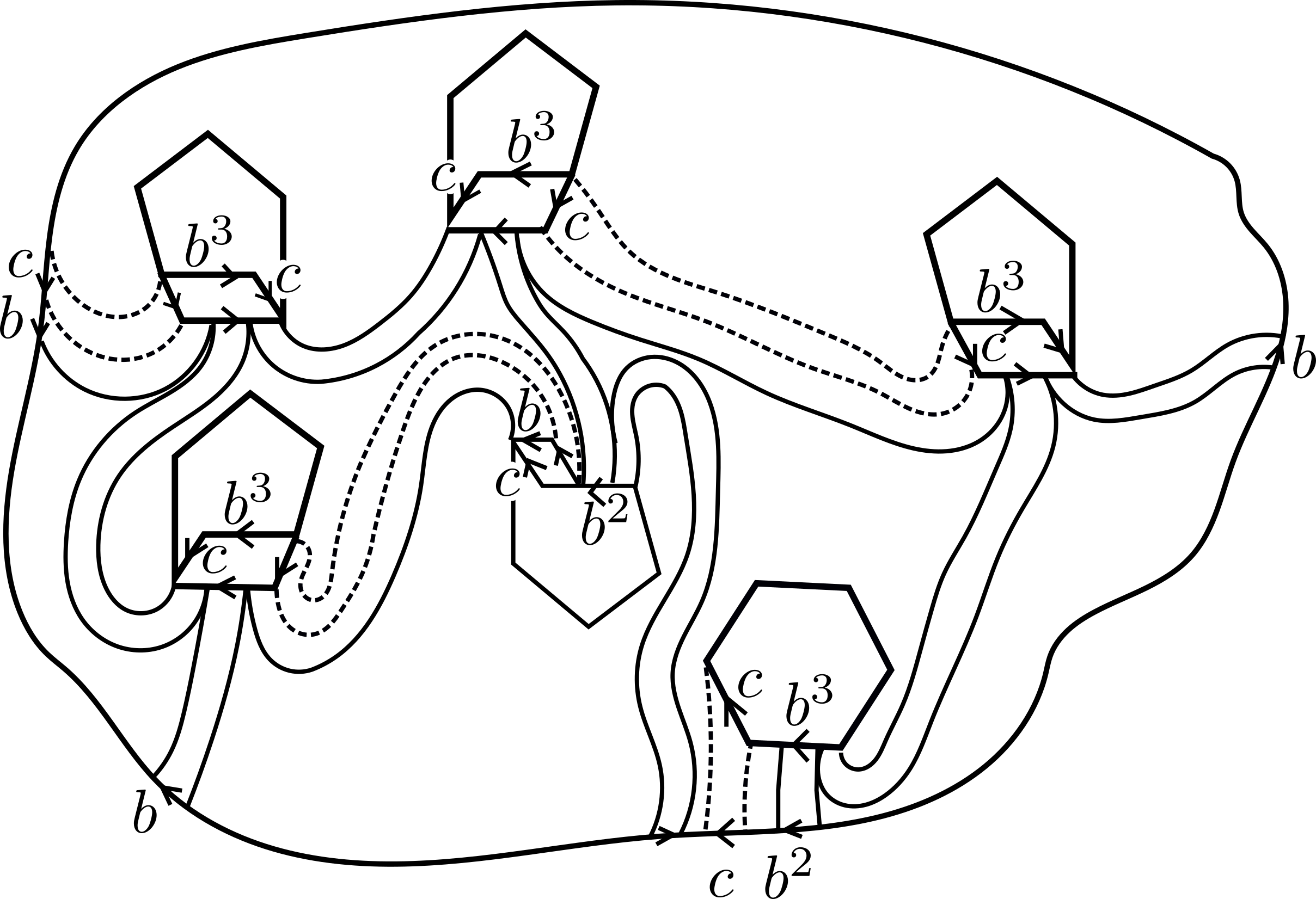}}
		\caption{Partnering in $\Gamma$ gives a consistent way to discharge $c$-charges}
		\label{fig: pairing in Gamma gives us discharging pattern for c charges}
	\end{figure}
	
	\textbf{V. Correcting the boundary.} Partial $c$-corridors follow partial $b$-corridors to the boundary in groups of $|k_a|$. The  new diagram has boundary length between $|\overline{w}|$ and $(|k_a|+1)|\overline{w}|$ and is a van Kampen diagram over $\mathcal{P}$ for some word $w^{\prime}$ in the pre-image of $\overline{w}$. Deleting all $c^{\pm 1}$  from $w^{\prime}$ produces $\overline{w}$, but the arrangement of the $c^{\pm1}$ letters  in $w^{\prime}$ may differ from that in $w$. As was described in Section \ref{sec:electrostatic model} we glue around the outside of this diagram an annular diagram with the word $w^{\prime}$ along the inner boundary component and the word $w$ along the outer boundary component. Together, they form $\Delta$, a van~Kampen diagram for $w$ over $\mathcal{P}$. This annular diagram has area at most $(|k_a|+1)^2|\overline{w}|^2$, and summing our area estimates, $\Delta$ has area no more than ${(1+  (|\beta|+1)|k_a|)\Area(\overline{\Delta}) + (|k_a|+1)^2|\overline{w}|^2}$. Since $\Area(\overline{\Delta}) \leq C|\overline{w}|^2$, it follows that there is constant $A>0$ such that for any given word $w$ in the generators of $\mathcal{P}$ that represents the identity, this construction produces a van~Kampen diagram of area  at most $A|w|^2$. \end{proof}

 \section{\texorpdfstring{Mapping tori of $G = \mathbb{Z}^2 \ast \mathbb{Z}= \langle a, b \mid [a,b] \rangle \ast \langle c \rangle$}{G=Z2 ast Z =<a,b| [a,b] > ast <c>}}
 \label{Z2*Z}
 
 \subsection{Automorphisms of \texorpdfstring{$\mathbb{Z}^2 \ast \mathbb{Z}$}{Z2 ast Z}}\label{sec: autos of Z2astZ}

 Servatius \cite{Servatius} and Laurence \cite{Laurence} found a generating set for the automorphism group of a RAAG $A(\Gamma)$ based on the underlying graph $\Gamma$. In the instance of $$\mathbb{Z}^2 \ast \mathbb{Z}  \ = \ \langle a, b \mid [a,b] \rangle \ast \langle c \rangle,$$  their generating set for the automorphism group consists of  the inner automorphisms, inversions, the one non-trivial graph isomorphism ($a \mapsto b$, $b \mapsto a$, and $c \mapsto c$),  and the four transvections 
 $$\begin{array}{llll}
 \tau_a  : & a \mapsto ab, & b \mapsto b, & c \mapsto c, \\ 
 \tau_b : & a \mapsto a, & b \mapsto ba, & c \mapsto c,  \\ 
 \psi_a : & a \mapsto a, & b \mapsto b,  & c \mapsto ca,   \\  
 \psi_b : & a \mapsto a, & b \mapsto b,  &  c \mapsto cb. 
 \end{array}$$

 The following lemma and then proposition are  steps towards Theorem~\ref{Z2astZ} in that they let us focus on particular presentations for the purposes of classifying Dehn functions of mapping tori of  	$\mathbb{Z}^2 \ast \mathbb{Z}$.  Recall that $\iota_h$ denotes the inner automorphism $x \mapsto h^{-1} x h$.
 
 \begin{lemma} \label{get Phi} 	
 	For all $\Psi \in \Aut( \mathbb{Z}^2 \ast \mathbb{Z})$, there exists $\Phi \in \Aut( \mathbb{Z}^2 \ast \mathbb{Z})$ such that $$ \Phi: \ \ a \mapsto \phi(a), \ \ b \mapsto \phi(b), \ \ c \mapsto wc^{\pm 1}x, $$ where $\phi\in \Aut(\Z^2)$,  $w$ and $x$ are words on $a$ and $b$, and $[\Phi] = [\Psi]$ in $\Out(\Z^2 \ast \Z)$.  Explicitly, suppose that ${\Psi(a) =  u_1c^{\epsilon_1}\dots u_n c^{\epsilon_n}u_{n+1}}$, where $\epsilon_i \neq 0$  and each ${u_i \in \langle a, b\rangle}$ for all $i$, and     $u_2, \ldots, u_n \neq 1$. Then $n=2m$ is even and for ${g := c^{\epsilon_{m+1}}u_{m+2}\dots u_{2m} c^{\epsilon_{2m}}u_{2m+1}}$, we find  $\Phi' :=  \iota_{g^{-1}}\circ \Psi$ is an example of a map satisfying the properties given for $\Phi$.  
 	
 	Moreover, for any $\Phi$ of the given form,  $M_{\Psi}$ and $M_{\Phi}$ have equivalent Dehn functions. 
 \end{lemma}

 \begin{proof}
 	Since $\Inn(\mathbb{Z}^2 \ast \mathbb{Z}) \trianglelefteq \Aut(\mathbb{Z}^2 \ast \mathbb{Z})$, all automorphisms $\Psi \in \Aut(\mathbb{Z}^2 \ast \mathbb{Z})$ can be written as $\iota_h \circ \Phi$ where $\iota_h$ denotes conjugation by some $h \in \Z^2 \ast \Z$ and $\Phi$ is some product of the inversions, transvections, and graph isomorphisms in the generating set above. All inversions, transvections, and graph isomorphisms restrict  to  automorphisms of  the subgroup ${\langle a, b \mid [a,b] \rangle}$, and they all map the subset $\langle a, b \rangle c^{\pm 1} \langle a, b \rangle$ to itself.  So $$\Phi: a \mapsto \phi(a), \  \ b \mapsto \phi(b),  \  \ c \mapsto wc^{\pm 1}x,$$ for some $\phi \in \Aut(\Z^2)$ and  some words $w$ and $x$  on $a^{\pm 1}$ and $b^{\pm 1}$. This proves the existence of  a $\Phi$ with the required properties.   We turn next to how to find such an automorphism.

 	Now suppose $\Psi(a)$ is as per the statement. For   $h \in \Z^2 \ast \Z$ as  above we have that  $\Psi (a) \in \iota_{h}( \langle a, b \rangle)$.  So  $$   \Psi(a)   \ = \    u_1c^{\epsilon_1}\dots u_n c^{\epsilon_n}u_{n+1}  \ \in \   h^{-1}  \langle a, b \rangle h.$$  But, given the free product structure of $\Z^2 \ast \Z$, that  implies that   $n=2m$ is even
 	and  $$h \ = \  v c^{\epsilon_{m+1}}u_{m+2}c^{\epsilon_{m+2}}\cdots u_{2m}c^{\epsilon_{2m}}u_{2m+1} \ = \ vg$$ where $v$ is some element of $\langle a, b \rangle$ and  $g$ is as defined in the statement.   
 	
 	It follows then that $\Phi' := \iota_{g^{-1}} \circ \Psi = \iota_{v} \circ \iota_{h^{-1}} \circ \Psi =   \iota_{v} \circ   \Phi$.  So $\Phi'$ maps $a \mapsto \phi'(a), \  \ b \mapsto \phi'(b),  \  \ c \mapsto w'c^{\pm 1}x'$ for some $\phi' \in \Aut(\Z^2)$ and  some words $w'$ and $x'$  on $a^{\pm 1}$ and $b^{\pm 1}$.  
 	
 	By Lemma \ref{lemma: simplifying automorphisms doesn't change DF}, $M_{\Psi}$ and $M_{\Phi}$ have equivalent Dehn functions.
 \end{proof}
 
 \begin{proposition}\label{groomZ2*Z} Given $\Phi$ as per Lemma~\ref{get Phi}, there exists  $\Xi \in \Aut( \mathbb{Z}^2 \ast \mathbb{Z})$ of the form $$\Xi: \ \ a \mapsto \xi(a), \ \ b \mapsto \xi(b), \ \ c \mapsto cz, $$ where $\xi \in \Aut(\Z^2)$, $z \in \langle a, b \rangle$, and $M_{\Phi}$ and $M_{\Xi}$ have equivalent Dehn functions. Moreover, conditions \emph{(\emph{1}), (\emph{2}), and (\emph{3})} of Theorem~\ref{Z2astZ} apply to $\Xi$ exactly when they apply to $\Phi$. 
 	
 	Additionally,   
 	\begin{itemize}
 		
 		\item when $\xi$  has finite order (Condition~\emph{(\emph{1})} of Theorem~\ref{Z2astZ}),  we may further assume $\xi: a \mapsto a, \ b \mapsto b$, so that
 		$$M_{\Xi} \  = \  \langle a,b,c,t \ \mid \ [a,b]=[a,t]=[b,t]=1,  \ c^t=ca^kb^l \rangle,$$

 		\item when $\xi$ is  of infinite order and has only unit eigenvalues  
 		(Condition~\emph{(\emph{3})} of Theorem~\ref{Z2astZ}), we may further assume $\xi: a \mapsto ab^k, \ b \mapsto b$ for some $k \neq 0$, so that for some $l,m \in \Z$, 
 		$$ M_{\Xi} \  =  \  \langle a, b ,c, t \mid [a,b]=1, \  a^{t}=ab^{k},  \ b^{t}=b,\  c^{t}=c a^lb^m \rangle.$$

 	\end{itemize}
 \end{proposition} 
 
 \begin{proof} 	 
 	
 	Recall $\Phi (c)= wc^{\pm 1}x$ as per Lemma~\ref{get Phi}. Define $\Xi_1 := \iota_{w} \circ \Phi$ and $\Xi_2 := \iota_{\phi(w)x^{-1}}\circ\Phi^2$.   So if $\Phi(c) = wcx$, then $\Xi_1(c) = \iota_w  \circ \Phi(c) =cz$ where $z=xw$.  If  $\Phi(c) = wc^{-1}x$, then $\Xi_2(c) = \iota_{\phi(w)x^{-1}} \circ \Phi^2(c) = cz$ where $z=w^{-1}\phi(xw)x^{-1}$.

 	For $i=1,2$, let  $\xi_i :=  \Xi_i \restricted{\langle a, b \rangle}$, the restriction of $ \Xi_i$  to the $\Z^2$ factor.  
 	
 	Suppose  $\xi_i$  has exponential growth. Define $\Xi :=\Xi_i$. Then $\Xi$ has the general form claimed in the proposition, and since $[\Xi_1] = [\Phi]$ and $[\Xi_2] = [\Phi]^2$, Lemma~\ref{lemma: simplifying automorphisms doesn't change DF} implies that $M_{\Phi}$ and $M_{\Xi}$ have equivalent Dehn functions.  
 	
 	Suppose $\xi_i \in \Aut(\mathbb{Z}^2)$ has finite  order $n$ (i.e.\ $\xi_i$ has trivial growth).   Define $\Xi= \Xi_i^n$. This has the promised form: its restriction to   $\langle a,b \rangle$   is the identity  and $\Xi(c) = cz$ for some $z \in \langle a,b \rangle$. Lemma~\ref{lemma: simplifying automorphisms doesn't change DF} implies that $M_{\Xi}$ and $M_{\Phi}$ have equivalent Dehn functions.

 	Finally, suppose $\xi_i$ is  of infinite order and has only unit eigenvalues. Lemma~\ref{linear growth} implies that $\xi_i$ is conjugate in $\Aut(\Z^2)$ to the automorphism {$\xi: a \mapsto ab^{k}$, $b \mapsto b$,} for some $k \neq 0$. Therefore, for some $f \in \Aut(\Z^2)$, ${\xi} = f^{-1}\circ \xi_i \circ f$. Define $\Xi:= F^{-1}\circ \Xi_i \circ F$, where $F$ restricts to $f$ on $\Z^2$ and maps $c \mapsto c$. By Lemma~\ref{lemma: simplifying automorphisms doesn't change DF}, $M_{\Phi}$ and $M_{\Xi}$ have equivalent Dehn functions. Let $z' = f^{-1}(z)$. The map $\Xi$ has the desired form: $$\Xi(c) = \eta^{-1} \circ \Xi_i \circ \eta (c) = \eta^{-1}(\Xi_i(c)) = \eta^{-1}(cz)= cf^{-1}(z)=cz'.$$

 	In every case,   conditions  (\emph{1}), (\emph{2}), and (\emph{3}) of Theorem~\ref{Z2astZ} apply to $\Xi$ exactly when they apply to $\Phi$.  After all, in each case, the restriction $\xi$ of $\Xi$ to the $\mathbb{Z}^2$ factor  is a conjugate of a power of the restriction  $\phi$ of $\Phi$. Let $A$ be the Jordan Canonical Form (JCF) of $\phi$. For all $k \in \mathbb{N}$, $A$ is finite order if and only if $A^k$ is finite order, and $A$ has a non-unit eigenvalue if and only if $A^k$ has one too. The JCF is invariant under conjugation.
 \end{proof}

 \subsection{Corridors}\label{intersectingcorridors}
 In every instance of Proposition~\ref{groomZ2*Z},   
 $$M_{\Xi} \ = \  \langle a, b, c, t \mid [a,b]=1, \ a^t= \xi(a), \  b^t= \xi(b), \ c^t=cz \rangle$$
 for some $\xi \in \Aut(\Z^2)$ and some $z \in \langle a,b \rangle$.   In this section we prove some preliminary results about  van~Kampen diagrams over this presentation.  Such diagrams can have both $c$- and $t$-corridors. 
 
 \begin{definition} Suppose that $\tau$ is a $t$-corridor and $\eta$ is a $c$-corridor. Suppose ${\hat{\tau}} \subseteq \tau$ and ${\hat{\eta} \subseteq \eta}$ are subcorridors. We say  $\hat{\tau}$ and $\hat{\eta}$ \textit{form a bigon} when they have exactly two common 2-cells, specifically their first and last ones.\end{definition}

 We leave the proof of our next lemma as an exercise---the essential points are (1) neither a  $c$-corridor nor a $t$-corridor self-intersects, and  (2) look at an \emph{innermost} pair of crossings. 
 
 \begin{lemma}\label{lemma:bigon}
 	Suppose $\tau$ is a $t$-corridor and $\eta$ is a   $c$-corridor. If  $\tau$ and $\eta$ intersect more than once, then there are subcorridors $\hat{\tau} \subseteq \tau$ and $\hat{\eta} \subseteq \eta$ forming a bigon. 
 \end{lemma}
 
 Recall that a $c$-corridor is called \textit{reduced} if it contains no two 2-cells sharing a $c$-edge for which the word around the boundary of their union is freely reducible to the identity in the group.

 \begin{lemma}\label{lemma:ct 1x intersections} In a van Kampen diagram where $c$-corridors are reduced, if a $t$-corridor $\tau$ intersects a $c$-corridor $\eta$, it will do so only once.
 \end{lemma}
 
 \begin{proof} Since $c$-corridors are made up of a single kind of 2-cell (arising from the defining relation $c^t=cz$), all 2-cells in a reduced $c$-corridor have the same labels and are oriented the same way along the  corridor.  Let us assume for the contradiction that $\eta$ is reduced and that $\tau$ and $\eta$ intersect at least twice.
 	
 	By Proposition \ref{lemma:bigon}, there exist subcorridors $\hat{\tau}$ and $\hat{\eta}$ that form a bigon, with precisely the first and final 2-cells, $E_1$ and $E_2$, in common, as in Figure~\ref{fig:ctdoublecross}. The orientation of the edges labeled by $t$ in $E_1$ fixes an orientation for all the $t$-labeled 1-cells along the bottom of $\hat{\eta}$ (see Remark~\ref{reduced remark}) since $\hat{\eta}$ is reduced. It also fixes an orientation for $t$-labeled 1-cells in $\hat{\tau}$. But these two specifications are inconsistent   for the $t$-labeled 1-cells in $E_2$. \end{proof}
 
 \begin{figure}[!ht]
 	\centering
 	\includegraphics[scale=0.45]{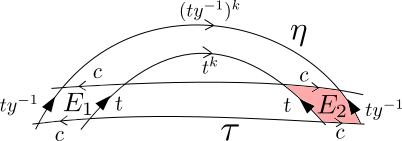}
 	\caption{If a $t$-corridor and a $c$-corridor cross at least twice, the $c$-corridor cannot be reduced.}
 	\label{fig:ctdoublecross}
 \end{figure}
 
 We will use the same argument for alternating corridors and $c$-corridors in Lemma \ref{combinedlemma}(\ref{alternating1}) and for $\alpha$- and $t$- partial corridors in Lemma \ref{combinedlemma}(\ref{alternating3}). 
 
 \begin{corollary} In a van Kampen diagram with reduced $c$-corridors, there are no $c$-annuli, and $t$-annuli do not intersect $c$-corridors.
 \end{corollary}
 \begin{proof}
 	The word around the outside of a $c$-annulus contains $t$'s, so it would have to intersect once (and therefore intersect at least twice) with a $t$-corridor, which is impossible by Lemma \ref{lemma:ct 1x intersections}. Similarly, if a $c$-corridor has common 2-cells with a $t$-annulus, it must have at least two in common---again  impossible by Lemma \ref{lemma:ct 1x intersections}.
 \end{proof}
 
 The following   corollary allows us to determine the lengths of $c$-corridors in a diagram $\Delta$ in terms of  the word around its boundary and the way $c$-edges are paired up by $c$-corridors---the so-called $c$-\textit{corridor pairing} (see Definition \ref{def:pairingpattern}).

 \begin{corollary}\label{c-corridor length} Suppose $\Delta$ is a van~Kampen diagram with reduced $c$-corridors.  Suppose further   that its boundary word is $w_1c^{\pm 1}w_2c^{\mp 1}$ for some words $w_1$ and $w_2$  and that    $\eta$  is a $c$-corridor beginning and ending on the edges labelled by these distinguished $c^{\pm 1}$. Then the length of $\eta$ is the absolute value of the index sum of the $t^{\pm 1}$ in $w_1$ (or, equivalently, in $w_2$).\end{corollary}
 
 \begin{proof} All $t$-corridors intersecting $\eta$ have the same orientation with respect to $\eta$. In particular, the word along one side of $\eta$ is $t^k$ for some $k$, without any free reductions.  Thus the $t$-corridors starting at $t$-edges in $w_2$ that are oppositely oriented to the $t$'s in $\eta$ cannot cross it, and so must have oppositely oriented partners on the same side of $\eta$, as shown in Figure \ref{lengthsOfcorridors}. This leaves exactly the absolute value of the index-sum of $t$ in $w_1$ many $t$-corridors which have no partners on the same side of $\eta$, and so must cross it. By Lemma \ref{lemma:ct 1x intersections}, each of these $t$-corridors can cross $\eta$ exactly once.  
 \end{proof} 
 
 \begin{figure}[!ht]
 	\centering
 	\includegraphics[scale=0.45]{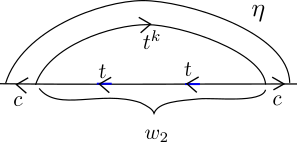}
 	\caption[Lengths of $c$-corridors]{The $t$-corridors of oppositely oriented $t$-edges in $w_2$ cannot cross $\eta$. \label{lengthsOfcorridors}}
 \end{figure}
 
 \begin{figure}[!ht]
 	\centering
 	\includegraphics[scale=0.12]{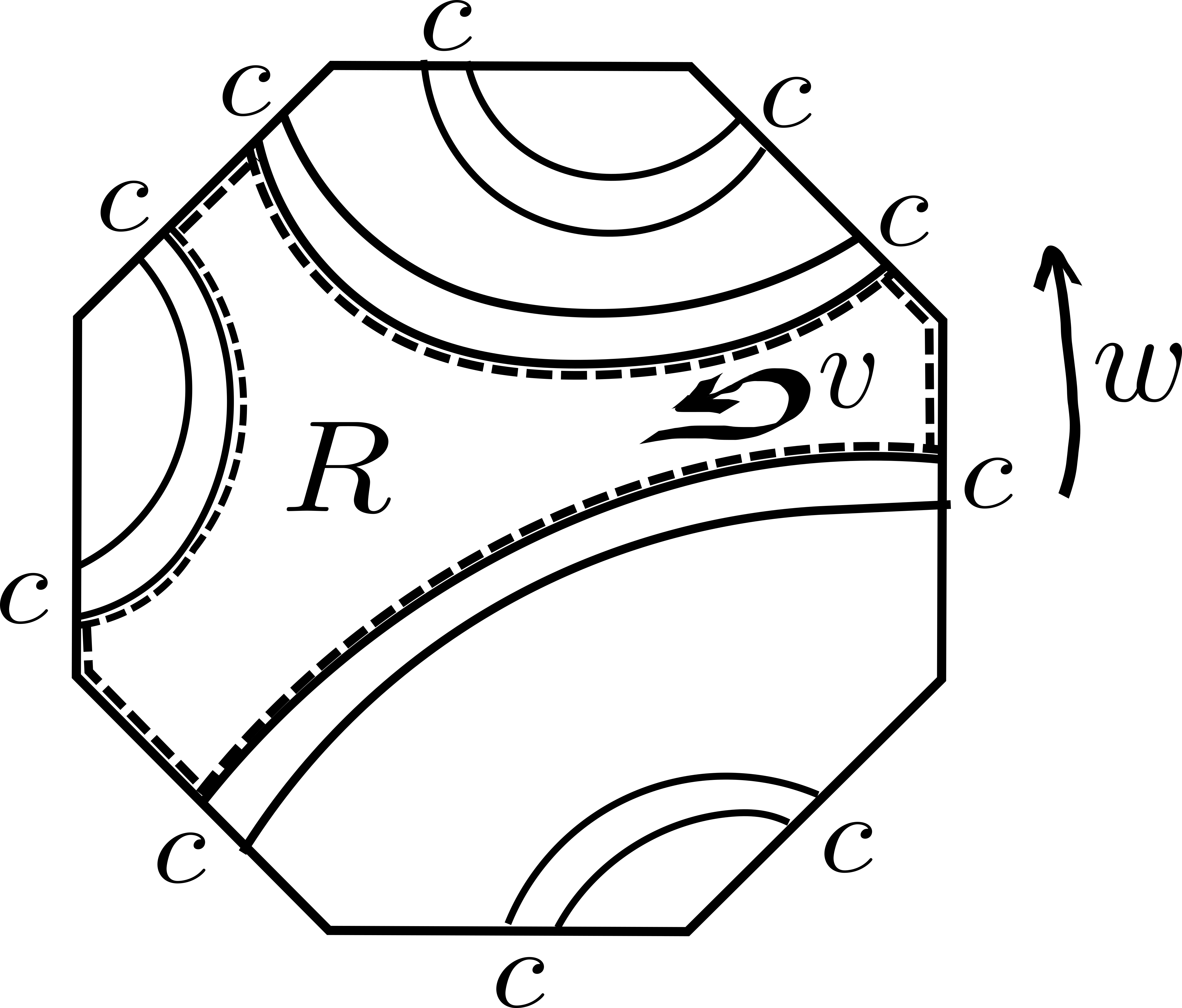}
 	\caption{a c-complementary region}
 	\label{fig:Cor6}
 \end{figure}

 Since $c$-corridors cannot cross, removing all the $c$-corridors leaves a set of connected subdiagrams called \textit{$c$-complementary regions}. The words around the perimeters of each of these regions contain no $c^{\pm 1}$. See Figure \ref{fig:Cor6}.
 
 \begin{corollary}\label{c complementary region length} Let $R$ be a $c$-complementary region in a van Kampen diagram for the word $w$. If the word around the perimeter of $R$ is called $v$, then $|v|\leq|w|$.\end{corollary}
 \begin{proof}
 	Suppose that after cyclic conjugation $w$ has the form $x_0c^{\epsilon_0} v_0c^{-\epsilon_0}x_1c^{\epsilon_1}v_1 c^{-\epsilon_1}\cdots x_nc^{\epsilon_n}v_nc^{-\epsilon_n}$. Then the perimeter of $R$ can be labeled by the word $v=x_0 v'_0x_1v'_1\cdots x_nv'_n$ where $x_0,x_1,\dots, x_n$ are part of $w$ and $v'_0, \dots, v'_n$ label the $c$-corridors, with $v'_i = c^{\epsilon_i}v_ic^{-\epsilon_i}$. By  Corollary \ref{c-corridor length}, $|v'_i| \leq |v_i|$, and so $|v|\leq |w|$.
 \end{proof}	 
 
 \subsection{Alternating corridors}	 
 
 When $$M_{\Xi} \ = \ \langle a, b ,c, t \mid a^{t}=ab^{k},  \ b^{t}=b,\  c^{t}=c a^lb^m, \ [a,b]=1\rangle$$ for some $k, l, m \in \Z$,  killing $b$ maps $M_{\Xi}$ onto   $$Q_l  \ :=  \ \langle a, c, t \mid a^t=a,  \ c^t = ca^l \rangle.$$  
 The elements $b$ and $c$ do not commute in  $M_{\Xi}$, so  $M_{\Xi}$ is  not a central extension of $Q_l$.  Nevertheless, we will use a variant of the electrostatic model to establish upper bounds on area in   $M_{\Xi}$.    The purpose of this section is to  provide necessary  information about van~Kampen diagrams over $Q_l$. We begin with the case $l=1$.   Setting $\alpha := ta^{-1}$, we see that  
 $${Q}'  \ := \  \langle \alpha, c, t \! \mid \! \alpha^t=\alpha, ~ t^c = \alpha \rangle$$ 
 and $Q_1$ are the same group.
 
 \begin{definition}\label{def:alternatingcorridor}
 	A \emph{$c$-face} is a  2-cell in a van~Kampen diagram $\Delta$ over ${Q}'$ corresponding to the defining relation $t^c = \alpha$.   Partial $\alpha$- and $t$-corridors in  $\Delta$ fit together in an alternating way: where a partial $\alpha$-corridor ends at a $c$-face in the interior of a diagram, a partial $t$-corridor   begins, and where this ends, another partial $\alpha$-corridor begins.  An \textit{alternating corridor} in $\Delta$ is a maximal union of $\alpha$-partial corridors, $t$-partial corridors and the $c$-faces  between them, fitting together in this way---see Figure \ref{fig:alternatingcorridor}. \end{definition} 
 
 Like  a standard corridor, an alternating corridor either closes up on itself or it connects two boundary edges. It is possible for alternating corridors to self-intersect, but, as we will see shortly, in a reduced diagram,  alternating corridors do not self-intersect or form annuli. Every face in $\Delta$ is part of some alternating corridor. Like standard and partial corridors, an  alternating corridor has a top and a bottom: the internal $\alpha$- and $t$-edges are directed from the bottom to the top (again, see the figure).
 
 \begin{center}
 	\begin{figure}[!ht]
 		\includegraphics[scale=.3]{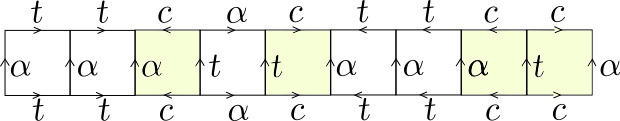}
 		\caption{An alternating corridor} \label{fig:alternatingcorridor}
 	\end{figure}
 \end{center}

 \begin{lemma}\label{combinedlemma} Suppose $\Delta$ is a van Kampen diagram over $\langle \alpha, c, t \! \mid \! \alpha^t=\alpha,~ t^c = \alpha \rangle$ in which all $c$-corridors   and all $\alpha$- and $t$-partial corridors are reduced. (See Figures \ref{fig:ok non-reduced} and \ref{notallowed nonreductions}.) Then in $\Delta$:
 	\begin{enumerate} 
 		\item \label{alternating1} A $c$-corridor $\eta$ and an alternating corridor $\tau$ can cross at most once.
 		\item \label{annuli} Alternating corridors do not form annuli.
 		\item \label{alternating2} A single alternating corridor can never cross itself.
 		\item \label{alternating3} Two alternating corridors cannot cross more than once.
 	\end{enumerate}
 \end{lemma}
 
 \begin{center}
 	\begin{figure}[!ht]
 		\centering
 		\includegraphics[scale=0.35]{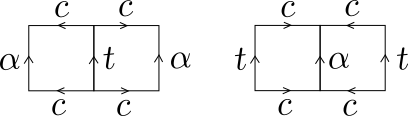}
 		\caption{Non-reduced subdiagrams that can occur in $\Delta$} \label{fig:ok non-reduced}
 	\end{figure} 
 	\begin{figure}[!ht]
 		\centering
 		
 		\begin{subfigure}[t]{4.5cm}
 			\centering
 			\includegraphics[scale=0.35]{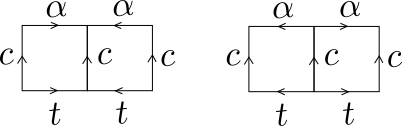}
 			\caption{$c$-corridors are reduced in $\Delta$}
 			\label{fig:reduced c-corridors}
 		\end{subfigure} \hspace{.8cm} \begin{subfigure}[t]{7cm}
 		\centering
 		\includegraphics[scale=0.35]{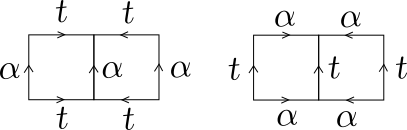}
 		\caption{$\alpha$- and $t$- partial corridors are reduced in $\Delta$}
 		\label{fig: t and alpha reduced}
 	\end{subfigure}
 	\caption{Non-reduced subdiagrams \emph{not} occurring in $\Delta$} \label{notallowed nonreductions}
 \end{figure}
\end{center}

\begin{proof}
	For (\ref{alternating1}) it suffices (see Lemma \ref{lemma:bigon}) to prove that it is impossible to have a bigon of an alternating corridor $\tau$ and a $c$-corridor $\eta$ in $\Delta$. Since $c$-corridors in $\Delta$ are reduced, the top of the $c$-corridor is labeled by a power of $\alpha$ without any free reduction. As in our proof of Lemma \ref{lemma:ct 1x intersections}, $\tau$ and $\eta$ specify inconsistent orientations for the $t$ edge in the second common 2-cell, as in Figure \ref{fig:cdoublecross}.
	
	For (\ref{annuli}), suppose for a contradiction, that there is such an annulus. It cannot contain any $c$-faces, as this would force a $c$-corridor to cross the alternating annulus twice. If our annulus contains no $c$-faces, then it is either a $t$- or $\alpha$-annulus. The word along the top of the annulus is a power of $\alpha$ or $t$, respectively. Such an annulus would imply that $t$ or $\alpha$ have finite order, but both are infinite order elements of $N_{1}$. 
	
	\begin{figure}[!ht]
		\centering
		\begin{subfigure}[t]{5cm}
			\includegraphics[scale=0.5]{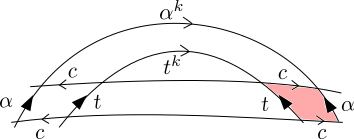}
			\centering
			\caption{$c$-corridors cannot cross alternating corridors more than once}
			\label{fig:cdoublecross}
		\end{subfigure}\hspace{1cm}
		\begin{subfigure}[t]{5.5cm}
			\centering
			\includegraphics[scale=0.5]{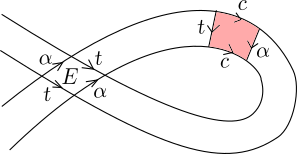} 
			\caption{Alternating corridors do not self-intersect.}
			\label{fig:self-intersection}
		\end{subfigure}
		
		\vspace{0.5cm}
		
		\begin{subfigure}[b]{5cm}
			\centering
			\includegraphics[scale=0.5]{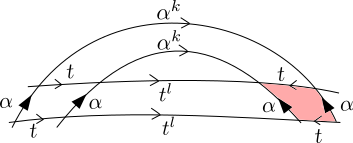}
			\caption{$\alpha$- and $t$- partial corridors cannot cross more than once.}
			\label{fig:tadoublecross}
		\end{subfigure}
		\hspace{1cm}
		\begin{subfigure}[b]{5.5cm}
			\centering
			\includegraphics[scale=0.5]{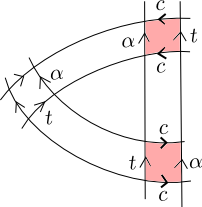}
			\caption{Distinct alternating corridors cannot cross more than once.}
			\label{fig:c-triangle}
		\end{subfigure}
		\caption{Impossible behavior for alternating corridors.}\label{fig:alternatingcorridors}
	\end{figure}

	For (\ref{alternating2}), suppose for a contradiction that an alternating corridor  $\eta$ has a self-intersection. An alternating corridor can only have a self-intersection at a 2-cell corresponding to the relation $[\alpha, t]=1$. Let $\hat{\eta} \subset \eta$ be a subcorridor of $\eta$ that begins and ends at the self-intersection. Call this first and final 2-cell $E$. 
	
	The 2-cell $E$ is part of both $t$- and $\alpha$- partial corridors in $\hat{\eta}$; therefore $\hat{\eta}$ contains at least one $c$-face (in particular, an odd number of $c$-faces in order to get both an $\alpha$- and $t$-segment at the intersection). Each $c$-face in $\hat{\eta}$ is part of a $c$-corridor. By (\ref{alternating1}), $c$-corridors can only cross $\hat{\eta}$ once, but each $c$-corridor must cross $\hat{\eta}$ at least twice, since $\hat{\eta}$ is an annulus. 
	
	For (\ref{alternating3}), assume for the contradiction that two alternating corridors cross at least twice. Again, we can find a bigon of alternating corridors. There are two cases. In one, no $c$-corridors intersect the bigon. In this case, one of the alternating corridors is a partial $t$-corridor, and the other is a partial $\alpha$-corridor. An argument like Lemma~\ref{lemma:ct 1x intersections} shows that this kind of double intersection is impossible when $t$- and $\alpha$- partial corridors are reduced (see Figure~\ref{fig:tadoublecross}). In the other case, at least one $c$-corridor intersects the bigon. We look at the triangle formed by the two bigons and the first $c$-corridor to cross them. Since it is the first such $c$-corridor, we have an $\alpha$- and $t$-partial corridor that both need to end on the same side of a $c$-corridor. However, $c$-corridors always have $t$'s along the bottom and $\alpha$'s along the top --- there cannot be both $\alpha$'s and $t$'s on the same side of the $c$-corridor. Figure \ref{fig:c-triangle} illustrates this contradiction. Therefore neither case happens.\end{proof}

\subsection{Quadratic area diagrams over \texorpdfstring{$Q_1$}{Q1}} 

\begin{definition} \label{def:pairingpattern} 
	A $c$-\textit{pairing} for a word $w$  is any pairing off of the $c$ in $w$ with the $c^{-1}$   in $w$.

	If $w$ represents the identity in $M_{\Xi}$, then a van~Kampen diagram $\Delta$ for $w$ induces a $c$-pairing: a $c$ and a $c^{-1}$ are paired when they are joined by a $c$-corridor in $\Delta$. We say that a $c$-pairing is   \textit{valid} if it is induced by a van Kampen diagram for $w$. (Not all $c$-pairings of a word  need be valid. Valid $c$-pairings need not be unique.)
\end{definition}

The Dehn function of $${Q}  \ := \  {Q}_1   \ = \  \langle a, c, t \mid a^t=a,  \ c^t = ca \rangle$$ grows at most quadratically, as it is a free-by-cyclic group.  The point of the following lemma is that this quadratic area bound can be realized on diagrams witnessing any prescribed valid $c$-pairing.

\begin{lemma}\label{lemma:definingtheta1} There exists $A >0 $ such that for any word $u$ representing the identity in ${Q}$ (not necessarily freely reduced), and for any valid $c$-pairing $P$ of $u$,  there is a van~Kampen diagram $\Theta$ for $u$ over ${Q}$  that induces  $P$, has $\Area(\Theta) \leq A |u|^2$, and has reduced $c$-corridors.  
\end{lemma}

\begin{proof}
	Let $\Delta$ be a van~Kampen diagram for $u$ over ${Q}$ that realizes the given $c$-pairing.
	
	Instead of ${Q}$ we will work with 
	$${Q}'  \ := \  \langle \alpha, c, t \! \mid \! \alpha^t=\alpha, ~ t^c = \alpha \rangle,$$ which,  recall, we can see   presents the same group by setting $\alpha := ta^{-1}$. 
	
	Two finite presentations $\langle A_1 \mid R_1 \rangle$ and $\langle A_2 \mid R_2 \rangle$ of the same group have $\simeq$-equivalent Dehn functions \cite{Alonso, Gersten}. In outline, the  proofs   in  \cite{Alonso, Gersten} go as follows.  For each $a \in  A_1$, pick  a word $u_a=u_a(A_2)$ representing the same group element.  Suppose a word $w_1 = w_1(A_1)$  represents $1$ in  $\langle A_1 \mid R_1 \rangle$.  Let $w_2$  be the word obtained from $w_1$ by replacing all of its letters  $a^{\pm 1}$ by ${u_a}^{\pm 1}$.   A van~Kampen diagram $w_1$ over  $\langle A_1 \mid R_1 \rangle$  can be  converted  to a  van~Kampen diagram for $w_2$  over $\langle A_2 \mid R_2 \rangle$  of comparable area by converting each edge labeled $a$ to a path labeled $u_a$ and then filling all the faces. Each relator in $R_1$ can be rewritten as a word representing the identity in $A_2$, and each can then be filled with at most some constant number of  relators in $R_2$, so the area of the diagram over $\langle A_2 \mid R_2 \rangle$ will be no more than a constant multiple of the area of the diagram over $\langle A_1 \mid R_1 \rangle$.

	In the instance of  ${Q}$ and ${Q}'$, the $c$-pairings induced by the two diagrams agree, and so  it suffices to prove the lemma for ${Q}'$ instead of $Q$. 
	
	Given $u=u(a,c,t)$, let $u'$ be the word obtained from  $u(\alpha^{-1} t, c,t)$ by cancelling away all $\alpha^{\pm 1} \alpha^{\mp 1}$ and all $t^{\pm 1} t^{\mp 1}$ (but not all $c^{\pm 1} c^{\mp 1}$).   Then $|u^{\prime}|\leq 2|u|$. 
	Construct a van~Kampen diagram $\Theta'$ for $u'$ over ${Q}'$ as follows.  Begin with a planar polygon with edges directed and labeled so that one reads $u'$ around the perimeter.  Insert reduced $c$-corridors of 2-cells  (each with perimeter $t^c \alpha^{-1}$) mimicking the pattern of $c$-corridors in $\Delta$.  Fill the complementary regions with minimal area sub-diagrams over $\langle \alpha, t \mid \alpha^t =\alpha \rangle$. The words around their perimeters represent the identity in  $\langle \alpha, t \mid \alpha^t =\alpha \rangle$ because the words around the corresponding loops in  $\Delta$ represent the identity in $\langle a, t \mid  a^t = a \rangle$. Since the complementary regions are filled with minimal area subdiagrams, all $\alpha$- and $t$- partial-corridors in $\Theta'$ are reduced.
	
	Lemma~\ref{combinedlemma} implies that the length of any alternating corridor $\mathcal{A}$ in our diagram is bounded above by the total number of $c$-corridors and alternating-corridors that intersect $\mathcal{A}$. Since there are in total no more than $|u'|/2$ $c$-corridors and alternating corridors, the length of $\mathcal{A}$ is at most $|u'|/2$. Similarly, the length of each $c$-corridor is at most $|u|/2$  by Lemma \ref{c-corridor length}, and there are fewer than $|u|/2$ many $c$-corridors. So altogether, $$\Area(\Theta^{\prime}) \ \leq \  \frac{|u'|^2 + |u|^2}{4}  \ \leq \  2|u|^2.$$
\end{proof}

\subsection{Quadratic area diagrams over \texorpdfstring{$Q_l$}{Ql}} 
In the previous section we established that given a valid $c$-pairing for a word representing the identity in $Q_1$, we can construct a quadratic area van Kampen diagram with that $c$-pairing. In this section, we leverage Lemma~\ref{lemma:definingtheta1} to the case where we have a valid $c$-pairing for a word representing the identity in $Q_l$. Our main strategy is to rewrite words representing the identity in $Q_l$ to words in $Q_1$, where we can apply Lemma~\ref{lemma:definingtheta1} to build a van~Kampen diagram. Then we convert it to a diagram over $Q_l$. 

Recall that  $$Q_l  \ :=  \ \langle a, c, t \mid a^t=a,  \ c^t = ca^l \rangle.$$ 

Define $${Q}^{\tau}_{1}  \ :=  \ \langle a, c, \tau | a^{\tau}=a, c^{\tau}=ca\rangle.$$  Identifying $t$ with $\tau^l$  gives an isomorphism of $Q_{l}$ with the index $l$ subgroup of $Q^{\tau}_{1}$ generated by $a,c,$ and $\tau^l$.

\begin{proposition}\label{lemma:wordconvertingZ2astZ} If  $u$ is a (not necessarily freely reduced) word representing the identity in ${Q}_l$  and $P$ is a  valid $c$-pairing  of $u$, there exists a van~Kampen diagram for the corresponding word $v:=u(a,c,\tau^l)$ in $Q_1^{\tau}$ with a corresponding $c$-pairing.    
\end{proposition}

\begin{proof} 
	Suppose $u(a,c,t)$ represents the identity in  ${Q}_l$ and $\Theta_0$ is a van~Kampen diagram over ${Q}_l$ for $u$ inducing the $c$-pairing $P$.  
	Define $v:=u(a,c,\tau^l)$---that is, obtain $v$ by substituting a $(\tau^l)^{\pm 1}$ for every $t^{\pm 1}$ in $u$. Then $v$ represents the identity in ${Q}_{1}^{\tau} =  \ \langle a, c, \tau | a^{\tau}=a, c^{\tau}=ca\rangle$ and  $P$  induces a  valid $c$-pairing for $v$ (which we will also call $P$) since $\Theta_0$ can be converted to a van~Kampen diagram for $v$ over ${Q}_{1}^{\tau}$ with the same pattern of $c$-corridors as follows. First replace each $t$-edge  in $\Theta_0$ by a concatenation of $l$ $\tau$-edges. The resulting diagram has 2-cells of two types---those originating from the relation $a^t=a$ and those from the relation $c^t=ca^l$.  The perimeter words of these 2-cells become $a^{\tau^l}a^{-1}$  and $c^{\tau^l}(ca^{l})^{-1}$.  These words  are relators in ${Q}_{1}^{\tau}$: the first can be derived by $l$ applications of $a^{\tau}=a$ and the second by $l$ applications of $c^{\tau}=ca$ and $l(l-1)/2$ applications of $a^{\tau}=a$.  Accordingly, refine the diagram by replacing the  $a^{\tau^l}a^{-1}$ 2-cells with an $a$-corridor of $l$ 2-cells each labeled  $a^{\tau} a^{-1}$, and the  $c^{\tau^l}(ca^{l})^{-1}$ 2-cells   with a  $c$-corridor of $l$ 2-cells labeled  $c^{\tau}(ca)^{-1}$ together with $l(l-1)/2$ of the $a^{\tau} a^{-1}$  2-cells. The substitutions in the case $l=3$ are shown in Figure \ref{fig:replacement_Prop615}.  This process maintains the $c$-pairing during the change from $\mathcal{Q}_l$ to $\mathcal{Q}_1^{\tau}$.
	\begin{figure}[!ht]
		\centering
		\includegraphics[scale=0.2]{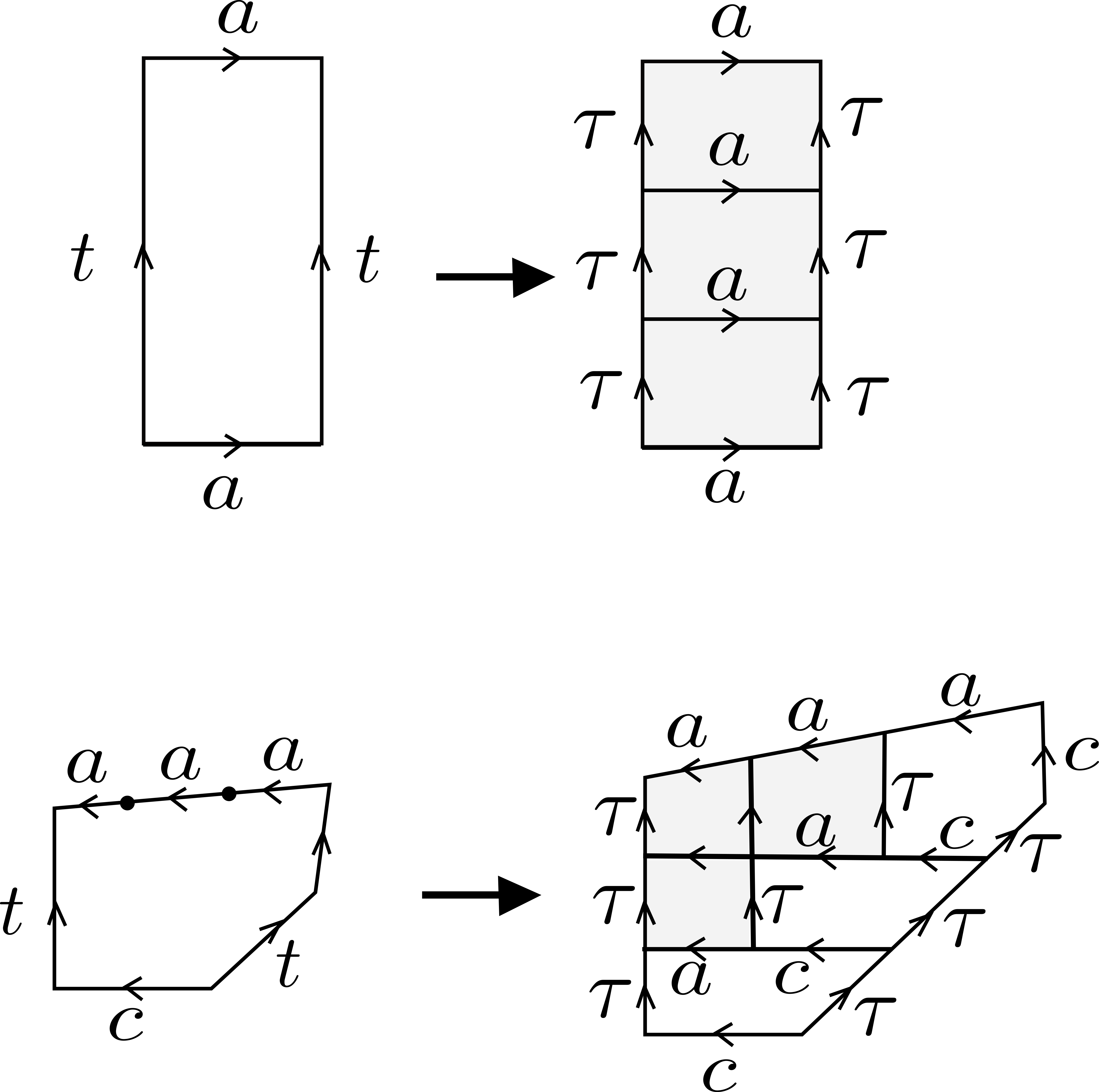}
		\caption{Converting $\Theta_0$ from $Q_l$ to $Q_1^{\tau}$ (illustrated with $l=3$) }
		\label{fig:replacement_Prop615}
	\end{figure}
\end{proof}

After producing a quadratic area van~Kampen diagram for $v$ in $Q_1^{\tau}$ that has $c$-pairing $P$, we want to use it to build a quadratic area van~Kampen diagram for $u$ in $Q_l$ that also has $c$-pairing $P$. The following lemma tells us that we will be able to replace $c$-corridors over $Q_1^{\tau}$ with $c$-corridors over $Q_l$, as they always occur in multiples of $l$.

\begin{lemma}{\label{Length of corridors}}
	Suppose $\Theta$ is van~Kampen diagram for a word $v = v(a, c, \tau^l)$ over 
	${{Q}^{\tau}_{1}   =    \langle a, c, \tau | a^{\tau}=a, c^{\tau}=ca\rangle}$, with reduced $c$-corridors.  Then every $c$-corridor in $\Theta$  has  length a multiple of $l$.
\end{lemma}	
\begin{proof}
	Let $\mathcal{T}$   be the  tree dual to  the $c$-corridors in $\Theta$---that is,  $\mathcal{T}$ has a  vertex dual  to each  $c$-complementary region and an edge dual to each   $c$-corridor; the leaves of $\mathcal{T}$ correspond to regions which have one single $c$-corridor in their perimeter.  (See Section \ref{sec: vK diagrams,corridors, DF}.) Pick any leaf $r$ of $\mathcal{T}$ to serve as the \emph{root}.  There is a bijection between vertices $v \neq r$ of $\mathcal{T}$ and $c$-corridors $C_v$: take $C_v$ to be  dual to the first edge of the geodesic in $\mathcal{T}$ from $v$ to $r$.
	
	We will show by reverse induction on distance in $\mathcal{T}$ from $v$ to $r$ (i.e.\ starting from the leaves and working towards  $r$), that the length of $C_v$ is a multiple  of $l$.  Indeed  when $v$ is a leaf,   the length of $C_v$ is  the index-sum of the $\tau^{\pm 1}$ in the boundary between the paired $c$-edges and $\tau$ only appears in multiples of $l$ in $v$, so the result holds.  For the induction step, suppose  $v \neq r$.  The length of $C_v$ is the exponent sum of the lengths of $C_{v'}$ (with appropriate signs) over every parent $v'$ of $v$ (each a multiple of $l$, by induction hypothesis) and of the  $\tau^l$ in the boundary of  $\Theta_1$  that are also  in the boundary of the subdiagram dual to $v$. 
\end{proof}

Next we examine how to build a filling for a $c$-complementary region over $Q_{l}$ from a filling for a $c$-complementary region over $Q_{1}^{\tau}$ when their boundaries are compatible.

\begin{lemma}\label{shufflingsequence} Suppose $w = w(a, \tau^l)$ has  a van Kampen diagram $\mathcal{D}$ over ${\langle a, \tau \mid a^\tau = a \rangle}$ of  area $A$.  Then $w(a, t)$ has a van Kampen  diagram $\mathcal{D}''$ over ${\langle a, t \mid a^t=a \rangle}$ of area at most $A$.
\end{lemma}

\begin{proof}
	Define a \textit{$\tau$-segment} to be $l$ consecutive $\tau$-labeled edges in the  boundary circuit of $\mathcal{D}$. Such segments have a natural orientation that agrees with the orientation of the constituent $\tau$. We will find a van~Kampen diagram for $w$ over  ${\langle a, \tau \mid a^\tau = a \rangle}$ for which the  $\tau$-segments are connected by blocks of parallel $\tau$-corridors.  (Call this a  $\tau^l$-pairing.) 
	
	The first edge in any $\tau$-segment can only be paired by a $\tau$-corridor in $\mathcal{D}$ with the first edge of an oppositely oriented $\tau$-segment.   Indeed, suppose that an initial $\tau$ in a $\tau$-segment is connected by a corridor $C$  to a $\tau$ in position $i$ on another segment, with $1\leq i \leq l$. Let $\widehat{w}$ be the subword of $w$ between them, as in Figure~\ref{fig: pairing in tau-segments}. Because $\tau$ corridors do not cross, the $\tau$-index sum of $\widehat{w}$ must be zero.  If $i \neq 1$, the $\tau$-index sum of $\widehat{w}$ will not be a multiple of $l$, as $\widehat{w}$ either includes an entire $\tau$-segment or entirely misses it, except for the partial segment which contains the $\tau$ in position $i$. In particular, the index-sum of $\tau$ in $\widehat{w}$ can only be 0 when $i=1$. 
	
	\begin{figure}[!ht]
		\centering
		\begin{subfigure}[t]{7cm}
			\centering
			\includegraphics[scale=0.53]{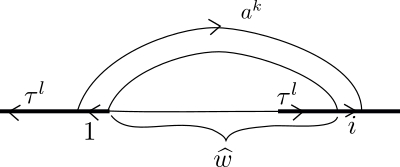}
			\caption{A $\tau$ in position 1 in a $\tau$-segment can only pair with another initial $\tau$}
			\label{fig: pairing in tau-segments}
		\end{subfigure}\hspace{.5cm}
		\begin{subfigure}[t]{7cm}
			\centering
			\includegraphics[scale=0.55]{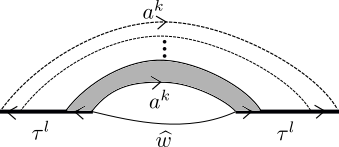} 
			\caption{Initial $\tau$ corridors provide a guide for joining the rest of the $\tau$-segment when building $D'$}
			\label{fig: tau-segment pairing}
		\end{subfigure}\hspace{.5cm}
		\caption{In a word $w$ on $a$ and $\tau^l$, there is a valid $\tau$-pairing that pairs whole $\tau$-segments. }\label{fig: tau-segments}
	\end{figure}
	
	Construct a new van~Kampen diagram $\mathcal{D}'$ for $w$ over  $\langle a, \tau \mid a^\tau = a \rangle$  as follows. Begin with a planar loop with edges labeled so that we read  $w(a, \tau^l)$ around the perimeter. Add in all initial $\tau$ corridors from $D$. If an initial $\tau$-corridor $C$ connects $\tau$-segments $S$ and $S'$, we will pair each $\tau$ in $S$ to the corresponding $\tau$ in $S'$ using copies of $C$, as in Figures~\ref{fig: tau-segment pairing} and \ref{fig:fillingex3}.  The remaining regions that have to be filled have perimeters labeled by words on $a^{\pm 1}$ alone, as all $\tau$ edges have been paired. Moreover, the index-sum of $a$ is zero, so these can be folded together to complete the construction of $\mathcal{D}'$ without the addition of any further 2-cells.     
	
	The area of $\mathcal{D}'$ will be $l$ times the sum of the initial $\tau$-corridor contributions, and so in particular, the area of the new diagram is at most $l A$. Let   $\mathcal{D}''$  be the  van~Kampen diagram for $w(a,t)$  over  $\langle a, t \mid a^t = a \rangle$ of area at most $A$  obtained by replacing each stack of $l$   $\tau$-corridors in $\mathcal{D}'$ by a single  $t$-corridor and each $\tau$-segment in the boundary by a single $t$-edge, as in Figure~\ref{fig:fillingex4}.
\end{proof}

\begin{center}
	\begin{figure}[!ht]
		\centering
		\begin{subfigure}[t]{3.5cm}
			\centering
			\includegraphics[scale=0.24]{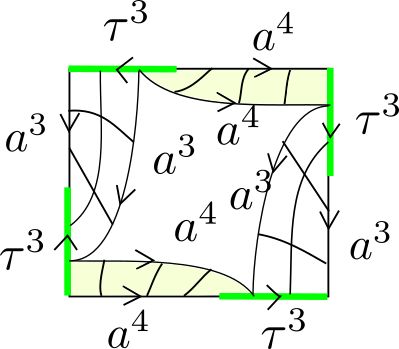}
			\caption{Initial filling $\mathcal{D}$}
			\label{fig:fillingex1}
		\end{subfigure}\hspace{.5cm}
		\begin{subfigure}[t]{3.5cm}
			\centering
			\includegraphics[scale=0.24]{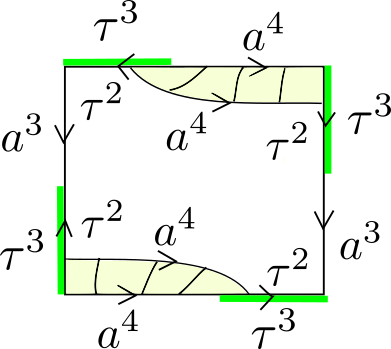} 
			\caption{Refilling all initial $\tau$-corridors}
			\label{fig:fillingex2}
		\end{subfigure}\hspace{.5cm}
		\begin{subfigure}[t]{3.5cm}
			\centering
			\includegraphics[scale=0.24]{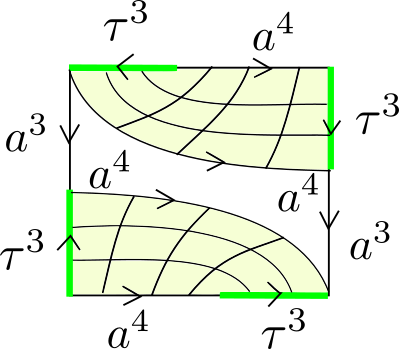}
			\caption{Filling with stacks of $\tau$-corridors to get $\mathcal{D}'$}
			\label{fig:fillingex3}
		\end{subfigure}
		\hspace{.5cm}
		\begin{subfigure}[t]{3.5cm}
			\centering
			\includegraphics[scale=0.24]{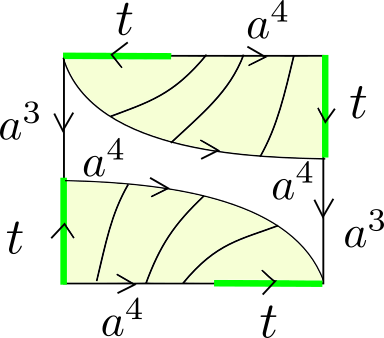}
			\caption{Grouping $\tau^l$'s to get $\mathcal{D}''$}
			\label{fig:fillingex4}
		\end{subfigure}
		\caption{A toy example of the procedure of Lemma \ref{shufflingsequence}}\label{fig:filling}
	\end{figure}
\end{center}

We will promote Lemma~\ref{lemma:definingtheta1}  to the following   result concerning ${Q_l = \langle a, c, t~| a^t=a, c^t=ca^l \rangle}.$

\begin{proposition}\label{lemma:generalcaseZ2*Zv2} There exists $A_l >0 $ such that if  $u$ is a (not necessarily freely reduced) word representing the identity in ${Q}_l$  and $P$ is a  valid $c$-pairing  of $u$,  then there exists a van~Kampen diagram $\Theta$ for $u$ over ${Q}_l$  which has reduced $c$-corridors,  induces  $P$, and has $\Area(\Theta) \leq A_l |u|^2$.  
\end{proposition}

\begin{proof} Suppose $u(a,c,t)$ represents the identity in  ${Q}_l$ and $u$ has a valid $c$-pairing $P$. Lemma \ref{lemma:wordconvertingZ2astZ} implies that $P$ is also a valid $c$-pairing for the corresponding word $v:=u(a,c,\tau^l)$ in $Q_1^{\tau}$, which we get by substituting a $(\tau^l)^{\pm 1}$ for every $t^{\pm 1}$ in $u$. 
	Now we can use what we know about building diagrams over $Q_1^{\tau}$: by Lemma~\ref{lemma:definingtheta1}, there is a constant $A_1>0$ such that $v$ admits a new van~Kampen diagram $\Theta_{1}$ over ${Q}_{1}^{\tau}$ that induces $P$ and has area at most $A_1{|{v}|^2} \leq A_1 l^2{|u|}^2$. Guided by $\Theta_{1}$, we will construct a van~Kampen diagram $\Theta_{l}$ for $u$ over ${Q}_{l}$ which has comparable area.

	\begin{figure}[!ht]
		\centering
		\includegraphics[scale=0.2]{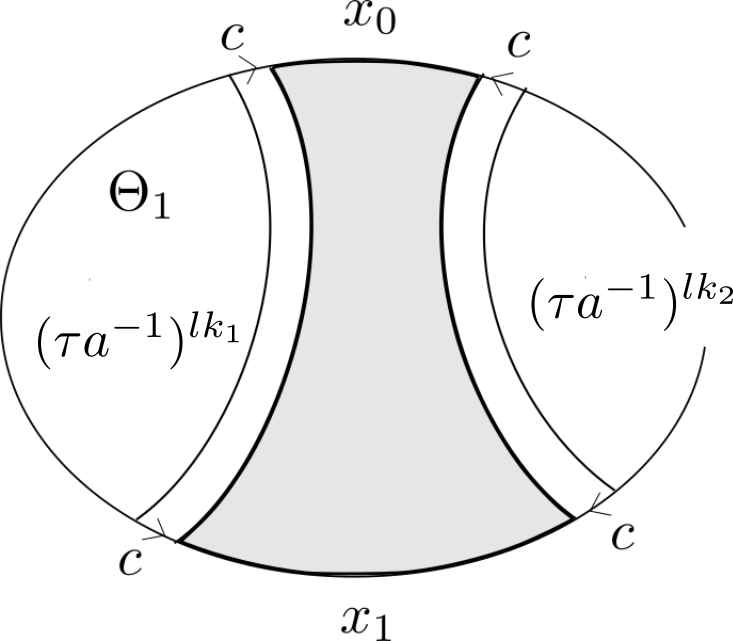} 
		\label{fig:theta1}
		\caption{The quadratic area diagram $\Theta_1$}
	\end{figure}
	By Lemma \ref{Length of corridors}, $c$-corridors in $\Theta_1$ all have length that is a multiple of $l$. To build $\Theta_{l}$, we begin by inserting reduced $c$-corridors into a polygonal path labeled by $u$, mimicking the $c$-corridors in  $\Theta_1$.  Corresponding $c$-corridors in the two diagrams differ in length by exactly the factor $l$: where a $c$-corridor in $\Theta_1$ has $\tau^{nl}$ along one side and $(\tau a^{-1})^{nl}$ along the other, the corresponding $c$-corridor in $\Theta_l$ has  $t^n$ along one side and $(t a^{-l})^{n}$ along the other.
	
	Next we fill  the $c$-complementary regions. We wish to use Lemma \ref{shufflingsequence} to convert the filling in $c$-complementary regions of $\Theta_1$ to fillings in $\Theta_l$, but for any $c$-complementary region, the word along the perimeter of the region will not generally have an appropriate form. Its perimeter has the form $x_0(\tau a^{\epsilon_1})^{l k_1} x_1 \cdots (\tau a^{\epsilon_n})^{lk_n}x_n$, where $\epsilon_i \in \{0, -1\}$, $k_i \neq 0$,  and $x_i$ is a subword of $v$ and therefore is a word in $a$ and $\tau^l$. We add a collar of $2$-cells to change the boundary of the $c$-complemetary region to $x_0 (\tau^l a^{\epsilon_1 l})^{k_1} x_1 \cdots (\tau^l a^{\epsilon_n l})^{k_n}x_n$. In particular, if $L$ is a minimal area diagram for the word $ (\tau a)^{-l}\tau^l a^l$, $k_i$ copies of $L$ can be glued in to rewrite $(\tau a^{\epsilon_i})^{l k_i}$ to $(\tau^l a^{l})^{k_i}$. The result is a region with boundary that is a word in $a$ and $\tau^l$. 
	
		Apply Lemma~\ref{shufflingsequence} to convert each of these diagrams, without increasing area, to diagrams over $\langle a, t \mid a^t =a \rangle$  with boundary $x_0(t a^{l \epsilon_1})^{k_1} x_1 \dots (t a^{l\epsilon_n})^{k_n}x_n$ (as in Figure~\ref{fig:example5}), and use them to fill the $c$-complementary regions of $\Theta_l$. This produces a van Kampen diagram $\Theta_l$ for $u$ over ${Q}_l$.		
	
	\begin{figure}[!ht]
		\centering
		\begin{subfigure}[t]{4cm}\centering
			\includegraphics[scale=0.2]{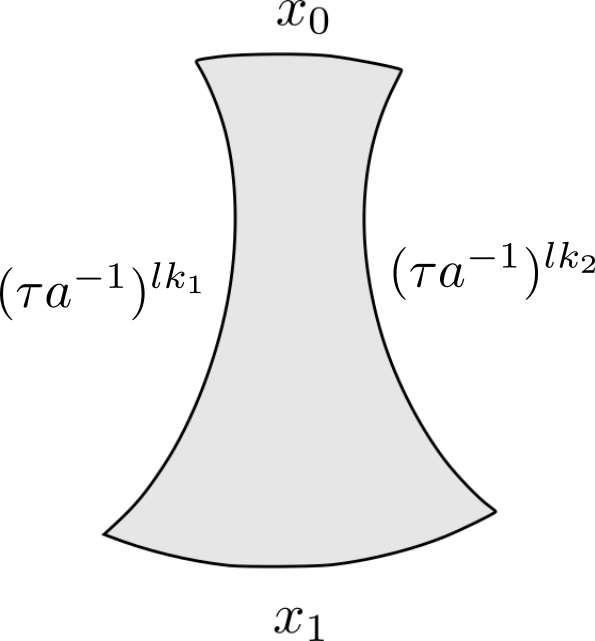}
			\caption{One $c$-complementary region in $\Theta_1$}
			\label{fig:example2}
		\end{subfigure}
		\hspace{0.2cm}
		\begin{subfigure}[t]{4cm}
			\centering
			\includegraphics[scale=0.2]{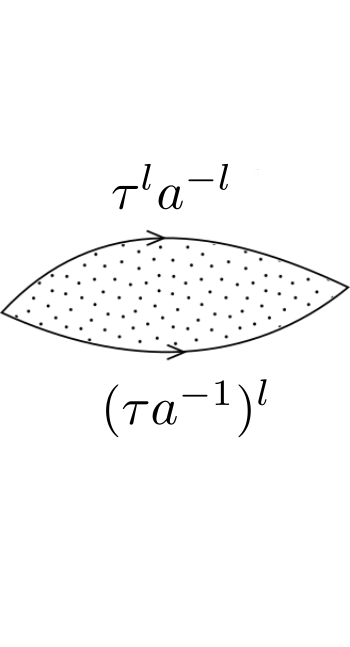}
			\caption{The diagram $L$. $\Area(L) \leq l^2$}
			\label{fig:example3}
		\end{subfigure}
		\hspace{0.2cm}
		\begin{subfigure}[t]{4cm}
			\centering
			\includegraphics[scale=0.2]{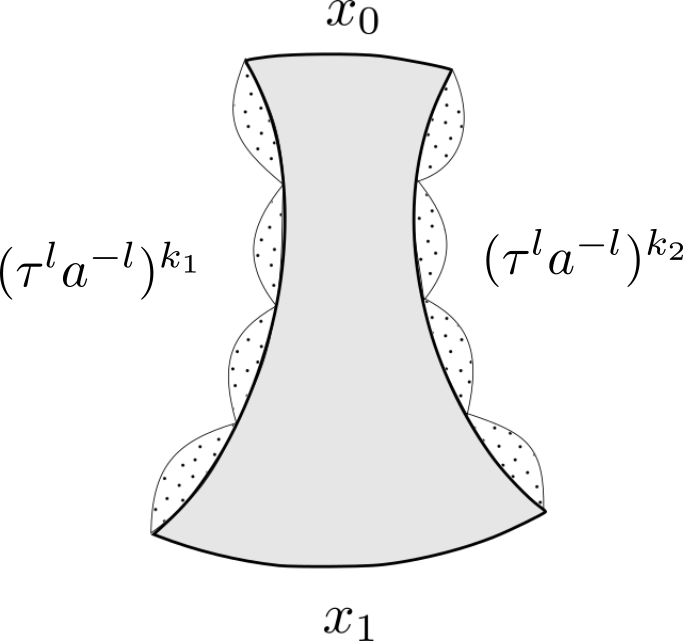}
			\caption{Gluing copies of $L$ along edges of the tops of $c$-corridors in $\Theta_1$}
			\label{fig:example4}
		\end{subfigure}
		\hspace{0.2cm}
		\begin{subfigure}[t]{4cm}
			\centering
			\includegraphics[scale=0.2]{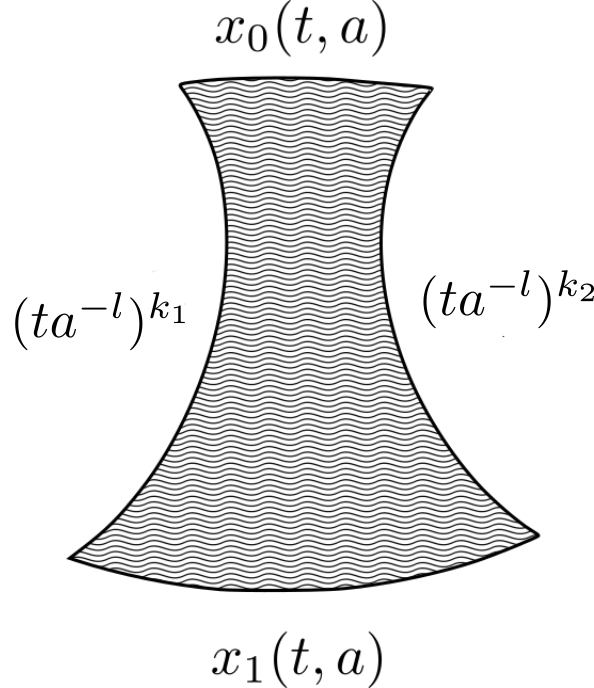}
			\caption{A new filling from Lemma \ref{shufflingsequence}, in terms of $t$.}
			\label{fig:example5}
		\end{subfigure}
		\caption{Converting a $c$-complementary region filling from $\Theta_1$ over ${Q}_1^{\tau}$ to one for $\Theta_l$ over ${Q}_l$}
		\label{fig:scallopedpotatoes}
	\end{figure}
	
	Finally we come to area estimates for $\Theta_l$. First observe that the total number of 2-cells in the $c$-corridors in $\Theta_1$ is at most the area of $\Theta_1$, which we determined earlier to be at most  $A_1 l^2{|u|}^2$. Correspondingly, there are  at most  $A_1 l {|u|}^2$  2-cells in the $c$-corridors in $\Theta_l$.  The number of copies of $L$ glued on to the $c$-complementary regions is at most $A_1 l{|u|}^2$, since it is the sum of the lengths of the $c$-corridors, divided by $l$. Since $L$ has area bounded above by $l^2$, the total area taken by copies of $L$ is at most $A_1 l^3 {|u|}^2$.  
	The total area of the $c$-complementary regions in  $\Theta_1$ is also at most  $A_1 l^2{|u|}^2$.   They, along with the attached copies of $L$, are converted to  the $c$-complementary regions  in  $\Theta_l$ without an increase in their area, as per Lemma~\ref{shufflingsequence}.  Therefore the area of   $\Theta_l$ is at most   $A_1 l {|u|}^2  + A_1 l^3 {|u|}^2 +  A_1 l^2{|u|}^2 \leq A_l |u|^2$, where $A_l := A_1(l + l^3+ l^2)$.  \end{proof}

\subsection{Completing our proof of Theorem \ref{Z2astZ}}

\begin{proof}[Proof of Theorem \ref{Z2astZ}(\ref{two1})]
	
	This is the case where $\phi$ has finite order.  By Lemma~\ref{get Phi} and Proposition~\ref{groomZ2*Z}, for the purposes of determining the Dehn function of $M_{\Phi}$, and thus $M_{\Psi}$, we can work with $M_{\Xi}$, which has the form
	$$M_{0, l, m}  \  :=  \  \langle a, b ,c, t \mid [a,b]=1, \   a^{t}=ab^{0},  \ b^{t}=b,\  c^{t}=c a^lb^m \rangle = \langle a, b ,c, t \mid [a,b]=[a,t] = [b,t]=1,\  c^{t}=c a^lb^m \rangle$$ for some $l, m \in \Z$. Let
	$$N_l  \ := \  \langle a, c, t \mid a^t=a,  \ c^t = ca^l \rangle.$$  
	These groups are not hyperbolic, so  their Dehn functions grow at least quadratically. We will show that these mapping tori have quadratic Dehn functions for all $l, m \in \Z$. All proofs of the quadratic upperbound for these groups can be reduced to the proof for ${M_{0,l,0} =  \langle a, b ,c, t \mid [a,b]=[a,t]=[b,t],\  c^{t}=c a^l \rangle}$, so we begin with this special case.\\

	Suppose $w$ is a word representing the identity in $M_{0,l,0}$.  Let $\Delta$ be a minimal area van~Kampen diagram for $w$ over $M_{0,l,0}$.  Let $\overline{w}$ be $w$ with all $b^{\pm 1}$  removed.  Then $\overline{w}=1$ in  $N_l$.   The $c$-pairing $P$ induced by $\Delta$ in turn induces a valid $c$-pairing $\overline{P}$ for $\overline{w}$  because  collapsing each $b$-corridor to the path along its bottom side gives a van~Kampen  diagram $\overline{\Delta}$ for $\overline{w}$ over $N_l$. 
	
	By Proposition~\ref{lemma:generalcaseZ2*Zv2}, there is a constant $A_l >0$ and a van~Kampen diagram $\overline{\Theta}$ for $\overline{w}$ over $N_l$ which has reduced $c$-corridors, induces $\overline{P}$, and has area at most $A_l |\overline{w}|^2$.

	The defining relations for $N_l$ are also defining relations for $M_{0,l,0}$ (as $m=0$), so $\overline{\Theta}$ is \emph{a fortiori} a van~Kampen diagram over $M_{0,l,0}$.  We aim to convert it from a van~Kampen diagram for $\overline{w}$, which contains no letters $b^{\pm1}$, to a van~Kampen diagram $\Theta$ for the original $w$, which may contain letters $b^{\pm 1}$.  We will do this without altering its $c$-corridors.  Rather, we will replace each $c$-complementary region in $\overline{\Theta}$ with an \emph{inflated} version so that the word around the boundary becomes $w$. 
	
	In $\Delta$ there are no partial $b$-corridors and no $b$-corridor  can cross  a $c$-corridor.    Therefore each word read  around the boundary of a $c$-complementary region in $\Delta$ contains the same number of $b$ letters as $b^{-1}$ letters. Since  the layout of $c$-corridors in $\overline{\Theta}$ agrees with that in $\overline{\Delta}$ (and so in $\Delta$), for each $c$-complementary region $\overline{C}$ in $\overline{\Theta}$, there is a corresponding    $c$-complementary region in $\Delta$. As in $\Delta$, each word $\overline{u}$ read around the boundary of the $c$-complementary region $\overline{C}$ in $\overline{\Theta}$ contains the same number of $b$ letters as $b^{-1}$ letters.  Therefore  $\overline{C}$ can be \emph{inflated} to put the necessary $b$ and  $b^{-1}$ in place by adding $b$-corridors to the boundary of $ \overline{C}$.    
	
	The total number of such $b$-corridors that we must insert is at most $|w|/2$.  The length of each $b$-corridor is at most the length of the boundary circuit $\partial \overline{C}$ of  the relevant $c$-complementary region  $\overline{C}$ in  $\overline{\Theta}$---at most a constant times $|\overline{w}|$---by an argument equivalent to Corollary \ref{c complementary region length}.   Thus the area of the resulting diagram $\Theta$ is at most the area of    $\overline{\Theta}$ (which is at most $A_l |\overline{w}|^2$) plus the number of 2-cells in $b$-corridors, which is no more than a constant times $|\overline{w}| \,  |w|$. In total, the area of $\Theta$ is at most a constant times $|w|^2$, as required.   
	
	Now we consider the case of $M_{0, l,m}$ for $m \neq 0$. If $l$ and $m$ are relatively prime, by Bezout's Lemma, there is a pair of integers $(x,y)$ such that $l y - m x = 1$. So there is a generating set $A, B$ of $\Z^2 = \langle a, b \rangle$ with ${A= a^{l} b^{m}}$ and ${B= a^xb^y}$ (generating since $a = A^yB^{-m}$ and $b = B^l A^{-x}$), for which our group has the presentation $${\langle A, B, c, t \mid [A, B]=1, A^t =A, B^t=B, c^t=cA \rangle},$$ the same as $M_{0,1,0}$. Therefore the Dehn function is quadratic. Finally, if $l$ and $m$ are not relatively prime, let $n:=\mbox{gcd}(l, m)$. Then $M_{0,l,m}$ is a subgroup of index $n$ in $M_{0,\frac{l}{n}, \frac{m}{n}}$.  But then $M_{0,\frac{l}{n}, \frac{m}{n}}$ has a quadratic Dehn function and hence so does $M_{0,l,m}$.
\end{proof}

\begin{proof}[Proof of Theorem \ref{Z2astZ}(\ref{two2})]  
	This is the case where $\phi$ has a non-unit eigenvalue. As $K:=\langle a, b \rangle \isom \Z^2$ quasi-isometrically embeds in $\Z^2 \ast \Z$ and    $\Phi \restricted{K} = \phi$    
	is an automorphism of $K$, Lemma~\ref{BridsonGersten} implies that  the Dehn function of $M_{\Phi}$ is bounded below by an exponential function. From Lemma \ref{Exponential}, the Dehn functions of mapping tori of RAAGs are always bounded above by exponential functions. Thus $M_{\Phi}$ and so $M_{\Psi}$ has exponential Dehn function.	\end{proof}

\begin{proof}[Proof of Theorem \ref{Z2astZ}(\ref{two3})]   
	This is the case where  $\phi$ has infinite order and only unit eigenvalues.  We will show that $M_{\Phi}$ and thus $M_{\Psi}$ has a cubic Dehn function. By Lemma~\ref{get Phi} and Proposition~\ref{groomZ2*Z}, for the purposes of determining the Dehn function, we can work with $M_{\Xi}$ which has the form
	$$M_{k, l, m}  \  :=  \  \langle a, b ,c, t \mid [a,b]=1, \   a^{t}=ab^{k},  \ b^{t}=b,\  c^{t}=c a^lb^m \rangle$$ for some $k, l, m \in \Z$ with $k \neq 0$. Let  
	$$N_l  \ := \  \langle a, c, t \mid a^t=a,  \ c^t = ca^l \rangle.$$

	Suppose $w$ is a freely reduced word representing the identity in $M_{k,l,m}$. Let $\Delta$ be a minimal area van~Kampen diagram for $w$ over $M_{k,l,m}$. Let $\overline{w}$ be $w$ with all $b^{\pm 1}$  removed. Then $\overline{w}=1$ in  $N_l$. As in Case \eqref{two1} above, the $c$-pairing $P$ induced by $\Delta$ induces a valid $c$-pairing $\overline{P}$ for $\overline{w}$.

	By Proposition~\ref{lemma:generalcaseZ2*Zv2} there is a constant $A_l >0$ dependent only on $l$ such that   $\overline{w}$ admits a van~Kampen diagram $\overline{\Theta}$ over $N_l$ which also induces $\overline{P}$ and has area at most $A_l |\overline{w}|^2$. 
	Again, as in Case \eqref{two1} above, by Corollary \ref{c complementary region length}, there exists a constant $K>0$ such that the boundary circuit of any $c$-complementary region $\overline{R}$ in $\overline{\Theta}$ has length at most $K | \overline{w} |$. Each such region $\overline{R}$ is a diagram over $\langle a,t \mid a^t =a \rangle$.

	Each such $\overline{R}$ has a maximal geodesic tree in its 1-skeleton---that is, a tree reaching all vertices and with the property that there is a root vertex  $v_{\overline{R}}$ on the boundary $\partial \overline{\Theta}$ such that for every vertex $v$ in  $\overline{R}$, the distance from $v_{\overline{R}}$ in the tree is the same as in the 1-skeleton  of $\overline{R}$.
	
	The diameter of each $c$-complementary region $\overline{R}$ is linear in $|\overline{w}|$ and so in $|w|$.  After all, every vertex in $\overline{R}$  is contained in an $a$-corridor that extends to the boundary of $\overline{R}$. The length of each $a$-corridor is  the number of $t$-corridors that cross it, and there are at most   $K |\overline{w}|/2$ many $t$-corridors in $\overline{R}$. So the maximum distance  to the boundary is  $K |\overline{w}|/2$ and thus the diameter of the $c$-complementary region  $\overline{R}$ is at most $(K+1)|\overline{w}|$.  
	
	We now apply the electrostatic model from Section~\ref{sec:electrostatic model} to \emph{inflate} $\overline{\Theta}$ to a van~Kampen diagram for $w$ over $M_{k,l,m}$. 
	
	Since $\overline{\Theta}$ induces a valid $c$-pairing, this can be done by inserting $b$-corridors within the $c$-complementary regions.   
	
	First we charge the diagram with at most $A_l|\overline{w}|^2$ many $b$-charges (in effect, replacing all of the 2-cells for defining relations from $N_{l}$ with the corresponding 2-cells for defining relations from $M_{k,l,m}$).   Next connect each charge in $\overline{R}$ by a   $b$-partial corridor of length no more than $B|w|$ to the root $v_{\overline{R}}$.   
	The total area of these $b$-partial corridors is at most a constant times $|\overline{w}|^3$.   
	Finally, insert  $b$-corridors (each of at most a constant times $|w|$) along the  boundaries  of the $c$-complementary regions to rearrange the (at most a constant times $|w|^2$ many) $b$ and $b^{-1}$ until the perimeter word is $w$.  
	
	The resulting diagram $\Theta$ for $w$ over $M_{k, l, m}$ has at most the area of $\overline{\Theta}$ (at most quadratic in $|\overline{w}|$), plus the total area of the $b$-partial corridors (at most cubic in $|\overline{w}|$), plus the total area of the $b$-corridors (at most cubic in $|w|$)---in total, at most cubic in $|w|$.

	So the Dehn function of  $M_{k,l,m}$ grows at most cubically.
	
	As $\phi$ has infinite order and only unit eigenvalues, it has a $2 \times 2$ Jordan block $A$ and so, by Lemma \ref{BridsonGersten}, the Dehn function of the mapping torus has a cubic lower bound.
\end{proof}	

\section{Mapping tori of RAAGs of the product of two free groups}  
\label{HigherRank}

Here we will prove  Theorem~\ref{F_k x F_l} concerning Dehn functions of mapping tori of products $F_k \times F_l$ of free groups.  

\subsection{Automorphisms of \texorpdfstring{$F_k \times F_l$}{Fk x Fl}}

Suppose $X$ and $Y$ are disjoint finite sets with $\abs{X} = k$, $\abs{Y} = l$, and  $k, l \geq 2$.  Let $\Gamma$ be the bipartite graph with vertex set $X \cup Y$ and an edge between a pair of vertices if and only if one is in $X$ and the other is in $Y$.  So  $G= F_k \times F_l$ is the RAAG $A_{\Gamma}$. 

Our first task is to  explain the opening part of   Theorem~\ref{F_k x F_l}, which amounts to:

\begin{lemma} \label{first part of thm} Given $\Psi \in \Aut(G)$, we can find $\phi_1 \in  \Aut(F_k)$ and $\phi_2 \in  \Aut(F_l)$ such that  $\Phi = \phi_1 \times \phi_2$  has the property that  $[\Phi] =  [\Psi^2]$  in $\Out(F_k \times F_l)$.  
\end{lemma}

This lemma allows us to work with $\Phi$ instead of $\Psi$ when trying to find the Dehn function of $M_{\Psi}$, since $\delta_{M_{\Psi}} \simeq \delta_{M_{\Phi}}$  by Lemma~\ref{lemma: simplifying automorphisms doesn't change DF}.  

For a vertex $x$ in a graph, $\star(x)$ is the subgraph consisting of all edges incident with $x$ and $\link(x)$ is the set of vertices adjacent to $x$.  We will prove Lemma~\ref{first part of thm} with the help of:

\begin{lemma}[Laurence \cite{Laurence},  Servatius  \cite{Servatius}]\label{generatorsforFkxFl}  If $A_{\Gamma}$ is a RAAG, then the following is a generating set for  $\Aut(A_{\Gamma})$:
	\begin{enumerate}
		\item All inner automorphisms: for a vertex $x$ of $\Gamma$, $\iota_x: y \mapsto x^{-1} y x$ for all $y\in A_{\Gamma}$.
		\item All inversions: maps  that send $x \mapsto x^{-1}$ for some vertex $x$ of $\Gamma$ and leave all other vertices fixed.   
		\item All partial conjugations: for a vertex $x$ in $\Gamma$ and a  connected component $C$ of $\Gamma - \star(x)$, map $y \mapsto x^{-1} y x$ for all vertices $y$ in $C$ and fix all other vertices.		
		\item All transvections:  for a pair of vertices $x, y$ of $\Gamma$  such that $\link(x) \subseteq \star(y)$,   $\tau_{x,y}$ maps $x \mapsto xy$ and fixes all other vertices. 
		\item All graph symmetries:   automorphisms induced by  the restriction of a graph symmetry to the vertex set.   
	\end{enumerate}
\end{lemma}

We can see how this generating set reflects the product structure in the instance of $A_{\Gamma} = F_k \times F_l$.

\begin{cor} \label{LS cor}
	When $A_{\Gamma} = F_k \times F_l$, the inversions, partial conjugations, and transvections of the Laurence--Servatius generators of  $\Aut(A_{\Gamma})$ restrict to automorphisms of $F_k = F(X)$ and $F_l  = F(Y)$.  The same is true of the graph symmetries, except when $k =l$, in which case there are automorphisms exchanging $X$ and $Y$.   
\end{cor}

\begin{proof} This is immediate for the inversions.  
	It is true of the partial conjugations because  if $x \in X$, then ${\Gamma-\star(x) = X-\{x\}}$.
	As for the transvections, suppose $y \in Y$, and so $\star(y) = \{y\} \cup X$.    If $w \in Y,$ then $\link(w) = X$, and so ${\link(w) \subseteq \star(y)}$.  So $\tau_{w,y}: w \mapsto wy$ (and fixes all other elements of $X \cup Y$), and   $\tau_{w,y}$ restricts to automorphisms of $F_k = F(X)$ and $F_l  = F(Y)$ as claimed.  
	If, on the other hand, $w \in X$, then since $\link(w)=Y$, $\link(w) \subseteq \star(y)$ if and only if $Y=\{y\}$, and so, as $l \geq 2$, there are no transvections    $\tau_{w,y}$.  
	Likewise the result holds for  transvections    $\tau_{w,x}$ with $x \in X$.		The result for graph symmetries is straight-forward.
\end{proof}

\begin{proof}[Proof of Lemma~\ref{first part of thm}]
	Every automorphism $\Omega$ of $F_k \times F_l$  is a product $\Pi$ of the Laurence--Servatius generators.   As $\Inn(F_k \times F_l) \trianglelefteq \Aut(F_k \times F_l)$, the inversions, partial conjugations, transvections, and graph symmetries in this product can be shuffled to the end as a suffix $\Omega_0$, so as to express  $\Omega$ as $\iota_g \circ \Omega_0$ for some inner automorphism $\iota_g$.   
	
	If $\rho \in \Aut(F_k \times F_l)$ is a graph symmetry and $\tau$ is an inversion, partial conjugation, or transvection, then $\rho^{-1} \tau \rho$ is again an inversion, partial conjugation, or transvection (respectively). So $\Omega_0 = \Omega_1 R$   where $R$ is the product of the graph symmetries in the product $\Pi$ and, 
	by  Corollary~\ref{LS cor}, $\Omega_1$ restricts to   automorphisms of   $F_k = F(X)$ and $F_l  = F(Y)$. 
	
	Now the lemma concerns some $\Psi \in \Aut(G)$.  
	Take  $\Omega = \Psi^2 = \iota_g\circ\Omega_1R$.  In this case the product $R$ will have an even number of terms and so (whether or not $k = l$), Corollary~\ref{LS cor} tells us that $R$    restricts to  automorphisms  of    $F(X)$ and $F(Y)$.  So taking $\Phi =   \Omega_1 R$ we have that  $\Phi = \phi_1 \times \phi_2 $ for some  $\phi_1 \in  \Aut(F_k)$ and $\phi_2 \in  \Aut(F_l)$  and $[\Psi^2] = [\Phi] \in \Out(F_k \times F_l)$, as required. 
\end{proof}

Here is a further lemma we will use to adapt a RAAG automorphism to one better suited to calculation  of the  Dehn function of the mapping torus. 

\begin{lemma}\label{lemma: F_k x F_l}  Suppose $\phi_1, \psi_1 \in \Aut(F_k)$ and  $\phi_2, \psi_2 \in \Aut(F_l)$ are such that $[\phi_1]= [\psi_1]$ in $\Out(F_k)$ and $[\phi_2] =[\psi_2]$ in $\Out(F_l)$.  Then $\delta_{M_{\phi_1 \times \phi_2}} \simeq  \delta_{M_{\psi_1 \times \psi_2}}$.
\end{lemma}

\begin{proof} Suppose $\phi_1 =  \iota_{a}\circ\psi_1 $ for $a \in F_k$ and $\phi_2 =   \iota_{b}\circ\psi_2$ for $b \in F_l$. Then viewing $a$ and $b$ as elements of $F_k \times F_l$ via the natural embeddings $F_k \to F_k \times F_l$ and $F_l \to F_k \times F_l$, we have that $\phi_1 \times \phi_2 = \iota_{ab}\circ  (\psi_1 \times \psi_2) $, as $b$ commutes with all elements of $F_k$ and $a$ commutes with all elements of $F_l$.  So $[\phi_1 \times \phi_2]=[\psi_1 \times \psi_2]$  in $\Out(F_k \times F_l)$ and it follows from Lemma \ref{lemma: simplifying automorphisms doesn't change DF} that $\delta_{M_{\phi_1 \times \phi_2}} \simeq  \delta_{M_{\psi_1 \times \psi_2}}$. \end{proof} 

\subsection{Growth of  free group automorphisms} \label{subsec:growth}

Suppose $F$ is a  finite-rank free group.   The \emph{growth} $g_{\phi, X}: \N \to \N$ of an automorphism $\phi: F \to F$ with respect to a free basis $X$  is defined by  $$g_{\phi, X}(n)  \ :=  \ \max_{x\in X}\{|\phi^n(x)|\},$$ where $|  \, g  \,  |$ denotes the length of a shortest word on $X$ representing $g$.    We write $f \simeq_{\ell} g$ when $f, g : \N \to \N$ are Lipschitz equivalent; that is, when there exist $C_1, C_2>0$ such that ${C_1 g(n) < f(n) < C_2 g(n)}$ for all $n$.   Up to $\simeq_{\ell}$, free group growth $g_{\phi,X}$ does not depend on the choice of finite basis $X$.

We say that  $\phi \in \Aut( F )$ is \textit{periodic} when there is $l>0$ such that $\phi^l$ is an inner automorphism.  We say that $\phi$ is  \emph{polynomially growing} when there is $d \geq 0$ such that  $g_{\phi} (n) \simeq_{\ell} n^d$, and  $\phi$ is called \textit{exponentially growing}  otherwise.

Levitt \cite[Theorem~3]{Levitt} shows that in the polynomially growing case, for every $x \in F$,  there exists $d_x \geq 0$ such that  $| \phi^n(x) |  \simeq_{\ell} n^{d_x}$, and in the exponentially growing case, there exists $x \in F$ and    $\lambda>1$ such that $|\phi^n(x)| > \lambda^n$ for all $n \in \N$.  For $g \in F$, let  $||g||$ denote the cyclically reduced length, that is, the length of the shortest word representing a conjugate of $g$.
The corresponding result holds for $|| \, \cdot \, ||$ in place of $| \, \cdot \, |$ (though possibly with different powers and exponential functions)  \cite[Theorem~6.2]{Levitt}.   

\begin{definition}
	If $\phi$ is polynomially growing, let $d$ be the largest degree so that for some $g \in F(X)$, $||\phi^n(g)|| \simeq_{\ell} n^d$. In this case, define $g^{cyc}_{\phi}(n) = n^{d}$. Otherwise, define $g^{cyc}_{\phi}(n) = 2^n$.
\end{definition} 

	In contrast with growth, $ g^{cyc}(n) \not\simeq_{\ell}\ \max_{x\in X}\{||\phi^n(x)||\}$ in general.  That is, knowing what happens to generators is not enough to understand $g_{\phi}^{cyc}$.  Indeed,    \cite[Lemma 5.2]{Levitt} gives  a family of automorphisms $\phi_l$ and bases $X_l$ ($l \in \mathbb{N}$) such that  $d_x \in \{0,1\}$ for all $x\in X_l$, but  there exists $g \in F(X_l)$ such that $d_g = l$.   

To establish lower bounds for the Dehn function of $M_{\phi_1 \times \phi_2}$, we will use cyclically reduced growth to find lower bounds for the growth of a family of words under repeated application of our automorphism. To  establish upper bounds for the Dehn function of $M_{\phi_1 \times \phi_2}$, we will use growth to provide upper bounds for the growth of words under our automorphism. Results of Levitt provide a way to bridge the gap between the two types of growth: in all cases where the cyclically reduced growth and traditional growth disagree, it is possible to exchange $\phi_1 \times \phi_2$ with a related automorphism $\hat{\xi}_1 \times \hat{\xi}_2$ for which  $g_{\hat{\xi_i}} \simeq_{\ell} g^{\text{cyc}}_{\hat{\xi_i}}$ for $i \in \{1,2\}$. The mapping tori $M_{\phi_1\times \phi_2}$ and $M_{\hat{\xi_1} \times \hat{\xi_2}}$ have equivalent Dehn functions. We expand on this below.

Here is a summary  of results of Levitt \cite{Levitt} and Piggot \cite{Piggott} on properties of growth and cyclically reduced growth in free groups:

\begin{lemma}  \label{equiv conveniences} Suppose   $\phi \in \Aut( F)$. 
	\begin{enumerate}
		\item \label{equality of cyclic growth}
		$g^{cyc}_{\phi}=g^{cyc}_{\psi}$ if $[\phi] = [\psi] \in \Out(F)$.
		\item \label{equality of cyclic growth under powers} If $\phi$ is polynomially growing,
		$g^{cyc}_{\phi}\simeq_{\ell}g^{cyc}_{\phi^k}$ for all $k\in \mathbb{N}$.
		\item (Theorem 0.4 of \cite{Piggott}) $g_{\phi} \simeq g_{\phi^{-1}}$.  \label{inv the same}
		\item (Section~2 of \cite{Levitt}) $g^{cyc}_{\phi} \simeq_{\ell} g^{cyc}_{\phi^{-1}}$. \label{inv the same 2} 
		\item (Theorem 3 of \cite{Levitt}, cf.\ Bestvina--Feighn--Handel \cite{BFH1}).   Either $g_{\phi}(m) \simeq 2^m$,  or $g_{\phi}(m) \simeq  m^{d}$ for some $d\in \mathbb{N}$.
		
		\item (By Corollary 1.6 of \cite{Levitt}) \label{Levitt FP} If  $n \mapsto ||\phi^n(g)||$ grows polynomially, then  there exists $p\geq 1$ and  $\xi \in \Aut(F)$ such that $[\phi^p] = [\xi]$ in $\Out(F)$ and $\xi$ admits a non-trivial fixed point.
		\item (By Lemma 2.3 of  \cite{Levitt}) \label{Levitt lem 2} Suppose   $\xi \in \Aut(F)$ satisfies $g_{\xi}^{cyc} (n)\simeq_{\ell} n^d$ with $d>0$ and $\xi$ has a non-trivial fixed point set. Then $g_{\xi}(n) \simeq_{\ell} n^d$. (There are two other possible behaviors for $[\xi]$: either for some $k$, $[\xi^k] = [\textup{Id}]$, or $\xi$ has exponential growth.)
	\end{enumerate}
\end{lemma}

In building van Kampen diagrams and shuffling relators we will use both forward and backward iterates of our automorphism. Lemma \ref{equiv conveniences} (\ref{inv the same}) and (\ref{inv the same 2}) imply that we can use the same functions to estimate both. (\ref{equality of cyclic growth}) implies that cyclic growth can be defined for outer automorphisms. This can fail for growth.

\begin{lemma}\label{equiv2} Suppose $\phi \in \Aut(F)$ has polynomial growth and is not periodic. Then there exists $\xi  \in \Aut(F)$ and $p\geq 0$ with $[\xi] = [\phi^p]\in \Out(F)$ and $g^{cyc}_{\phi} \simeq g^{cyc}_{\xi} \simeq g_{\xi}$. Moreover, for any $q\geq 1$,  $g^{cyc}_{\phi} \simeq g^{cyc}_{\xi^q} \simeq g_{\xi^q}$.
\end{lemma}

\begin{proof} We use Lemma \ref{equiv conveniences}: take $\xi$ as per \eqref{Levitt FP} and then apply     \eqref{Levitt lem 2}, \eqref{equality of cyclic growth}, and \eqref{equality of cyclic growth under powers}.
\end{proof}

\begin{lemma} \label{no change to Dehn fns}  Suppose $\phi_1 \in \Aut(F_k)$ and $\phi_2 \in \Aut(F_l)$.   For $i =1,2$, suppose $p_i \geq 0$  is such that $[\phi_i^{p_i}] = [\xi_i]$  as in Lemma~\ref{equiv2}.  Define $\hat{\xi}_1 = \xi_1^{p_2}, ~\hat{\xi}_2 = \xi_2^{p_1}$.  Then $M_{\phi_1\times \phi_2}$ and $M_{\hat{\xi}_1\times \hat{\xi}_2}$ have equivalently growing Dehn functions and $g_{\phi_i}^{cyc} \simeq  g_{\hat{\xi}_i}^{cyc} \simeq g_{\hat{\xi}_i}$. \end{lemma}

\begin{proof}
	By Lemma~\ref{lemma: simplifying automorphisms doesn't change DF}, the Dehn functions of $M_{\phi_1 \times \phi_2}$ and ${M_{(\phi_1 \times \phi_2)^{p_1p_2}} = M_{(\phi_1^{p_1})^{p_2} \times (\phi_2^{p_2})^{p_1}}}$ are equivalent. By Lemma~\ref{lemma: F_k x F_l}, we may also pick convenient representatives of the outer automorphism classes without changing the Dehn function, so $M_{\xi_1^{p_2} \times \xi_2^{p_1}} = M_{\hat{\xi}_1\times \hat{\xi}_2}$ will also have equivalent Dehn function to  $M_{\phi_1 \times \phi_2}$.
\end{proof}

\subsection{Dehn function lower bounds}

A result similar to Lemma \ref{part2} was proved by Brady and Soroko \cite{Soroko} in the context of Bieri doubles.

\begin{lemma}\label{part2} Suppose that $\Phi \in \Aut(F_k \times F_l)$ has the form $\Phi =\phi_1 \times \phi_2$ where  $\phi_1 \in \Aut(F_k)$ and $\phi_2 \in \Aut(F_l)$. Suppose that $g_{\phi_1}^{cyc} \preceq_{\ell} g_{\phi_2}^{cyc}$. 
	\begin{enumerate}
		\item If $g_{\phi_1}^{cyc}(n) \simeq_{\ell} n^{d_1}$, for  some $d_1\geq 0$, then $n^{d_1 +2} \ \preceq \  \delta_{M_{\Phi}}(n).$
		\item If $g_{\phi_1}(n) \succeq 2^n$, then the Dehn function $\delta_{M_{\Phi}}(n) \succeq 2^n$.
	\end{enumerate}
\end{lemma}

\begin{proof}  If $g_{\phi_1}^{cyc}(n) \simeq_{\ell} n^{d_1}$ then by Lemma \ref{equiv conveniences}~(\ref{inv the same 2}), $g_{\phi_1^{-1}}^{cyc}(n) \simeq_{\ell} n^{d_1}$. Let $x \in F_k$ and $C_1>0$ be such that ${||\phi_1^{-n}(x)|| \geq C_1n^{d_1}}$ for all $n \in \N$. If $\phi_2$ is polynomially growing, there is $y \in F_l$ and $C_2>0$  such that ${||\phi_2^n(y)|| \geq C_2 n^{d_2}}$, and if $\phi_2$ is exponentially growing, choose $y$ such that for some $b>1$, $b^n \preceq ||\phi_2^n(y)||$.

	Consider the word ${w_n = t^{-4n}y^nt^{4n}x^n t^{-4n} y^{-n}t^{4n}x^{-n}}$. We will show that $\Area(w_n) \succeq n^{d_1+2}$. Since  $$ {20n \ \leq|w_n| \ \leq \ 16n + 2n|x| + 2n|y|},$$ this lower bound on the area will imply that the Dehn function dominates the polynomial $n^{\min\{d_1, d_2\} +2}$.
	Consider the following picture: 
	\begin{figure}[!ht]
		\centering
		\begin{subfigure}[b]{3.7cm} 
			\centering
			\includegraphics[scale=0.17]{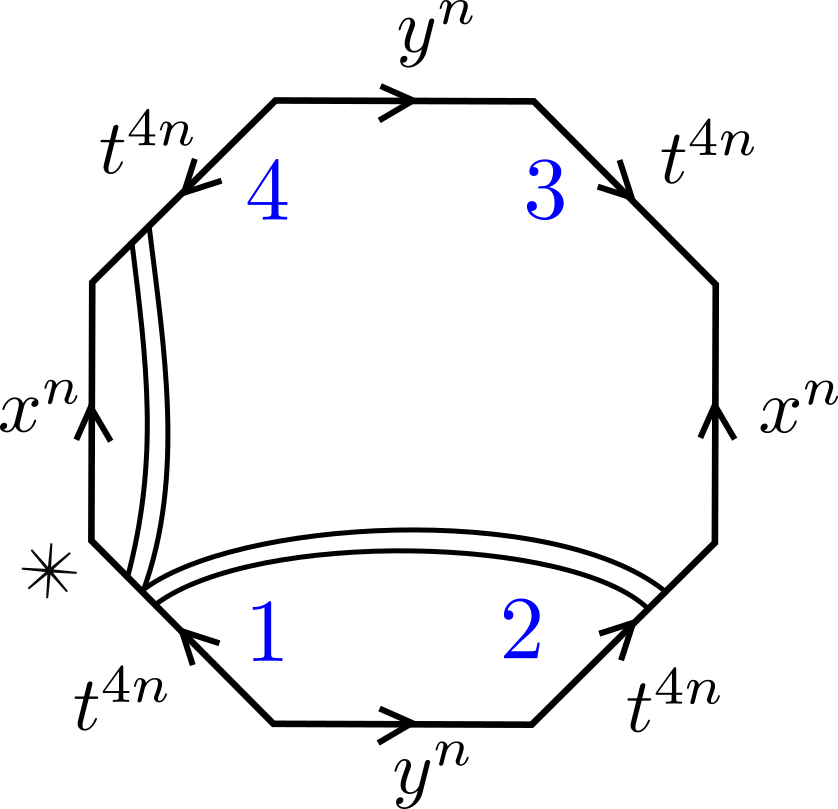}
			\caption{There are few choices for $t$-corridor patterns}
			\label{fig:BridsonGerstenunfilled}
		\end{subfigure}\hspace{1cm}
		\begin{subfigure}[b]{5cm}  
			\centering
			\includegraphics[scale=0.17]{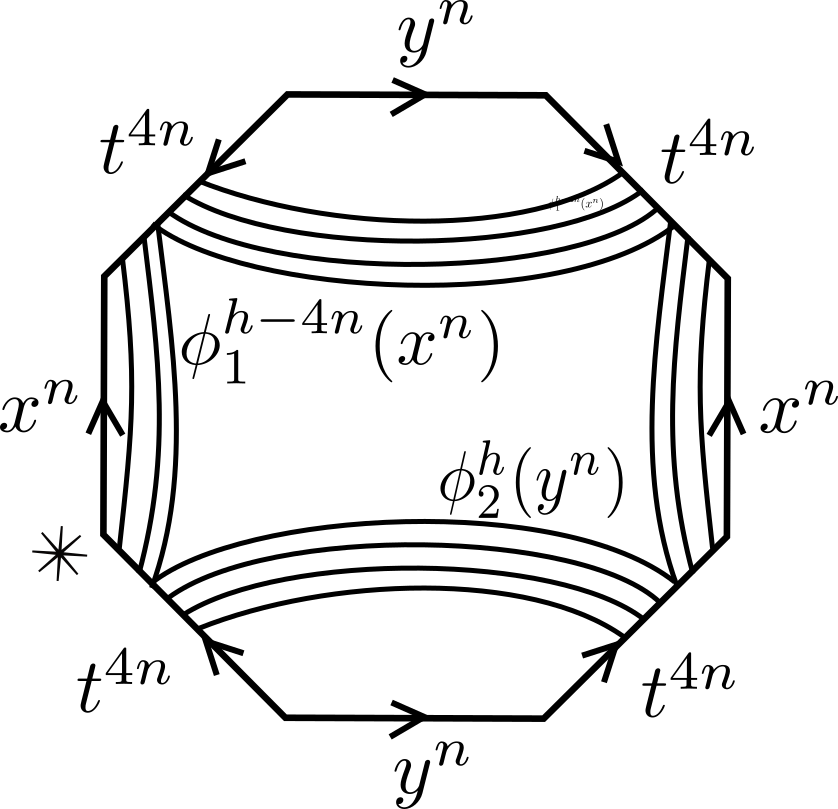}
			\caption{The adjacent $t$-corridors at $h$ determine the diagram.}
			\label{fig:BridsonGerstenfilled}
		\end{subfigure}
		\caption[Fillings in the mapping torus with base $F_k \times F_l$]{Fillings in the mapping torus with base $F_k \times F_l$ and automorphism $\phi_1 \times \phi_2$}
	\end{figure}

	A $t$-corridor beginning on side 1 can only end on sides 2 or 4. Since $t$-corridors cannot cross, there is some value $h \in \{0, \dots, 4n\}$ so that the first $h$ $t$-corridors emanating from side 1 end on side 2 and the remainder end on side 4. This switching point $h$ determines the diagram, as seen in Figure \ref{fig:BridsonGerstenfilled}. If $h \geq 2n$, then a stack of at least $2n$ $t$-corridors $\mathcal{C}_1, \mathcal{C}_2, \ldots$ (emanating from the 1st, 2nd etc., edge of side 1) start on side 1 and end on side 2.  If $h < 2n$, then  a stack of at least $2n$ $t$-corridors start on side 1 and end on side 4: in this case take $\mathcal{C}_1$  to be that emanating from the final edge of side 1, $\mathcal{C}_2$  to be that emanating from the penultimate edge, etc.  Let $|\mathcal{C}_i|$ be the area of corridor $\mathcal{C}_i$, that is, the number of 2-cells in the corridor.
	
	The area of each corridor can be bounded from below by the length of the shortest side, and that can be bounded below by the cyclically reduced length of the shortest side. For  $g \in F$ and $n \in \mathbb{N}$, $||g^n|| = n||g||$, so we get 
	$$|\mathcal{C}_i| \ \geq\   \min\{ ||\phi_1^{-i}(x^n)||, ||\phi_2^{i}(y^n)||\} \ = \ \min\{ n||\phi_1^{-i}(x)||, n||\phi_2^{i}(y)||\}  \ \geq \  n \min\{C_1i^{d_1}, C_2i^{d_2}\}.$$
	
	Summing the areas of corridors $\mathcal{C}_n, \ldots, \mathcal{C}_{2n-1}$, we find that 
	$$\Area(\Delta) \ \geq \  \sum_{i=n}^{2n-1}|\mathcal{C}_i|  \  \geq \ n^2\min\{C_1n^{d_1}, C_2n^{d_2}\}.$$ 
	
\end{proof}

\subsection{Dehn function upper bounds}

\begin{lemma}\label{part1} Suppose $\Phi \in \Aut(F_k \times F_l)$ has the form $\Phi = \phi_1 \times \phi_2$, where $\phi_1 \in \Aut(F_k)$ and $\phi_2 \in \Aut(F_l)$.  If $n^{d_1} \simeq  g_{\phi_1}(n) \preceq g_{\phi_2}(n)$, then $\delta_{M_{\Phi}}(n) \preceq n^{d_1+2}$. In the case that $\phi_1$ is periodic, $\delta_{M_{\Phi}}(n) \preceq n^{2}$. 
\end{lemma}

\begin{proof} We have a finite presentation  $$\langle x_1, \ldots, x_k, \, y_1, \ldots, y_l, \,  t \mid   [x_i,y_j]=1, \  t^{-1}x_it= \phi_1(x_i),  \  t^{-1}y_jt= \phi_2(y_j), \  \forall i, j  \rangle$$   for $M_{\Phi}$.
	Suppose $w$ is a word $$w \ = \ w_1 t^{c_1} w_2 t^{c_2} \cdots w_m t^{c_{m}}$$   in which the subwords $w_i$ are in $\langle x_1, \dots, x_k, y_1, \dots y_l\rangle$.	Let $n$ be the length of $w$ and suppose $w$ represents the identity in $M_{\Phi}$. To  bound  $\delta_{M_{\Phi}}(n)$ from above  we will estimate how many defining relators need to be applied to $w$ to reduce it to the empty word. (We are also allowed to insert or remove inverse pairs of generators $x_i^{-1}x_i$ or $x_ix_i^{-1}$, but only applications of  defining relators will count towards the area.)  
	
	By applying fewer than $n^2$ commutators, convert  each $w_i$ to $u_iv_i$ for some reduced $u_i \in \langle x_1, \dots, x_k\rangle$ and ${v_i \in \langle y_1, \dots, y_l \rangle}$, thereby rewriting $w$ as a word ${w^{\prime}  =  u_1v_1t^{c_1} \dots u_m v_m t^{c_m}}$, which has length at most $n$. 
	
	Next     convert  $w^{\prime}$ to a product $\bar{v}u$ of the word  $\bar{v} = v_1t^{c_1}\cdots v_m t^{c_m}$ with a word $u$ in $\langle x_1, \cdots, x_k\rangle$, by applying defining relators to shuffle all  the $x_1^{\pm 1}, \dots, x_k^{\pm 1}$ in $w'$ to the right.  The  word $\bar{v}$ represents the identity in $F_l \rtimes_{\phi_2} \langle t \rangle$ and   $u$  represents the identity in $F_k$. Indeed, the index sum of $t$ in $w$  is zero, so gathering all powers of $t$ together on the left would produce a word of the form $vu$ with $u \in F_k$ and $v \in F_l$  which represents the identity in $F_k \times F_l$, and so $u$ and $v$ freely reduce to the identity---in particular, $\bar{v} = v =1$ in $F_l \rtimes_{\phi_2} \langle t \rangle$.

	This shuffling of $w'$ into  $\bar{v}u$  results in growth of slow-growth elements (the $F_k$ factor), but not in growth of fast-growth elements (the $F_l$ factor).	 We can (crudely) estimate its cost  by giving an upper bound on the length  to which a letter $x_i^{\pm 1}$ can grow in the process: it passes at most $n$ letters $t$ or $t^{-1}$, each time with the effect of applying $\phi_1$ or $\phi_1^{-1}$. We are given that $n^{d_1} \simeq  g_{\phi_1}(n)$, so $n^{d_1} \simeq  g_{\phi^{-1}_1}(n)$, by Lemma~\ref{equiv conveniences} \eqref{inv the same}. Thus there is a constant  $K>0$ such that   $x_i^{\pm 1}$ can grow  to length at most $K n^{d_1}$. The cost to shuffle  (and in the process transform) all the (at most $n$) letters $x_i^{\pm 1}$  of  the $u_1, \ldots, u_m$ to the right past the  letters of $\bar{v}$ (of which there are at most $n$)  is at most $K n^{d_1+2}$.
	
	Next freely reduce $u$ to the empty word (at no cost to area), leaving the word $\bar{v}$, which represents the identity in $F_l\rtimes_{\phi_2}\mathbb{Z}$ and has length at most $n$. By Bridson--Groves \cite{BridsonGroves}, $\bar{v}$  can be reduced to the empty word using   no more than a constant $c$ times $n^2$ defining relations. 
	
	In conclusion, we have an upper bound of $n^2 + K n^{d_1+2} + cn^2$, which gives that $\delta_{M_{\Phi}}(n) \preceq n^{d_1+2}$ as required.
	
	Finally, we address the periodic case: suppose $l$ is such that $\phi_1^l$ is an inner automorphism.	
	By Lemmas~\ref{lemma: simplifying automorphisms doesn't change DF} and \ref{lemma: F_k x F_l}, $\delta_{M_{\Phi}} \simeq \delta_{M_{\phi_1^l \times \phi_2^l}} \simeq \delta_{M_{\textup{Id} \times \phi_2^l}} \simeq \delta_{M_{\textup{Id} \times \phi_2}}$.  We can estimate $\delta_{M_{\textup{Id} \times \phi_2}}$ by the above argument in the special case that $\phi_1 = \textup{Id}$.   In this case the cost of shuffling the $x_i^{\pm 1}$  through the word is at most $n^2$ (rather than $K n^{d_1+2}$) since they do not grow in the process, and so $ \delta_{M_{\Phi}}(n) \simeq \delta_{M_{\textup{Id} \times \phi_2}} (n) \preceq n^2$.	 \end{proof}

\begin{proof}[Proof of Theorem~\ref{F_k x F_l}]
	We have  $G= F_k \times F_l$, where $k,l \geq 2$, and $\Psi \in \Aut(F_k \times F_l)$.   Lemma~\ref{first part of thm} identified a  $\Phi = \phi_1 \times \phi_2$ with $\phi_1  \in \Aut(F_k)$ and $\phi_2 \in \Aut(F_l)$ which (by Lemma~\ref{lemma: simplifying automorphisms doesn't change DF}) has  $\delta_{M_{\Psi}} \simeq \delta_{M_{\Phi}}$.  
	
	Provided   $\phi_i$ is not periodic,   Lemmas~\ref{equiv2} and \ref{no change to Dehn fns} imply that even if  $g_{\phi_i} \not \simeq g^{\text{cyc}}_{\phi_i}$, there is $\hat{\xi}_i$ such that $g^{\text{cyc}}_{\phi_i} \simeq g_{\hat{\xi}_i} \simeq g^{\text{cyc}}_{\hat{\xi}_i} $ with $M_{\hat{\xi}_1 \times \hat{\xi}_2} \simeq M_{\Phi}$.

	The theorem claims that
	\begin{enumerate}
		\item[\eqref{dos}] If $[\phi_1^p] = [\textup{Id}] \in \Out(F_k)$ for some $p \in \mathbb{N}$ (that is, $\phi_1$ is periodic), then  $\delta_{M_{\Psi}}(n) \simeq n^2$.
		\item[\eqref{uno}] If $n^{d_1} \simeq g_{\phi_1}^{cyc}(n) \preceq g_{\phi_2}^{cyc}(n)$, then  $\delta_{M_{\Psi}}(n) \simeq n^{d_1+2}$, and likewise with the indices $1$ and $2$ interchanged.		\
		\item[\eqref{quatro}] If   $g^{cyc}_{\phi_1}(n) \simeq g_{\phi_2}^{cyc}(n) \simeq 2^n$, then $\delta_{M_{\Psi}}$  grows exponentially.
	\end{enumerate}

	For \eqref{dos},  Lemma~\ref{part1} gives $\delta_{M_{\Phi}}(n) \preceq n^{2}$, and we have   $\delta_{M_{\Phi}}(n) \succeq n^{2}$ by Lemma~\ref{Exponential}.  For \eqref{uno}, Lemma~\ref{part2} gives the required lower bound on the Dehn function and {(since $g_{\hat{\xi}_1} \simeq g^{\text{cyc}}_{\phi_1}$)} Lemma~\ref{part1} gives the upper bound.    For \eqref{quatro}, Lemma~\ref{part2} again gives the  lower bound,  and Lemma~\ref{Exponential} gives the upper bound.\end{proof}

\bibliography{MTreferences}{}

\begin{thebibliography}{10}

\bibitem{Alonso}
J.~M. Alonso.
\newblock In\'{e}galit\'{e}s isop\'{e}rim\'{e}triques et quasi-isom\'{e}tries.
\newblock {\em C. R. Acad. Sci. Paris S\'{e}r. I Math.}, 311(12):761--764,
  1990.

\bibitem{BFH1}
M.~Bestvina, M.~Feighn, and M.~Handel.
\newblock The {T}its alternative for {${\rm Out}(F_n)$}. {I}. {D}ynamics of
  exponentially-growing automorphisms.
\newblock {\em Ann. of Math. (2)}, 151(2):517--623, 2000.

\bibitem{Bestvina-Handel}
M.~Bestvina and M.~Handel.
\newblock Train tracks and automorphisms of free groups.
\newblock {\em Ann. of Math. (2)}, 135(1):1--51, 1992.

\bibitem{Bogopolski}
O.~Bogopolski, A.~Martino, and E.~Ventura.
\newblock The automorphism group of a free-by-cyclic group in rank 2.
\newblock {\em Comm. Algebra}, 35(5):1675--1690, 2007.

\bibitem{Bowditch}
B.~H. Bowditch.
\newblock Relatively hyperbolic groups.
\newblock {\em Internat. J. Algebra Comput.}, 22(3):1250016, 66, 2012.

\bibitem{BridsonGersten}
M.~R. Bridson and S.~M. Gersten.
\newblock The optimal isoperimetric inequality for torus bundles over the
  circle.
\newblock {\em Quart. J. Math. Oxford Ser. (2)}, 47(185):1--23, 1996.

\bibitem{BridsonGroves}
M.~R. Bridson and D.~Groves.
\newblock The quadratic isoperimetric inequality for mapping tori of free group
  automorphisms.
\newblock {\em Mem. Amer. Math. Soc.}, 203(955):xii+152, 2010.

\bibitem{BridsonHaefliger}
M.~R. Bridson and A.~Haefliger.
\newblock {\em Metric spaces of non-positive curvature}, volume 319 of {\em
  Grundlehren der Mathematischen Wissenschaften [Fundamental Principles of
  Mathematical Sciences]}.
\newblock Springer-Verlag, Berlin, 1999.

\bibitem{BridsonPittet}
M.~R. Bridson and C.~Pittet.
\newblock Isoperimetric inequalities for the fundamental groups of torus
  bundles over the circle.
\newblock {\em Geom. Dedicata}, 49(2):203--219, 1994.

\bibitem{Brinkmann}
P.~Brinkmann.
\newblock Hyperbolic automorphisms of free groups.
\newblock {\em Geom. Funct. Anal.}, 10(5):1071--1089, 2000.

\bibitem{ButtonKropholler}
J.~O. Button and R.~P. Kropholler.
\newblock Nonhyperbolic free-by-cyclic and one-relator groups.
\newblock {\em New York J. Math.}, 22:755--774, 2016.

\bibitem{Diestel}
R.~Diestel.
\newblock {\em Graph theory}, volume 173 of {\em Graduate Texts in
  Mathematics}.
\newblock Springer-Verlag, Berlin, third edition, 2005.

\bibitem{Farb}
B.~Farb.
\newblock Relatively hyperbolic groups.
\newblock {\em Geom. Funct. Anal.}, 8(5):810--840, 1998.

\bibitem{Gersten}
S.~M. Gersten.
\newblock Isoperimetric and isodiametric functions of finite presentations.
\newblock In {\em Geometric group theory, {V}ol. 1 ({S}ussex, 1991)}, volume
  181 of {\em London Math. Soc. Lecture Note Ser.}, pages 79--96. Cambridge
  Univ. Press, Cambridge, 1993.

\bibitem{GerstenRiley}
S.~M. Gersten and T.~R. Riley.
\newblock Some duality conjectures for finite graphs and their group theoretic
  consequences.
\newblock {\em Proc. Edinb. Math. Soc. (2)}, 48(2):389--421, 2005.

\bibitem{HermillerMeier}
S.~Hermiller and J.~Meier.
\newblock Algorithms and geometry for graph products of groups.
\newblock {\em J. Algebra}, 171(1):230--257, 1995.

\bibitem{Laurence}
M.~R. Laurence.
\newblock A generating set for the automorphism group of a graph group.
\newblock {\em J. London Math. Soc. (2)}, 52(2):318--334, 1995.

\bibitem{Levitt}
G.~Levitt.
\newblock Counting growth types of automorphisms of free groups.
\newblock {\em Geom. Funct. Anal.}, 19(4):1119--1146, 2009.

\bibitem{LyndonSchupp}
R.~C. Lyndon and P.~E. Schupp.
\newblock {\em Combinatorial group theory}.
\newblock Classics in Mathematics. Springer-Verlag, Berlin, 2001.
\newblock Reprint of the 1977 edition.

\bibitem{Osin}
Denis~V. Osin.
\newblock Relatively hyperbolic groups: intrinsic geometry, algebraic
  properties, and algorithmic problems.
\newblock {\em Mem. Amer. Math. Soc.}, 179(843):vi+100, 2006.

\bibitem{Papa}
P.~Papasoglu.
\newblock On the asymptotic cone of groups satisfying a quadratic isoperimetric
  inequality.
\newblock {\em J. Differential Geom.}, 44(4):789--806, 1996.

\bibitem{Piggott}
A.~Piggott.
\newblock Detecting the growth of free group automorphisms by their action on
  the homology of subgroups of finite index.

\bibitem{Servatius}
H.~Servatius.
\newblock Automorphisms of graph groups.
\newblock {\em J. Algebra}, 126(1):34--60, 1989.

\bibitem{Soroko}
I.~Soroko and N.~Brady.
\newblock Dehn functions of subgroups of right-angled {A}rtin groups.
\newblock {\em Geom. Dedicata}, 2019.

\end{thebibliography}
\bibliographystyle{plain}
\end{document}